\documentclass[12pt]{amsart}

\usepackage{fullpage}
\usepackage{amsfonts,amssymb}
\usepackage{enumerate}
\usepackage{verbatim}
\usepackage{graphicx}
\usepackage{tikz-cd}
\usetikzlibrary{matrix}

\newcommand{\C}{\mathbb{C}}
\newcommand{\N}{\mathbb{N}}
\newcommand{\R}{\mathbb{R}}
\newcommand{\Sp}{\mathbb{S}}
\newcommand{\Z}{\mathbb{Z}}

\newcommand{\calF}{{\mathcal {F}}}

\newcommand{\supp}{\operatorname{supp}}

\newcommand{\ov}{\overline}
\newcommand{\ch}{\mathbf 1}

\newcommand{\ve}{\varepsilon}
\newcommand{\vp}{\varphi}

\newcommand{\vr}{\varrho}
\newcommand{\lan}{\langle}
\newcommand{\ran}{\rangle}
\newcommand{\di}{\mathrm d}

\newcommand{\interior}{\operatorname{int}}
\newcommand{\grad}{\operatorname{grad}}
\newcommand{\diam}{\operatorname{diam}}

\def\Koniec{\hbox to\hsize{\hfil$\diamond$}}
\def\th{\vartheta}

\def\1{\mathbf 1}

\newtheorem{theorem}{Theorem}[section]
\newtheorem{corollary}[theorem]{Corollary}
\newtheorem{lemma}[theorem]{Lemma}
\newtheorem{proposition}[theorem]{Proposition}

\theoremstyle{definition}
\newtheorem{definition}{Definition}[section]
\theoremstyle{remark}
\newtheorem{remark}{Remark}[section]

\numberwithin{equation}{section}

\begin{document}

\title{Smooth orthogonal projections on Riemannian manifold}

\author{Marcin Bownik}
\address{Department of Mathematics, University of Oregon, Eugene, OR 97403--1222, USA
and Institute of Mathematics, Polish Academy of Sciences, ul. Wita Stwosza 57,
80--952 Gda\'nsk, Poland}
\email{mbownik@uoregon.edu}

\author{Karol Dziedziul}
\address{
Faculty of Applied Mathematics,
Gda\'nsk University of Technology,
ul. G. Narutowicza 11/12,
80-952 Gda\'nsk, Poland}

\email{kdz@mifgate.pg.gda.pl}

\author{Anna Kamont}
\address{
Institute of Mathematics, Polish Academy of Sciences, ul. Wita Stwosza 57, 80--952 Gda\'nsk, Poland}

\email{anna.kamont@impan.pl}

\keywords{Riemannian manifold, Hestenes operator, smooth orthogonal projection, latitudinal projection,  smooth decomposition of identity, Morse function, Sobolev space}

\thanks{The first author was partially supported by the NSF grant DMS-1665056 and by a grant from the Simons Foundation \#426295.}

\subjclass[2000]{42C40, 46E30, 46E35, 58C35}

\date{\today}

\begin{abstract}
We construct a decomposition of the identity operator on a Riemannian manifold $M$ as a sum of smooth orthogonal projections subordinate to an open cover of $M$. This extends a decomposition of the real line by smooth orthogonal projection due to Coifman, Meyer \cite{CM} and Auscher, Weiss, Wickerhauser \cite{AWW}, and a similar decomposition when $M$ is the sphere by the first two authors  \cite{BD}.
\end{abstract}

\maketitle


\section{Introduction and main result}

The goal of the paper is to construct a decomposition of the identity operator on the Riemannian manifold $M$ as a sum of smooth orthogonal projections with desired localization properties. This can be thought as an operator analogue of the ubiquitous smooth partition of unity subordinate to an open cover of a manifold. However, smooth partitions of unity do not give rise in any obvious way to orthogonal projections and much more complicated constructions are needed to achieve this goal.

Smooth projections on the real line have appeared implicitly in the construction of local sine and cosine bases by Coifman and Meyer \cite{CM}. They were systematically studied by Auscher, Weiss, and Wickerhauser \cite{AWW} in the construction of smooth wavelet bases in $L^2(\R)$. For a detailed exposition on smooth projections on the real line we refer to the book by Hern\'andez and Weiss \cite{HW}. A standard tensoring procedure can be used to extend smooth projections to the Euclidean space $\R^d$. However, an extension of smooth projections to the sphere $\Sp^d$ is already far less trivial. This was shown recently by the first two authors in \cite{BD}. 
In this paper we show a general construction of smooth orthogonal projections on Riemannian manifolds, which is based in part on Morse theory. In order to formulate our main result we need to define the class of Hestenes operators.

Let $(M,g)$ be a smooth connected Riemannian manifold with a Riemannian metric $g$ on $M$. We consider for simplicity that the manifold is without boundary. The metric $g$ induces a Riemannian measure $\nu=\nu_M$ on $M$.
For any $1\le p\le \infty$, let $L^p(M)$ be the real Lebesgue space on the measure space $(M,\nu)$. In the special case $p=2$, $L^2(M)$ is a Hilbert space with the inner product
\[
\lan f,h \ran=\int_{M} f(x) h(x)  d\nu(x).
\]
We will employ a class of Hestenes \cite{Hes} operators, which was originally introduced in the work of Ciesielski and Figiel \cite[Section 5]{CF2}. However, we shall use a simplified variant of the class of $H$-operators introduced in \cite{BD}. 
 
\begin{definition}\label{H}
Let $M$ be a smooth connected Riemannian manifold (without boundary). Let $\Phi:V \to V'$ be a $C^\infty$ diffeomorphism between two open subsets $V, V' \subset M$. Let $\vp: M \to \R$ be a 
compactly supported 
$C^\infty$ function such that 
\[
\supp \vp =\ov{\{x\in M : \vp(x) \ne 0\}} \subset V.
\]
We define a simple $H$-operator $H_{\vp,\Phi,V}$ acting on a  function $f:M\to\R$ by
\begin{equation}\label{H2}
H_{\vp,\Phi,V}f(x) = \begin{cases} \vp(x) f(\Phi(x)) & x\in V
\\
0 & x\in M \setminus V.
\end{cases}
\end{equation}
Let $C_0(M)$ be the space of continuous functions on $M$ that are vanishing at infinity, which is equipped with the supremum norm. Clearly, a simple $H$-operator induces a continuous linear map of the space $C_0(M)$ into itself. We define an $H$-operator to be an operator $T:C_0(M) \to C_0(M)$ which is a finite combination of such simple $H$-operators. The space of all $H$-operators is denoted by $\mathcal H(M)$.
\end{definition}

It can be checked  that $H$-operators map the space of smooth functions $C^\infty(M)$ into itself. Our main result is a generalization of the result of the first two authors \cite{BD} from the setting of the sphere $\Sp^d$ to a  Riemannian manifold $M$.

\begin{theorem}\label{Main}
Let $M$ be a smooth connected Riemannian manifold (without boundary).
Suppose $\mathcal U$ is an open and precompact cover of $M$. Then, there exists a family of operators $\{P_U\}_{U\in \mathcal U}$ defined on $C_0(M)$ such that:
\begin{enumerate}[(i)]
\item family $\{P_U\}_{U\in \mathcal U}$ is locally finite, i.e., for any compact $K\subset M$, all but finitely many operators $P_{U}$ such that $U \cap K \ne \emptyset$, are zero, 
\item each $P_U \in \mathcal H(M)$ is localized on an open set $U \in\mathcal U$; in particular, for any $f\in C_0(M)$,
\begin{align}
&\supp P_U f \subset U, \label{locz1}
\\
&\supp f \cap U=\emptyset \implies P_Uf=0, \label{locz2}
\end{align}
\item each $P_U$ has a unique extension to an operator $P_U:L^2(M) \to L^2(M)$ that is an orthogonal projection,
\item the projections $\{P_U\}_{U\in \mathcal U}$ are mutually orthogonal and they form a decomposition of the identity operator $\mathbf I$ on $L^2(M)$,
\begin{equation}\label{sum}
P_U \circ P_{U'}=0 \quad\text{for }U\ne U' 
\qquad\text{and}\qquad\sum_{U \in \mathcal U} P_U = \mathbf I.
\end{equation}
\end{enumerate}
\end{theorem}

A family of operators $\{P_U\}_{U\in \mathcal U}$  satisfying conclusions of Theorem \ref{Main} is said to be a {\it smooth orthogonal decomposition of identity in $L^2(M)$} subordinate to an open cover $\mathcal U$ of a manifold $M$.
 This concept should be contrasted with the ubiquitous smooth partition of unity. This is a family of smooth nonnegative  functions $\{\phi_U\}_{U\in\mathcal U}$ such that $\supp \phi_U\subset U$, $\sum_{U\in \mathcal U}\phi_U=1$, and for every $x\in M$ there is a neighborhood $O_x$ such that $\supp \phi_U \cap O_x= \emptyset$ for all but finitely many $U$. While the corresponding family of multiplication operators $S_U(f)=\phi_U f$ is localized and satisfies $\sum_{U\in\mathcal U} S_U= \mathbf I$, $S_U$ can not be a projection unless $\phi_U$ is discontinuous.
By constructing operators from a larger class of Hestenes operators $\mathcal H(M)$, we will preserve all of these properties and additionally guarantee that operators $P_U$ are orthogonal projections. 

The proof of Theorem \ref{Main} is quite long and involved as it occupies most of the paper. In Section \ref{S2} we show several properties of $H$-operators including the key concept of localization. Localized $H$-operators are shown to be bounded on the space of smooth functions $C^r(M)$ and the Sobolev Space $W^r_p(M)$, $1\le p<\infty$, $r\in\N$. We also introduce the concept of smooth decomposition of identity in $L^p(M)$, which is a natural generalization of orthogonal decomposition of identity in $L^2(M)$ as in Theorem \ref{Main}. We show that this concept behaves well under diffeomorphisms and change of weights. In Section \ref{S3} we show several results based on Morse theory, which are measure-theoretic analogues of well-known topological results, such as the regular interval lemma. In particular, we establish the critical point lemma which, outside of a small neighborhood of a critical point,  provides a convenient parametrization of the Riemann measure as a product measure of an interval and a level submanifold. In Section \ref{S4} we construct latitudinal projections which decompose a manifold $M$ along level sets of a Morse function. These are manifold analogues of smooth projections on the real line  \cite{AWW} and on the sphere \cite{BD}. 

The most technical results are contained in Section \ref{S5}, which develops the method of lifting an $H$-operator acting on a level submanifold to the whole manifold $M$. The resulting global lifting operator commutes with latitudinal projections and their composition is again an $H$-operator. As a result we show that a smooth decomposition of identity on a level submanifold can be lifted to a smooth decomposition of a latitudinal projection, which is localized on a strip between level sets of a Morse function. Though rather long and tedious, this procedure is straightforward for regular intervals. However, intervals containing critical values are very problematic since there is no direct method of lifting projections which are localized near a critical point. Fortunately, this problem affects only one projection which can lifted in a roundabout way using all other projections localized outside of this critical point. The least trivial aspect of this procedure involves showing the required localization property. 

In Section \ref{S6} we put together our results to prove the existence of a smooth decomposition of identity in $L^p(M)$, which is subordinate to an open and precompact cover of $M$. The main result of the paper is Theorem \ref{Main2}, which extends Theorem \ref{Main} to $L^p(M)$ spaces. This requires an inductive procedure which produces a smooth decomposition of a manifold $M$ of dimension $d$ based on smooth decompositions of carefully chosen level submanifolds of dimension $d-1$. In addition, we show that the overlaps of supports of the resulting projections are uniformly bounded by a constant independent of a cover of $M$ and depending only on a dimension $d$. Finally, in Section \ref{S7} we give applications of the main theorem in the study of Sobolev spaces on manifolds.

\section{Properties of Hestenes operators}\label{S2}

In this section we establish several properties of $H$-operators. Some of them were already shown in \cite{BD}. 
 For example, $\mathcal H(M)$ is an algebra of operators that is closed under tensoring operation, see \cite[Lemma 3.1]{BD}. In particular, we have the following formula for a composition of two simple $H$-operators $H_{\varphi_1,\Phi_1,V_1} \circ H_{\varphi_2,\Phi_2,V_2} = H_{\varphi,\Phi,V}$, where
\begin{equation*}\label{comp1}
\varphi(x)=\begin{cases} \varphi_1(x)\varphi_2(\Phi_1(x)) & x\in V,\\
0 & \text{otherwise,}
\end{cases}
\qquad
\Phi=\Phi_2 \circ \Phi_1|_V
\qquad
V=\Phi_1^{-1}(V_2).
\end{equation*}

\subsection{Localization of $H$-operators} We start by defining the concept of a localized $H$-operator.

\begin{definition}\label{localized}
 We say that an operator $T \in \mathcal H(M)$ is {\it localized} on an open set $U \subset M$, if it has a representation as a finite combination of simple $H$-operators $H_{\vp,\Phi,V}$ satisfying $V\subset U$ and $\Phi(V) \subset U$.
\end{definition}

The following lemma provides an intrinsic characterization of localized operators as it does not refer to a decomposition into simple $H$-operators. As an immediate consequence, we deduce that if $P_U \in \mathcal H(M)$ is localized on open and precompact $U \subset M$, then \eqref{locz1} and \eqref{locz2} hold.

\begin{lemma}\label{ilo}
Let $T \in\mathcal H(M)$ and let $U\subset M$ be open and precompact. Then, $T$ is localized on $U$ if and only if there exists a compact set $K \subset U$ such that for any $f\in C_0(M)$
\begin{align}
&\supp T f \subset K, \label{ilo1}
\\
&\supp f \cap K=\emptyset \implies Tf=0. \label{ilo2}
\end{align}
\end{lemma}

\begin{proof}
Let $T=T_1+\ldots + T_m$, where each $T_i$ is a simple $H$-operator of the form $T_i=H_{\vp_i,\Phi_i,V_i}$, $i=1,\ldots,m$, where $\supp \vp_i \subset V_i$ and $\Phi_i:V_i \to V_i'$ is a diffeomorphism. Suppose first that $T$ is localized on $U$. By Definition \ref{localized}, for each $i$ we have $V_i \subset U$ and $\Phi_i(V_i) \subset U$.
Define
\[ K_0:= \bigcup_{i=1}^m \supp \vp_i \cup \bigcup_{i=1}^m \Phi_i(\supp \vp_i) \subset U.
\]
Then, a simple calculation shows that \eqref{ilo1} and \eqref{ilo2} hold for $K=K_0$.

Conversely, assume that \eqref{ilo1} and \eqref{ilo2} hold for some $K$. Pick an open set $U'$ such that $K\subset U' \subset \overline{U'} \subset U$. Let $\vp: M \to [0,1]$ be a $C^\infty$ function such that
\begin{equation}\label{ilo4}
\supp \vp \subset U \qquad\text{and}\qquad \vp(x)=1 \text{ for all }x\in \overline{U'}.
\end{equation}
We claim that
\begin{equation}\label{ilo5}
T(f) = \vp T(\vp f) \qquad\text{for all }f\in C_0(M).
\end{equation}
Indeed, since $\supp(1-\vp) \cap K =\emptyset$, by \eqref{ilo2} we have $T((1-\vp)f)=0$. Hence, $T(f)=T(\vp f)$. On the other hand, by \eqref{ilo1} and \eqref{ilo4} we have $T(\vp f)=\vp T(\vp f)$. Thus, \eqref{ilo5} is shown. 

Using \eqref{ilo4} yields
\[
\begin{aligned}
\vp(x) T_i(\vp f)(x) 
& = 
\begin{cases} \vp_i(x) \vp(x)  \vp(\Phi_i(x)) f(\Phi_i(x)) & x\in V_i,
\\
0 & \text{otherwise.}
\end{cases}
\\
& = 
\begin{cases} \vp_i(x) \vp(x) \vp(\Phi_i(x)) f(\Phi_i(x)) & x\in V_i \cap U \cap \Phi_i^{-1}(U),
\\
0 & \text{otherwise.}
\end{cases}
\end{aligned}
\]
Hence, $f \mapsto \vp T_i(\vp f)$ is a simple $H$-operator localized on $U$. Combining this with
\[
Tf= \vp T(\vp f) = \sum_{i=1}^m \vp T_i(\vp f),
\]
shows that $T$ is also localized on $U$.
\end{proof}

In case of a simple $H$-operator $H=H_{\vp,\Phi,V}$  localized on some open and precompact set there is a minimal compact set $K(H)$ with respect to the  inclusion relation satisfying \eqref{ilo1} and \eqref{ilo2}. Namely, 
\[
K(H)=K_1(H)\cup K_2(H),
\]
where $K_1(H)=\supp \vp$ and $K_2(H)=\Phi(\supp \vp)$.
This remark  and  Lemma \ref{ilo} motivate the following definition. 

\begin{definition}
Let $T \in\mathcal H(M)$  be localized on  some open and precompact set. Let
\[
K_1(T)=\bigcap \left\{K_1:   \text{$K_1$ is a compact set satisfying  \eqref{ilo1} for all $f\in C_0(M)$ } \right\} ,
\]
\[
K_2(T)=\bigcap \left\{ K_2:  \text{$K_2$ is a compact set satisfying \eqref{ilo2} for all $f\in C_0(M)$ }
\right\}.
\]
We define  a localizing set for $T$ as
\[
K(T)=K_1(T)\cup K_2(T).
\]
\end{definition}

It is clear that a set $K_1(T)$ satisfies \eqref{ilo1}. In the course of proving  Lemma \ref{K(T)} we show a set $K_2(T)$ satisfies \eqref{ilo2}. For this we need the following technical lemma.

\begin{lemma}\label{tech} Suppose that  $T \in \mathcal H(M)$ is localized on  an  open and precompact set. Let $K_2$ be compact. Then $K_2$ satisfies  \eqref{ilo2} for all $f\in C_0(M)$ if and only if for all $x\notin K_2$, there exists $U_x$ open and precompact such that $x\in U_x$, $U_x\cap K_2=\emptyset$, and
\begin{equation}\label{warunek}
  \supp f\subset U_x \implies Tf=0\qquad \text{for all } f\in C_0(M).
\end{equation}
\end{lemma}

\begin{proof}
The implication \eqref{ilo2} $\implies$ \eqref{warunek} is clear. To check the converse, let $U$ be some open and precompact set on which $T$ is localized. It follows by Definition \ref{localized} that if $f,g\in C_0(M)$ and $f=g$ on $U$ then $Tf=Tg$ on $M$. 
Now let $f\in C_0(M)$ be such that $\supp f\cap K_2=\emptyset$.
For each $x\notin K_2$ let $U_x$ be as in \eqref{warunek}. Moreover, let $V$ be an open and precompact set such that $K_2\subset V$ and $\overline{V} \cap \supp f=\emptyset$. Let $\{\alpha_x:x\in M\setminus K_2\}\cup \{\alpha_V\}$ be a partition of unity subordinate to the open cover
$\{U_x:x\in M\setminus K_2 \}\cup \{V\}$. By compactness there are $x_1,\ldots,x_n$ such that $\alpha_V+\sum_{i=1}^n \alpha_{x_i}=1$ on $\overline{U}$. Since $\supp \alpha_V \cap \supp f=\emptyset$, we have
\[
f=f \bigg( \alpha_V+\sum_{i=1}^n \alpha_{x_i}\bigg) =\sum_{i=1}^n f \alpha_{x_i}\qquad \text{on $U$.}
\]
 Since $\supp (f\alpha_{x_i}) \subset U_{x_i}$, we have  $T(f\alpha_{x_i})=0$ by \eqref{warunek}.
Therefore 
\[
Tf=\sum_{i=1}^n T(f\alpha_{x_i})=0.
\]
This proves Lemma \ref{tech}.
\end{proof}

\begin{lemma}\label{K(T)}
Suppose that  $T \in \mathcal H(M)$ is localized on  an  open and precompact set. 
Let $U\subset M$ be open and precompact. 
Then $T$ is localized on $U$ if and only if $K(T)\subset U$.
\end{lemma}

\begin{proof}
It follows by Lemma \ref{ilo} that if $T$ is localized on $U$ then $K(T) \subset U$. To prove the converse we need to show that $K(T)$ satisfies \eqref{ilo1} and \eqref{ilo2}. Clearly \eqref{ilo1}  is satisfied by
$K_1(T)$ and hence by $K(T)$. It remains to show that $K_2(T)$ satisfies \eqref{ilo2}. 

Take any $x\notin K_2(T)$. Then there exists $K_2$ satisfying \eqref{ilo2} and $x\notin K_2$. By Lemma \ref{tech} we find $U_x$ such that $U_x\cap K_2=\emptyset$ and \eqref{warunek} is satisfied. 
Note that  $U_x\cap K_2=\emptyset$ implies 
$U_x\cap K_2(T)=\emptyset$. Hence, by Lemma \ref{tech} $K_2(T)$  satisfies \eqref{ilo2}.
\end{proof}

An immediate consequence of Lemma \ref{K(T)} is the following.

\begin{lemma}\label{il}

Suppose that $T \in \mathcal H(M)$ is localized at the same time on two open and precompact sets $U \subset M$ and $\tilde U \subset M$. Then, $T$ is localized on $U \cap \tilde U$.
\end{lemma}

In the sequel we need the following lemma.

\begin{lemma}\label{lco}
Let $T_1$, $T_2$ be two commuting $H$-operators 
localized on open and precompact sets $U_1,U_2$, respectively. Then, their composition $T_1 \circ T_2 =T_2 \circ T_1$ is localized on $U_1\cap U_2$. \end{lemma}

\begin{proof}
By Lemma \ref{ilo}, there exists compact subsets $K_i \subset U_i$ such that for all $f\in C_0(M)$ and $i=1,2$,
\begin{align}
&\supp T_i f \subset K_i, \label{lco1}
\\
&\supp f \cap K_i=\emptyset \implies T_if=0. \label{lco2}
\end{align}
Let $T=T_1 \circ T_2=T_2 \circ T_1$. By \eqref{lco1} we have
\[
\supp Tf = \supp T_1(T_2f) \subset K_1.
\]
Moreover, if $f\in C_0(M)$ is such that $\supp f \cap K_1=\emptyset$, then by \eqref{lco2}
\[
Tf=T_2(T_1f)=0.
\]
This implies that $T$ is localized on $U_1$ by Lemma \ref{ilo}. The same argument shows that $T$ is also localized on $U_2$. Hence, $T$ is localized on $U_1 \cap U_2$ by Lemma \ref{il}.
\end{proof}

\subsection{Boundedness of $H$-operators} We start by reminding the definition of  Sobolev spaces on Riemannian manifolds \cite{Au, He}.

\begin{definition}\label{sob}
Let $(M,g)$ be a smooth Riemannian manifold. There exists a unique torsion-free connection $\nabla$ on $M$ having the property that $\nabla g=0$, known as the Levi-Civita connection. 
Let $f:M\to \R$ be a smooth function. Let
 $\nabla^k f(x)$  be a covariant derivative of $f$ of order $k\in \N$ at point $x\in M$ in a local chart $(U,\psi)$; the coordinates of $\nabla^k f(x)$  are denoted by
$\nabla^k f(x)_{i_1\ldots i_k}$ (see \eqref{kowariantna}). Then
 the norm $|\nabla^k f|$  is independent of a choice of chart  $(U,\psi)$ and is given by
 \begin{equation}\label{norma}
  |\nabla^k f|^2(x)=\sum_{i_1,\ldots,i_k=1}^d \sum_{j_1,\ldots,j_k=1}^d
 g^{i_1 j_1}(x)\cdots g^{i_k j_k}(x) \nabla^k f(x)_{i_1\ldots i_k} \nabla^k f(x)_{j_1\ldots j_k},
  \end{equation}
where we write $g$ in the local chart as $g(x)=(g_{ij}(x))_{1\leq i,j\leq d}$ and  $g^{ij}(x)$ are such that
 \[
 \sum_{m=1}^d g_{im}(x)g^{jm}(x)=\delta_{ij}.
 \]
 Recall that $\nabla^k f(x)$ is $(k,0)$-tensor on $T_x(M)$.
 The Banach space $C^r(M)$ consists of all $C^r$ functions $f:M \to \R$ with the norm
\[
||f||_{C^{r}(M)} = \sum_{k=0}^r \sup_{x\in M} |\nabla^k f(x)| <\infty.
\] 
Let $\nu$ be the Riemannian measure on $M$. Given $1\le p< \infty$ we define the norm 
\begin{equation}\label{wrp}
||f||_{W^r_p} = \sum_{k=0}^r \bigg( \int_M |\nabla^k f(x)|^p d\nu(x) \bigg)^{1/p}<\infty.
\end{equation}
Let
\[
C^p_r(M)=\left\{f\in C^\infty(M) : ||f||_{W^r_p} <\infty \right\}.
\]
The Sobolev space $W^r_p(M)$ is the completion of $C^p_k(M)$ 
 with respect to the norm $||\cdot||_{W^r_p}$, see \cite{He}.
\end{definition}

The following analogue of \cite[Lemma 5.38 and Corollary 5.39]{CF2}, see \cite[Lemma 3.2]{BD}, plays a crucial role in our considerations.
 
\begin{theorem}\label{crm}
Suppose that $H \in \mathcal H(M)$ is localized on open and precompact set  $U \subset M$. Then, for any $r=0,1,\ldots$, the operator $H$ induces a bounded linear operator
\begin{align}\label{crm1}
&H: C^r(M) \to C^r(M), \qquad\text{where }r=0,1,\ldots,
\\
\label{crm2}
&H:W^r_p(M)\to W^r_p(M), \qquad\text{where }1\le p < \infty,\  r=0,1,\ldots.
\end{align}
\end{theorem}

For completeness  we present all arguments used in the proof of Theorem \ref{crm}. The technical
 lemma below is usually given without a proof, see \cite[Theorem 2.20]{Au}, \cite[Proposition 2.2]{He}, \cite[Theorem 7.4.5]{Tr}.

\begin{lemma}\label{technical} Let  $(U,\psi)$ be a chart of $M$. Let $K\subset U$ be compact and $r\in \N$. Then, there is a constant $C>0$ such that:
\begin{enumerate}[(i)]
\item for all $f\in C^r(M)$ and all $x\in K$
\begin{equation}\label{rownowaznosc1}
1/C \sum_{k=0}^r |\nabla^k f(x)|\leq  \sum_{|\beta|\leq r} |D^\beta (f\circ \psi^{-1})(\psi(x))| \leq C  \sum_{k=0}^r |\nabla^k f(x)|,
\end{equation}
\item
for all $f\in W^r_p(M)$, $1\leq p<\infty$, with $\supp f\subset K$,
\begin{equation}\label{rownowaznosc2}
1/C\|f\circ \psi^{-1} \|_{W^r_p(\R^d)}\leq \|f\|_{W^r_p(M)}\leq C\|f\circ \psi^{-1} \|_{W^r_p(\R^d)}.
\end{equation}
\end{enumerate}
\end{lemma}

\begin{proof} A remark is needed to explain the precise meaning of $f\circ \psi^{-1}$ in \eqref{rownowaznosc2}. We extend the domain of this function to the entire $\R^d$ by
\[
f\circ \psi^{-1}(y)=\begin{cases} f\circ \psi^{-1}(y) & y\in\psi(U),\\
0 & \text{otherwise.}
\end{cases}
\]
We start with the proof of $(i)$. Take $f\in C^r(M)$.
In the local chart  $(U,\psi)$ the covariant derivative of order $k$ at $x\in U$ is defined recursively by
\begin{equation}\label{kowariantna}
\nabla^k f(x)_{i_1\ldots i_k}=(\nabla_{i_1} \nabla^{k-1} f)(x)_{i_2\ldots i_k}
\end{equation}
\[
=\frac{\partial}{\partial x_{i_1}}(\nabla^{k-1} f)(x)_{i_2\ldots i_k}-
\sum_{\alpha=1}^d \sum_{l=2}^k \Gamma^\alpha_{i_1i_l}(x)\nabla^{k-1} f(x)_{i_2\ldots i_{l-1}\alpha i_{l+1}\ldots i_k},
\]
where $ \Gamma^\alpha_{i_1i_l}$ are Christoffel symbols given by
\[
\nabla_i \left( \frac{\partial}{\partial x_j}\right)(x)=\sum_{\alpha=1}^d
\Gamma^\alpha_{ij}(x) \left( \frac{\partial}{\partial x_\alpha}\right)_x,
\]
see  \cite[page 6]{He}.
Now by an induction argument there are smooth functions $\Lambda^\beta_{i_1\ldots i_k}$ defined on $U$ such that for all $x\in U$
\begin{equation}\label{pochodna1}
\nabla^k f(x)_{i_1\ldots i_k}=\frac{\partial^k}{\partial x_{i_1}\ldots \partial x_{i_k} }f(x)-\sum_{|\beta|\leq k-1} \Lambda^\beta_{i_1\ldots i_k}(x)\frac{\partial^{|\beta|}}{\partial x^\beta} f(x).
\end{equation}
Recall that 
for a multi-index  $\beta=(\beta_1,\ldots,\beta_d)\in \N_0^d$ we have 
\begin{equation}\label{pochodna2}
\frac{\partial^{|\beta|}}{\partial x^\beta} f(x)
=D^\beta (f\circ \psi^{-1})(\psi(x)).
\end{equation}
Note that functions $\Lambda^\beta_{i_1\ldots i_k}$ in \eqref{pochodna1} are some products of derivatives of Christoffel symbols. Since $K$ is compact, there exists a positive constant $C$ such that 
\[
1/C \delta \leq g\leq C\delta \qquad\text{on }K,
\]
where $\delta=(\delta_{i,j})_{1\leq ij\leq d}$ is identity $(2,0)$-tensor and the above inequalities are understood in the sense of bilinear symmetric forms. Consequently, 
\[
1/C \delta \leq g^{-1}\leq C\delta \qquad\text{on }K.
\]
Combining the above inequalities with \eqref{norma} we get
\begin{equation}\label{geodezyjna1}
 1/C^{k} \sum_{i_1,\ldots,i_k=1}^d |\nabla^k f(x)_{i_1\ldots i_k}|^2 \leq |\nabla^k f(x)|^2 \leq C^{k} \sum_{i_1,\ldots,i_k=1}^d |\nabla^k f(x)_{i_1\ldots i_k}|^2.
\end{equation}
By \eqref{pochodna1}, \eqref{geodezyjna1}, and the equivalence of $\ell^2$ and $\ell^1$ norms in finitely dimensional spaces
\begin{equation}\label{geodezyjna2}
 |\nabla^k f(x)| \asymp  \sum_{i_1,\ldots,i_k=1}^d \left|\frac{\partial^k}{\partial x_{i_1}\ldots \partial x_{i_k} }f(x)-\sum_{|\beta|\leq k-1} \Lambda^\beta_{i_1\ldots i_k}(x)\frac{\partial^{|\beta|}}{\partial x^\beta} f(x)\right|.
\end{equation}
Using \eqref{pochodna2} we note that \eqref{rownowaznosc1} is equivalent with the following claim for all $x\in K$
\begin{equation}\label{claim}
\sum_{k=0}^r \sum_{i_1,\ldots,i_k=1}^d \left|\frac{\partial^k}{\partial x_{i_1}\ldots \partial x_{i_k} }f(x)-\sum_{|\beta|\leq k-1} \Lambda^\beta_{i_1\ldots i_k}(x)\frac{\partial^{|\beta|}}{\partial x^\beta} f(x)\right|\asymp \sum_{|\beta|\leq r} \left|\frac{\partial^{|\beta|}}{\partial x^\beta} f(x)\right|.
\end{equation}
To prove \eqref{claim} observe that for $x\in K$ 
\[
\left|\frac{\partial^k}{\partial x_{i_1}\ldots \partial x_{i_k} }f(x)-\sum_{|\beta|\leq k-1} \Lambda^\beta_{i_1\ldots i_k}(x)\frac{\partial^{|\beta|}}{\partial x^\beta} f(x)\right|\leq \left|\frac{\partial^k}{\partial x_{i_1}\ldots \partial x_{i_k} }f(x)\right|+
C \sum_{|\beta|\leq k-1} \left|\frac{\partial^{|\beta|}}{\partial x^\beta} f(x)\right|.
\]
Hence, the left side of \eqref{claim} is dominated by the right side of \eqref{claim}. We will show the converse inequality by an induction argument. The base case $r=0$ is obvious. Next, we observe
the following inequality
\begin{equation}\label{indukcja}
 \left|\frac{\partial^k}{\partial x_{i_1}\ldots \partial x_{i_k} }f(x)\right| \leq \left|\frac{\partial^k}{\partial x_{i_1}\ldots \partial x_{i_k} }f(x)-\sum_{|\beta|\leq k-1} \Lambda^\beta_{i_1\ldots i_k}(x)\frac{\partial^{|\beta|}}{\partial x^\beta} f(x)\right|+
C  \sum_{|\beta|\leq k-1} \left|\frac{\partial^{|\beta|}}{\partial x^\beta} f(x)\right|.
\end{equation}
 By the induction hypothesis assume that \eqref{claim} is true for $r-1$,
\begin{equation}\label{indukcja1}
 \sum_{|\beta|\leq r-1} \left|\frac{\partial^{|\beta|}}{\partial x^\beta} f(x)\right| \leq \sum_{k=0}^{r-1} \sum_{i_1,\ldots,i_k=1}^d \left|\frac{\partial^k}{\partial x_{i_1}\ldots \partial x_{i_k} }f(x)-\sum_{|\beta|\leq k-1} \Lambda^\beta_{i_1\ldots i_k}(x)\frac{\partial^{|\beta|}}{\partial x^\beta} f(x)\right|.
\end{equation}
Then, the inequalities \eqref{indukcja} and \eqref{indukcja1} yield the remaining part of \eqref{claim}.

It remains to show $(ii)$. Take $f\in C^r(M)$ with $\supp f\subset K$. By  \eqref{rownowaznosc1} and the equivalence of $\ell^1$ and $\ell^p$ norms in finitely dimensional spaces we have
\begin{equation}\label{rowz}
\bigg(\sum_{k=0}^r |\nabla^k f(\psi^{-1}(y))|^p \bigg)^{1/p} \asymp  \bigg(\sum_{|\beta|\leq r} |D^\beta (f\circ \psi^{-1})(y)|^p \bigg)^{1/p} \qquad \text{for }y\in \psi(U).
\end{equation}
By the definition of the Riemannian measure $\nu$ and by \eqref{wrp} we have
\[
\|f\|^p_{W^r_p(M)}=\sum_{k=0}^r \int_{\psi(U)} \left(|\nabla^k  f|^p\sqrt{|g|} \right) \circ \psi^{-1}(y)\, dy.
\]
By the compactness of $K$, there are constant $C_1,C_2>0$ such that 
\[
C_1\leq \sqrt{|g|(x)}\leq C_2 \qquad\text{for all }x\in K.
\]
Thus, integrating \eqref{rowz} over $\psi(U)$ implies \eqref{rownowaznosc2} for $f\in C^r(M)$. By a density argument we obtain \eqref{rownowaznosc2} for all $f\in W^r_p(M)$ with $\supp f \subset K$.
\end{proof}

In the proof of Theorem \ref{crm} we also need following lemma.

\begin{lemma}\label{Hsimple}
Let $H_{\vp,\Phi,V}$ be a simple $H$-operator localized on a precompact set $U$. Then there is a finite collection  of simple $H$-operators $\{H_{\vp_s,\Phi_s,V_s}:s\in S\}$ localized on the set $U$ such that for every $s\in S$:
\begin{enumerate}[(i)]
\item $\vp_s \in C^\infty(M)$ satisfies $\supp \vp_s \subset V_s$ and
$\Phi_s: V_s \to \tilde{V_s}$,
\item open sets $V_s$ and $\Phi_s(V_s)=\tilde{V_s}$ are contained in domains of some charts on $M$, and
\item
$
H_{\vp,\Phi,V}=\sum_{s\in S} H_{\vp_s,\Phi_s,V_s}.
$
\end{enumerate}
\end{lemma}

\begin{proof}
Since $V$ and $V'$ 
are contained in $U$, there exists a finite collection of charts $(\Omega_j,\psi_j)$, $j\in F$ such that 
\[
V\cup V' \subset \bigcup_{j\in F} \Omega_j.
\]
Define for $i,j\in F$
\[
\Omega_{i,j}=\Omega_j\cap \Phi^{-1}(\Omega_i\cap V').
\]
Let $K=\supp \vp\subset V$. Consider a collection of sets $\Omega_{i,j}$, $i,j \in F$, which together with $M\setminus K$ form an open cover of $M$. Let $\alpha_{i,j}$, $i,j \in F$ be functions from a partition of unity subordinate to this open cover. From the above definition we have
\[
H_{\vp,\Phi,V}(f)(x)=\sum_{i,j\in F} \vp(x)\alpha_{i,j}(x) f(\Phi(x)) \qquad \text{for all }x\in V.
\]
Consequently
\[
H_{\vp,\Phi,V}=\sum_{i,j\in F} H_{\vp_{i,j},\Phi_{i,j},\Omega_{i,j}},
\]
where $\vp_{i,j}=\vp\alpha_{i,j}$ and  $\Phi_{i,j}$ is a restriction of the diffeomorphism $\Phi$ to $\Omega_{i,j}$. Observe that $\Omega_{i,j}\subset \Omega_j$  and $\Phi_{i,j}(\Omega_{i,j})\subset \Omega_i$. Taking
$V_{i,j}=\Omega_{i,j}$, where $(i,j)\in S:=F\times F$ concludes the proof.
\end{proof}

\begin{proof}[Proof of Theorem \ref{crm}] 
Without loss of generality we may assume that $H=H_{\vp,\Phi,V}$ is an $H$-operator satisfying conditions of Lemma \ref{Hsimple}. That is,
$H$ is localized on a precompact set $U$ such that
\[
\Phi: V \to V',
\]
where charts $(V,\psi)$ and $(V',\psi_1)$ belong to the atlas of $M$. Introduce 
\[
\tilde{\Phi}=\psi_1 \circ \Phi \circ \psi^{-1}:W \to W_1,
\]
where $W=\psi(V)$ and $W_1=\psi_1(V')$ are open and precompact subset of $\R^d$.
 Let
\[
\tilde{\vp}(y)=\begin{cases} \vp\circ \psi^{-1}(y) & y\in W
\\
0 & y\in \R^d\setminus W.
\end{cases}
\]
Observe that $\tilde{\vp}$ is a compactly supported $C^\infty(\R^d)$ function and $\tilde{\Phi}$ is a diffeomorphism.
We define $\tilde{H}=\tilde{H}_{\tilde{\vp},\tilde{\Phi},W}$ a simple $H$-operator, i.e.,
\[
\tilde{H}(g)(y)=
\begin{cases}
\tilde{\vp}(y) g(\tilde{\Phi}(y)) & y\in W \\
0 & y\in \R^d\setminus W.
\end{cases}
\]
Note that if $f: V'\to \R$ then for 
\[
g:=f\circ \psi_1^{-1}: W_1 \to \R,
\]
$H(f)(x)=\tilde{H}(g)(\psi(x))$. Indeed, for $x\in V$,
\[
\tilde{H}(g)(\psi(x))=\tilde{\vp}(\psi(x))g(\tilde{\Phi}(\psi(x))=\vp(x)
f\circ\psi_1^{-1}\circ\psi_1\circ\Phi\circ\psi^{-1}(\psi(x))=Hf(x).
\]
Hence,
\begin{equation}\label{Hlifting}
H(f)\circ\psi^{-1}=\tilde{H}(g).
\end{equation}

By the chain rule, a change of variables, and compactness of $\supp \tilde{\vp}$, we obtain that $\tilde{H}=\tilde{H}_{\tilde{\vp},\tilde{\Phi},W}$ is a bounded linear operator 
\[
\tilde{H}_{\tilde{\vp},\tilde{\Phi},W}: \calF(\R^d)\to \calF(\R^d),
\]
where $\calF(\R^d)$ is $C^r(\R^d)$ or $W^r_p(\R^d)$.
Now let $f\in C^r(M) $ be such that $\supp f \subset K$, where $K \subset V'$ is compact. Then by following arguments: Lemma \ref{technical}, \eqref{Hlifting}, boundedness of $\tilde{H}$, and once more Lemma \ref{technical} we have
\begin{equation}\label{lk}
\|H(f)\|_{\calF(M)} \asymp \|(Hf)\circ \psi^{-1} \|_{\calF(\R^d)} =
\|\tilde{H} g \|_{\calF(\R^d)} \leq C \|g\|_{\calF(\R^d)}\asymp \|f\|_{\calF(M)}.
\end{equation}
To complete the proof of Theorem \ref{crm} take $\eta\in C^\infty(M)$ such that $\eta=1$ on $\Phi(\supp \varphi)$ and $\supp \eta \subset V'$. Then, for any $f\in C^r(M)$ we have
\[
H(f)=H(\eta f)\qquad \text{and}\qquad \|\eta f\|_{\calF(M)}\leq C \|f\|_{\calF(M)}.
\]
Applying \eqref{lk} for $\eta f$ and $K=\supp \eta$ finishes the proof of Theorem \ref{crm}.
\end{proof}

The following lemma shows that the bounded extension $H: L^p(M) \to L^p(M)$, $1\le p<\infty$, in Theorem \ref{crm}  coincides with a pointwise formula in Definition \ref{H}.

\begin{lemma}\label{rep}
Suppose that $H \in \mathcal H(M)$ is localized on open and precompact set  $U \subset M$. That is, $H:C_0(M)\to C_0(M)$ is a finite sum of simple $H$-operators
\begin{equation}\label{rep1}
H = \sum_{i=1}^k H_{\vp_i, \Phi_i, V_i},
\end{equation}
where  $\Phi_i$ are diffeomorphisms defined on open subsets $V_i \subset M$ such that $V_i, \Phi_i(V_i) \subset U$ and $\vp_i \in C^\infty(M)$ have supports $\supp \vp_i \subset V_i$, $i=1,\ldots,k$. Then, the bounded extension $H: L^p(M) \to L^p(M)$, $1\le p<\infty$, is given for all $f\in L^p(M)$ by
\begin{equation}\label{rep2}
Hf(x) = \sum_{i=1}^k  H_{\vp_i, \Phi_i, V_i}f(x) \qquad\text{for a.e. }x\in M.
\end{equation}
\end{lemma}

\begin{proof}
Choose a sequence of function $\{f_n\}_{n\in \N}$ in $C_c(M)$ such that $f_n \to f$ in $L^p(M)$ norm and pointwise a.e. as $n\to \infty$. Then $Hf_n \to Hf$ in $L^p$ as $n\to \infty$. By choosing a subsequence we can assume that $Hf_n \to Hf$ pointwise a.e. as $n\to \infty$. Now it suffices to apply \eqref{rep2} for each $f_n\in C_c(M)$ and take a limit as $n\to \infty$.
\end{proof}

As a consequence of Definition \ref{H} and Lemma \ref{rep} we have the following useful fact. The proof of Lemma \ref{skon} is left to the reader. Let $L^1_{loc}(M)$ be the space of locally integrable functions with respect to the Riemannian measure $\nu$ on $M$.

\begin{lemma}\label{skon} Let $P_U\in  \mathcal H(M)$  be localized on an open and precompact set $U\subset M$. Then, for any two open sets $W_1,W_2 \subset M$ of finite measure and containing $U$ we have
\[
P_U(f \ch_{W_1}) = P_U(f \ch_{W_2})  \qquad\text{for }f\in L^1_{loc}(M).
\]
Hence, we can extend the domain of $P_U$ by defining for  $f\in L^1_{loc}(M)$,
\begin{equation}\label{skon1}
P_U(f) = P_U ( f\ch_{W}) \qquad\text{where }W \supset U \text{ is open and } \nu(W)<\infty.
\end{equation}
In particular, we have 
\begin{equation}\label{skon2}
P_U f(x)=0 \qquad\text{for a.e. }x \in M \setminus U.
\end{equation}
\end{lemma}

\subsection{Adjoints to $H$-operators} 

Ciesielski and Figiel \cite{CF2} used their version of Hestenes operators to define a decomposition of function spaces on a compact smooth manifold.
One of the steps in their proof was the fact that the adjoint to a Hestenes operator is again a Hestenes operator, see  \cite[Lemma 5.8]{CF2}. We show an analogous property for our version 
of $H$-operators.

Let $P\in  \mathcal H(M)$  be localized on an open and precompact set $U$.
Then, $P$, which is defined initially as $P:C_0(M) \to C_0(M)$, extends to a bounded linear operator
$P_p =P:L^p(M) \to L^p(M)$.
Our goal is to identify $(P_p)^* : L^{p'}(M) \to L^{p'}(M)$ as a Hestenes operator which is localized on the same set $U$. 
  
 We begin with a convenient formulation of the change of variables formula for diffeomorphisms between Riemannian manifolds.

\begin{lemma}
\label{wei}
Let $F:M\to N$ be a diffeomorphism between two smooth Riemannian manifolds $M$ and $N$ with Riemannian measures $\nu_M$ and $\nu_N$, resp.
Let $\omega_M \in C^\infty(M)$ be a weight, which defines a measure $\mu_M$ on $M$ by $d\mu_M=\omega_M d\nu_M$. Let $\omega_N \in C^\infty(N)$ be a positive weight satisfying
\begin{equation}\label{wei2}
\omega_M(x) = \omega_N(F(x)) |\det (DF(x))|, \qquad\text{for all }x\in M,
\end{equation}
where $DF(x)$ denotes the differential of $F$ at $x$.  Define $\mu_N:=F_*(\mu_M)$ as a push-forward measure of $\mu_M$, i.e., for any Borel set $B \subset N$, we have $\mu_N(B)=\mu_M(F^{-1}(B))$. Then, $d\mu_N=\omega_N d\nu_N$ and for any $f\in C_c(N)$ we have
\begin{equation}
\label{wei1}
\int_N f(y) \omega_N(y) d\nu_N(y)=\int_M f\circ F(x) \omega_M(x) d\nu_M(x).
\end{equation}
\end{lemma}

\begin{proof}
Lemma \ref{wei} is a consequence of the change of variables formula, see \cite[Theorem I.3.4]{Ch},
\begin{equation*}
\nu_N(B)= \int_{F^{-1}(B)}  |\det (DF(x))| d\nu_M(x) \qquad\text{for Borel sets }B \subset N.
\end{equation*}
Hence, for any integrable function $g$ on $N$ with respect to $\nu_N$ we have
\begin{equation}\label{cvf}
\int_N g(y) d\nu_N(y) = \int_M (g \circ F)(x) |\det (DF(x))| d\nu_M(x).
\end{equation}
Applying the above for $g=f \omega_N$ yields \eqref{wei}.
\end{proof}

The following lemma identifies an adjoint of a simple $H$-operator.

\begin{lemma}
\label{hest.adj}
Let $U \subset M$ be 
 an open and precompact subset of $M$.
Let
$\Phi:V \to V'$ be a $C^\infty$ diffeomorphism between two open subsets $V, V' \subset U$  and let
$\vp: M \to \R$ be a $C^\infty$ be function such that 
\[
\supp \vp =\ov{\{x\in M : \vp(x) \ne 0\}} \subset V.
\]
Consider operators $H=H_{\vp,\Phi,V}$ and 
  $ G=H_{\vp_1 ,\Phi^{-1},V'}$,
where
\[
\vp_1(y)=
\begin{cases}
\vp(\Phi^{-1}(y))\psi_1(y)& y\in V',
\\
0 & y\notin V'.
\end{cases}
\]
and $\psi_1$ is any $C^\infty(M)$ function such that
\[
\psi_1(y)=|\det D\Phi^{-1}(y)|\qquad\text{for } y\in \Phi(\supp \vp).
\]
Then, $G$  is a simple $H$-operator  localized on $U$. Moreover, treating $H$ and $G$ as operators $H_p = H :L^p(M) \to L^p(M)$ and $G_{p'} = G: L^{p'}(M) \to L^{p'}(M)$, where $1 < p < \infty$
and $1/p+1/p'=1$,
we have $(H_p)^*=G_{p'} $.
\end{lemma}

\begin{remark}
Using a standard partition of unity argument, it is possible to construct a function $\psi_1$ appearing in Lemma \ref{hest.adj}. 
\end{remark}

 \begin{proof} The fact that $G$ is a simple Hestenes operator localized on $U$ is a consequence of 
the formula defining $G$.  
Next, 
Lemma \ref{wei} implies that  for all $f,g\in C_c(M)$

\begin{align*}
\int_M H (f ) (x) g(x) d\nu(x) & =
 \int_V \vp(x) f(\Phi(x)) g(x) d\nu(x)
\\ & =
 \int_{V'} \vp(\Phi^{-1}(y)) f(y) g(\Phi(y)) |\det D\Phi^{-1}(y)| d\nu(y)
\\ 
& = \int_M f(y) G (g)(y) d\nu(y).
\end{align*}
Since $C_c(M)$ is dense both in $L^p(M)$ and $L^{p'}(M)$, 
 this equality yields $G =(H_p)^*$.
\end{proof}

\begin{remark}
Lemma \ref{hest.adj} justifies the following notation, which we use in the sequel:
 if $H$ is a simple $H$-operator and $G$ is a simple $H$-operator defined as 
in Lemma \ref{hest.adj}, then we write $H^* = G$. Likewise, we shall use the same convention for general $H$-operators which is justified by the following consequence of Lemma \ref{hest.adj}.
\end{remark}

\begin{corollary}
\label{adjoint}
Let $P\in  \mathcal H(M)$  be localized on open and precompact set $U$. That is, $P =  \sum_{i=1}^m H_i$, where each $H_i= H_{\vp_i, \Phi_i ,V_i}$ 
is a simple $H$-operator satisfying $V_i, \Phi_i(V_i) \subset U$.
Then, $Q = \sum_{i=1}^m (H_i)^* \in  \mathcal H(M)$ is  localized on  $U$ and $(P_p)^* = Q_{p'}$ for all $1 < p < \infty$. In particular, the action of $Q$ on $C_0(M)$ does not depend 
on a representation of $P$ as a combination of simple $H$-operators.
\end{corollary}
\begin{proof}
Since $ P\in  \mathcal H(M)$ is localized on $U$, so by definition there are simple $H$-operators 
$H_i = H_{\vp_i, \Phi_i ,V_i}$ with $V_i, \Phi_i(V_i) \subset U$ such that $P = \sum_{i=1}^m H_i$.
Let   $Q = \sum_{i=1}^m (H_i)^*$. Then  by Lemma \ref{hest.adj}, $Q$ is a Hestenes operator localized on $U$, and it follows 
that $(P_p)^* = Q_{p'}$ for all $1 < p < \infty$.
\end{proof}

\subsection{Smooth decomposition of identity}

 We are interested in
obtaining a version of Theorem \ref{Main} for spaces $L^p(M)$, $1\leq p<\infty$ and $C_0(M)$. For this we introduce
the concept of smooth decomposition of identity which is a generalization of smooth orthogonal decomposition of identity in $L^2(M)$ from Theorem \ref{Main}.

\begin{definition}\label{sdi}
Let $\mathcal U$ be an open and precompact cover of a Riemannian manifold $M$. We say that a family of operators $\{P_U\}_{U\in\mathcal U}$ is a {\it smooth decomposition of identity} in $L^p(M)$, $1\le p<\infty$, subordinate to an open cover $\mathcal U$ if:
\begin{enumerate}[(i)]
\item family $\{P_U\}_{U\in \mathcal U}$ is locally finite, i.e., for any compact $K\subset M$, all but finitely many operators $P_{U}$ such that $U \cap K \ne \emptyset$, are zero, 
\item each $P_U \in \mathcal H(M)$ is localized on an open set $U \in\mathcal U$,
\item each $P_U: L^p(M) \to L^p(M)$ is a projection,
\item $P_U \circ P_{U'}=0$ for any $U \neq U' \in \mathcal U$,
\item $\sum_{U\in\mathcal U}P_U= \mathbf I$, where $\mathbf I$ is the identity in $L^p(M)$ and the convergence is unconditional in strong operator topology,
\item 
there exists a constant $C>0$ such that
\begin{equation}\label{unc0}
\frac 1{C} ||f||_{p}  \le \bigg( \sum_{U \in\mathcal U} ||P_U f||_p^p \bigg)^{1/p} \le C ||f||_p \qquad\text{for all }f\in L^p(M).
\end{equation}
\end{enumerate}
\end{definition}

\begin{remark}\label{rsdi}
The above definition can be extended to $p=\infty$ by replacing $L^p(M)$ by $C_0(M)$ and \eqref{unc0} by
\[
(||P_U f||_\infty)_{U\in\mathcal U} \in c_0(\mathcal U) \quad\text{and}\quad
\frac 1{C} ||f||_{\infty} \le  \sup_{U \in\mathcal U} ||P_U f||_\infty \le C ||f||_\infty \qquad\text{for all }f\in C_0(M).
\]
The constant $C$ appearing above or in \eqref{unc0}
is called a {\it decomposition constant} for $\{P_U\}_{U\in \mathcal U}$ in $C_0(M)$ or in $L^p(M)$, $1\leq p <\infty$, respectively. 

When $p=2$, we shall require that the decomposition constant $C=1$, which forces projections $P_U$ to be orthogonal and satisfy \eqref{sum}. 
Consequently, when $p=2$, Definition \ref{sdi} is consistent with the concept of a smooth decomposition orthogonal decomposition of identity in $L^2(M)$ as in Theorem \ref{Main}.
\end{remark}

The following proposition shows that property (v) in Definition \ref{sdi} is automatically implied by the other conditions.

\begin{proposition}
Let $1\le p <\infty$. Suppose that $\{P_U\}_{U \in \mathcal U}$ is a family of operators on $L^p(M)$ satisfying conditions (iii), (iv), and (vi) in Definition \ref{sdi}. Then, (v) holds. The same holds for $p=\infty$ by replacing $L^p(M)$ by $C_0(M)$ as in Remark \ref{rsdi}.
\end{proposition}

\begin{proof} Without loss of generality we can assume that $\mathcal U$ is at most countable and each $P_U$, $U \in \mathcal U$ is non-zero. For $U \in \mathcal U$, define a closed subspace $X_U \subset L^p(M)$ by $X_U=P_U(L^p(M))$. Then, properties (iii), (iv), and (vi) imply that the linear span of subspaces $X_U$, $U\in\mathcal U$, is dense in $L^p(M)$. Moreover, the collection $\{X_U\}_{U \in\mathcal U}$ is a strong unconditional basis of $L^p(M)$ in the sense of Nazarov and Treil \cite[Section 4]{NT} for the sequence space $Y=\ell^p(\mathcal U)$. This is a special case of an unconditional Schauder decomposition, see \cite[Section 1.g]{LT}.  Alternatively, $L^p(M)$ is isomorphic with  $\ell^p$-direct sum of subspaces $X_U$, $U\in\mathcal U$, 
\[
L^p(M) \cong \bigg(\bigoplus_{U \in\mathcal U} X_U\bigg)_p,
\]
see \cite[Section II.B.21]{Woj}.
Then, \cite[Proposition 4.1]{NT} implies the property (v). 
The case $p=\infty$ is shown the same way.
\end{proof}

We have the following duality of smooth decompositions of identity.

\begin{theorem}
Suppose that  a family of operators $\{P_U\}_{U\in\mathcal U}$ is a  smooth decomposition of identity in $L^p(M)$, $1< p<\infty$, as in Definition \ref{sdi}. Then, $\{(P_U)^*\}_{U\in\mathcal U}$ is a smooth decomposition of identity in $L^{p'}(M)$, $1/p+1/p'=1$. 
\end{theorem}

\begin{proof}
By \cite[Proposition 4.5]{NT}, the properties (iii)--(vi) for $\{P_U\}_{U\in\mathcal U}$ in $L^p(M)$ imply the same properties for $\{(P_U)^*\}_{U\in\mathcal U}$ in $L^{p'}(M)$. The property (i) for adjoints $\{(P_U)^*\}_{U\in\mathcal U}$ is an immediate consequence of (i) for $\{P_U\}_{U\in\mathcal U}$. Finally, Corollary \ref{adjoint} implies that if $P_U\in \mathcal H(M)$ is localized on $U$, then the same holds for $(P_U)^*$.
\end{proof}

In Section \ref{S6} we shall establish the existence of a smooth decomposition of identity in $L^p(M)$, $1\leq p< \infty$, and in $C_0(M)$, which is subordinate to any open and precompact cover of $M$, see Theorem \ref{Main2}.

\subsection{Technical Lemmas}
We need to use some standard ``abstract nonsense'' facts about weighted Lebesgue spaces on general measure spaces.
 
\begin{proposition}\label{absn}
Let $(\Omega,\mu)$ be a positive measure space and let $\kappa: \Omega \to (0,\infty)$ be a measurable weight. Let $1\le p<\infty$. Define a multiplication operator 
\[
M_{\kappa^{-1/p}}: L^p(d\mu) \to L^p(\kappa d \mu), \qquad M_{\kappa^{-1/p}}f = \kappa^{-1/p} f.
\]
Then:
\begin{enumerate}[(i)]
\item
$M_{\kappa^{-1/p}}$ is an isometric isomorphism 
\[
\begin{aligned}
(M_{\kappa^{-1/p}})^{-1} &= M_{\kappa^{1/p}}:   L^p(\kappa d \mu) \to L^p(d\mu) \qquad\text{and}
\\
(M_{\kappa^{-1/p}})^* &= M_{\kappa^{1/p'}}:  L^{p'}(\kappa d \mu) \to L^{p'}( d\mu).
\end{aligned}
\]
\item
If $T: L^p(d\mu) \to L^p(d\mu)$ is bounded linear operator, then 
\[
\tilde T = M_{\kappa^{-1/p}} \circ T \circ M_{\kappa^{1/p}} : L^p(\kappa d \mu) \to L^p(\kappa d \mu),
\]
is also a bounded linear operator with the same norm $||\tilde T||=||T||$.
\item
If $T_1, T_2: L^p(d\mu) \to L^p(d\mu)$ are bounded linear operators, then $(T_1 \circ T_2)\tilde{} = \tilde T_1 \circ \tilde T_2$.
\item 
If $T_i:  L^p(d\mu) \to L^p(d\mu)$, $i\in I$, are bounded linear operators such that
\[
||f||_p \asymp \bigg( \sum_{i\in I} ||T_i f ||_p \bigg)^{1/p} \qquad\text{for all }f \in L^p(d\mu),
\]
then 
\[
||g||_p \asymp \bigg( \sum_{i\in I} ||\tilde T_i g||_p \bigg)^{1/p} \qquad\text{for all }g \in L^p(\kappa d\mu),
\]
with the same equivalence constants.
\item
If $T: L^p(d\mu) \to L^p(d\mu)$ is bounded linear operator, then let $S=T^*: L^{p'}(d\mu) \to L^{p'}(d\mu)$, $1/p+1/p'=1$, be its adjoint. Then
\[
\tilde S = M_{\kappa^{-1/p'}} \circ S \circ M_{\kappa^{1/p'}}: L^{p'}(\kappa d\mu) \to L^{p'}(\kappa d\mu)
\]
is the adjoint of $\tilde T$, i.e., $\tilde S= (\tilde T)^*$. 
\item
If $T, T_i:  L^p(d\mu) \to L^p(d\mu)$, $i\in \N$, are bounded linear operators such that
\[
\sum_{i=1}^\infty T_i  = T \qquad\text{in strong operator topology in $L^p(d\mu)$},
\]
then
\[
\sum_{i=1}^\infty \tilde T_i  = \tilde T \qquad\text{in strong operator topology in $L^p(\kappa d\mu)$}.
\]
\end{enumerate}
\end{proposition}
 
In our arguments we will consider weighted Lebesgue spaces $L^p(M,\omega)$, $1\le p<\infty$, where {\it weight} $\omega \in C^\infty(M)$ is a positive smooth function on Riemannian manifold $M$. The norm of a measurable function $f$ on $M$ is given by
\[
||f||_{L^p(M,\omega)} = \bigg(\int_M |f|^p \omega d\nu\bigg)^{1/p},
\]
where $\nu$ is the Riemannian measure on $M$. In particular, $L^2(M,\omega)$ is  a Hilbert space with the inner product
\[
\lan f,h \ran_\omega=\int_{M} f(x) h(x) \omega(x) d\nu(x).
\]

The following lemma allows transferring of smooth decompositions of identity for any weight.

\begin{lemma}
\label{weight1}
Let $M$ be a smooth Riemannian manifold. Let $\omega, \tilde{\omega}\in C^\infty(M)$ be two positive weights on $M$ and $1\le p<\infty$. 
\begin{enumerate}[(i)]
\item
Suppose that $P_U \in \mathcal H(M)$ is localized on an open set $U\subset M$, which induces a projection on $L^p(M,\omega)$. Define an operator $\tilde P_U$ by 
\begin{equation}\label{wei0}
\tilde{P}_U(f)=P_U ( f \cdot ( \tilde\omega/ \omega )^{1/p} ) \cdot ( \omega/\tilde\omega )^{1/p}
\qquad\text{for }f\in C_0(M).
\end{equation}
Then, an operator $\tilde P_U \in \mathcal H(M)$ is localized on $U$, which induces a  projection on $L^p(M,\tilde \omega)$, and
\begin{equation}\label{pup}
||\tilde P_U||_{L^p(M,\tilde \omega) \to L^p(M,\tilde \omega)} = || P_U||_{L^p(M,\omega) \to L^p(M,\omega)}.
\end{equation}

\item
Suppose $\{P_U\}_{U\in \mathcal U}$ is a smooth decomposition of identity in $L^p(M,\omega)$ in the sense of Definition \ref{sdi}. Then, the family $\{\tilde{P}_U\}_{U\in \mathcal U}$, that is defined by \eqref{wei0}, is  a smooth decomposition of identity in  $L^p(M,\tilde \omega)$ with the same decomposition constant.

\end{enumerate}
\end{lemma}

\begin{proof}
Since $P_U\in \mathcal H(M)$, $P_U$ is a finite sum of simple $H$-operators
$H_{\vp,\Phi,V}$ as in Definition \ref{H}. Note that for $x\in V$ and $f\in C_0(M)$
\[
H_{\vp,\Phi,V}\left(f \bigg( \frac{\tilde\omega}{\omega}\bigg)^{1/p}  \right)(x) \bigg( \frac{\omega(x)}{\tilde\omega(x)}\bigg)^{1/p}=
\vp(x)\bigg(\frac{\omega(x)}{\tilde\omega(x)} \bigg)^{1/p} \bigg( \frac{\tilde\omega(\Phi (x))}{\omega(\Phi (x))}\bigg)^{1/p}  f(\Phi (x)).
\]
Hence, $\tilde P_U$ is also a finite sum of simple $H$-operators $H_{\tilde \vp,\Phi,V}$ with appropriately modified $\tilde \vp$'s.
This proves that $\tilde P_U\in \mathcal H(M)$ is localized on $U$ by Definition \ref{localized}. To complete the proof of (i), we consider a weight $\kappa=\frac{\tilde\omega}{\omega}$ and a measure $\mu$ on $M$ given by $d \mu = \omega d \nu$. Then, $\kappa d\mu =\tilde  \omega d \nu$ and Proposition \ref{absn}(ii)--(iii) shows that $\tilde P_U$ is a projection satisfying \eqref{pup}.
Finally, part (ii) follows by direct verification of all properties of smooth decomposition of identity in Definition \ref{sdi} using Proposition \ref{absn}. 
\end{proof}

Using Lemma \ref{wei} smooth decompositions of identity can be transferred via diffeomorphisms. 

\begin{lemma}
\label{weight2}
Let $F:M\to N$ be a diffeomorphism between two smooth Riemannian manifolds $M$ and $N$. Let $\omega_M \in C^\infty(M)$ and $\omega_N \in C^\infty(N)$ be positive weights satisfying \eqref{wei2} and $1\le p<\infty$.

\begin{enumerate}[(i)]
\item
Suppose that $P_U \in \mathcal H(M)$ is localized on an open set $U\subset M$ and it induces a projection on $L^p(M,\omega_M)$. Define an operator $Q_B$, where $B=F(U)$ by
\begin{equation}
\label{l3}
Q_B(f) =P_{U}(f\circ F)\circ F^{-1} \qquad\text{for }f\in C_0(N).
\end{equation}
Then, an operator $Q_B \in \mathcal H(N)$ is localized on $B$ and it induces a projection on $L^p(N,\omega_N)$ with the same norm as $P_U$.

\item
 Let  $\{P_U\}_{U\in \mathcal U}$ be a smooth decomposition of identity in $L^p(M,\omega_M)$. 
For every ${B \in \mathcal B}$ in the open cover $\mathcal B:=F(\mathcal U)$ of $N$, define operators $Q_B$ by \eqref{l3}.
Then $\{Q_B\}_{B\in\mathcal B}$ is a smooth decomposition of identity in $L^p(N,\omega_N)$ with the same decomposition constant.

\end{enumerate}
\end{lemma}

\begin{proof} 
Since $P_U\in \mathcal H(M)$, $P_U$ is a finite sum of simple $H$-operators $H=H_{\vp,\Phi,V}$ as in Definition \ref{H}. 
Define an operator $G$ acting on functions $f\in C_0(N)$ by
\[
G(f)(x):=H(f\circ F)\circ F^{-1}(x)
=\begin{cases} \vp(F^{-1}(x)) f( F\circ \Phi \circ F^{-1}(x)) & x\in F(V)\\
0 & \text{otherwise}.
\end{cases}
\]
Hence, $G$ is a simple $H$-operator, 
\[
H^N(f)=H_{\tilde\vp,\tilde\Phi,\tilde V}(f)
\]
 with $\tilde\vp=\vp\circ F^{-1}$, $\tilde{V}=F(V)$, $\tilde\Phi=F\circ \Phi \circ F^{-1}$. Thus, $Q_B \in \mathcal H(N)$, where $B=F(U)$. Since $P_U$ is localized on $U$, it is easy to verify that $Q_B$ is localized on $B$ by Definition \ref{localized}.
 
The diffeomorphism $F: M \to N$ is a measure preserving transformation between measure spaces $(M, \omega_M d\nu_M)$ and $(N,\omega_N d\nu_N)$. For each $p$, it induces an isometric isomorphism 
\[
T: L^p(M,\omega_M) \to L^p(N,\omega_N), \qquad T(f)= f \circ F^{-1}.
\]
By \eqref{l3}, $Q_B: L^p(N,\omega_N) \to L^p(N,\omega_N)$ satisfies $Q_B= T \circ P_U \circ T^{-1}$. Hence, $Q_B$  induces a projection on $L^p(N,\omega_N)$ with the same norm as $P_U$.
Likewise, it is a matter of a simple verification that if properties (i)-(vi) in Definition \ref{sdi} hold for family $\{P_U\}_{U\in\mathcal U}$, then they also hold for $\{Q_B\}_{B\in\mathcal B}$. 
\end{proof}

\section{Background about Morse functions}\label{S3}

In this section we will show some rudimentary facts in Morse Theory following the books of Hirsch \cite[\S 6]{Hi} and Milnor \cite[\S 6]{Mi}. Instead of studying topological properties of smooth manifolds $M$ as in \cite{Hi}, we will merely employ Morse functions to obtain a convenient local decomposition of a Riemannian measure on $M$ as a product of measures on an interval and a level surface of $M$. In the absence of critical points this is a consequence of the regular interval theorem, see \cite[Theorem 2.2 in \S 6]{Hi}. In the presence of a critical point, it is a measure-theoretic analogue of topological result on attaching cells, see \cite[Theorem 3.1 in \S 6]{Hi}.

Let $M$ be a Riemannian manifold of dimension $d$. We say that $m: M \to \R$ is a Morse function if all critical points of  $m$ are nondegenerate. That is, the $d\times d$ Hessian matrix of $m$ has rank $d$ at every critical point. The following fact can be easily deduced from well-known properties of Morse functions, see \cite{Hi} and \cite{Mi}.

\begin{theorem}\label{mfact}
Suppose that $M$ is a connected Riemannian manifold (without boundary). Then, there exists a Morse function $m: M \to [0,\infty)$ such that:
\begin {itemize}
\item preimages $m^{-1}([0,b])$, $b>0$, are compact; in particular, level sets $m^{-1}(t)$ are compact for each $t\ge 0$,
\item every critical value corresponds to exactly one critical point.
\end{itemize}
Moreover, $m$ maps $M$ onto $[0,\infty)$ or $[0,1]$, if $M$ is non-compact or compact, respectively.  
\end{theorem}

\begin{proof}
By \cite[Corollary 6.7]{Mi} on any differentiable manifold there exists a Morse function with compact level sets. In fact, Milnor's argument shows that preimages $m^{-1}([0,b])$, $b>0$, are compact.
By Morse's lemma \cite[Lemma 1.1 in \S 6]{Hi}, critical values are isolated and each critical value corresponds to finitely many critical points. Moreover, by \cite[Theorem 1.2 in \S 6]{Hi} Morse functions form a dense open set in $C^\infty(M,\R)$ in a suitable strong topology \cite[Ch. 2.1]{Hi}. Hence, a perturbation argument as in \cite[p. 162]{Hi} yields a Morse function with critical values corresponding to exactly one critical point.
\end{proof}

For a regular value $t$, let $J_t= m^{-1}(t)$ be a level submanifold equipped with the Riemannian metric inherited from $M$. The following result is a variant of the regular interval theorem \cite[Theorem 2.2 in \S6]{Hi}.

\begin{theorem}\label{ril}
Let $m:M \to [0,\infty)$ be a Morse function as in Theorem \ref{mfact}. Let $I=(a,b)$ be an open interval such that $m$ has no critical values in $I$. Then there exists a family of diffeomorphisms $\{F_{\th,t}\}_{\th, t \in I}$ between level submanifolds
$
F_{\th,t}:J_\th \to J_t
$
such that:
\begin{equation}
\label{ril1}
F_{\th,t}^{-1}=F_{t,\th}: J_t \to J_\th,
\end{equation}
\begin{equation}
\label{ril2}
F_{t_2,t_3}\circ F_{t_1,t_2}=F_{t_1,t_3} \qquad\text{for }t_1,t_2,t_3\in I,
\end{equation}
and for each $\th \in  I$, the formula $F_\th(t,x)=F_{\th,t}(x)$, $(t,x) \in I \times J_\th$ defines a diffeomorphism 
\begin{equation}\label{ril3}
F_\th:I \times J_\th \to  M_I:=m^{-1}(I)=\bigcup_{t\in I} J_t\subset M.
\end{equation}
For any $\th \in I$ there exists a smooth function $\psi=\psi_\th$ on $I \times J_\th$ such that the pushforward of  the Riemannian measure $\nu_M$ under $F_\th^{-1}$ is
\begin{equation}
\label{ril4}
(F^{-1}_\th)_*(\nu_M)= \psi  (\lambda \times \nu_\th),
\end{equation}
where $\lambda$ is the Lebesgue measure on $I$ and $\nu_\th$ is the Riemannian measure on $J_\th$. Moreover, if $m$ has no critical values in $\bar I=[a,b]$, then there there exist constants $c_1,c_2>0$ such that 
\begin{equation}\label{ril5}
c_1\leq \psi(t,x)\leq c_2 \qquad\text{for all } (t,x) \in I \times J_\th.
\end{equation}
\end{theorem}

\begin{proof} The proof follows along the lines of the proof of  \cite[Theorem 2.2 in \S6]{Hi}. Let $\grad m$ be the gradient vector field corresponding to $m$. Let $X$ be a renormalized gradient vector field on $M$ given by
\begin{equation}\label{ril7}
X(x)=\frac{\grad m(x)}{|\grad m(x)|^2}.
\end{equation}
The vector field is well-defined and smooth everywhere on $M$ except the critical points of $m$. 

For any $\th \in I=\R$ and $x\in M$, we consider the initial value problem
\begin{equation}\label{ode}
\begin{cases}
\eta'(t)=X(\eta(t)) \\
\eta(\th)=x.
\end{cases}
\end{equation}
By the existence and uniqueness theorem for ordinary differential equations (ODEs), for every $x\in M$, which is not a critical point, there exists a maximal open interval $I(x)$ containing $\th$, such that the solution $\eta(t)$ satisfying \eqref{ode} exists for all $t\in I(x)$. 

In addition, if $\th \in I=(a,b)$ and $x\in J_\th$, then a direct calculation as in \cite{Hi} shows that $\eta(t) \in J_t$. 
In particular, for any $x\in J(\th)$, the maximal existence interval $I(x)$ contains $I$. Hence, the solution curves satisfying \eqref{ode} define a mapping $F_\th: I \times J_{\th} \to M$ given by $F_\th(t,x)=\eta(t)$. By the differentiability of solutions of ODEs, the mapping $F_\th$ is smooth and satisfies
\begin{equation}
\label{diff}
F_\th(\th,x)=x \quad\text{and}\quad F_{\th,t}(x):=F_\th(t,x) \in J_t \qquad\text{for } (t,x) \in I \times J_\th.
\end{equation}
Since $F_{t,\th}(F_{\th,t}(x))=x$, the mapping  $F_{\th,t}: J_\th \to J_t$ is a diffeomorphism satisfying \eqref{ril1}. Likewise, \eqref{ril2} follows from the uniqueness of solution curves of the vector field $X$. Since the solution curves of the gradient vector field are transverse to level submanifolds, $F_\th$ is an immersion. Hence, $F_\th$ is a diffeomorphism between $I \times J_\th$ and $M_I$.

Define a weight $\psi$ by
\begin{equation}\label{epsi}
\psi(t,x)  =|\det D F_\th(t,x)| \qquad\text{for }(t,x) \in I \times J_\th.
\end{equation}
By the change of variables formula \eqref{cvf} for any $f \in C_0(M_I)$,
\begin{equation}\label{ril6}
\int_{m^{-1}(I)} f(x) d\nu_M(x) = \int_I \int_{J_\th} (f \circ F_\th) (t,x) \psi(t,x) dt d\nu_\th(x).
\end{equation}
This implies that the pushforward measure relation $(F_\th)_*( \psi  (\lambda \times \nu_\th))=\nu_M$. Hence, \eqref{ril4} holds. If $m$ has no critical values in $\bar I$, then we can extend $F_\th$ to $\tilde I \times J_\th$, where $\tilde I$ is an open interval containing $\bar I$. 
Since $\bar I \times J_\th$ is compact, the differential of $F_\th$ is uniformly bounded and has uniformly bounded inverse on  $\bar I \times J_\th$, which implies \eqref{ril5}.
\end{proof}

We shall need an analogue of Theorem \ref{ril} that deals with critical points.

\begin{theorem}\label{cpl}
Let $m:M \to [0,\infty)$ be a Morse function as in Theorem \ref{mfact}.  Let $I=(a,b) \subset m(M)$ be an open interval and $t_z \in I$ be a unique critical value of $m$ in $\ov{I}=[a,b]$, which corresponds to a single critical point $z\in M$. Let $U_z$ be an open neighborhood of $z\in M$. Then the following holds:
\begin{enumerate}[(i)]
\item
There exist open submanifolds $\tilde J_{t} \subset J_t$, $t\in I$, and a family of diffeomorphisms $F_{\th,t}: \tilde J_{\th} \to \tilde J_{t}$ such that the analogues of \eqref{ril1}, \eqref{ril2}, and \eqref{ril3} are satisfied with $J_t$ replaced by $\tilde J_t$. 
\item
There exists an open neighborhood $V_z$ of $z\in M$  and $\delta_z>0$ such that
\begin{equation}\label{cpl1}
J_t \setminus \tilde J_t \subset V_z \subset \ov{V_z} \subset  U_z \qquad\text{for all }t \in [t_z-\delta_z,t_z+\delta_z],
\end{equation}
and
\begin{equation}\label{cpl2}
J_t \setminus F_{\th,t}(J_{\th} \setminus V_z) \subset U_z \qquad\text{for all }\th, t \in [t_z-\delta_z,t_z+\delta_z].
\end{equation}
\item
For any $\th \in I$ there exists a smooth function $\psi=\psi_\th$ on $I \times \tilde J_\th$ such that the pushforward of the Riemannian measure $\nu_M$, restricted to $\tilde M_I=\bigcup_{t\in I} \tilde J_t$, under $F_\th^{-1}$ is given by
\begin{equation}
\label{cpl4}
(F^{-1}_\th)_*(\nu_M)= \psi  (\lambda \times \nu_\th),
\end{equation}
where $\nu_\th$ is the Riemannian measure restricted to $\tilde J_\th \subset J_\th$. Moreover, \eqref{ril5} holds with $J_\th$ replaced by $\tilde J_\th$.
\end{enumerate}
\end{theorem}

\begin{figure}
\includegraphics{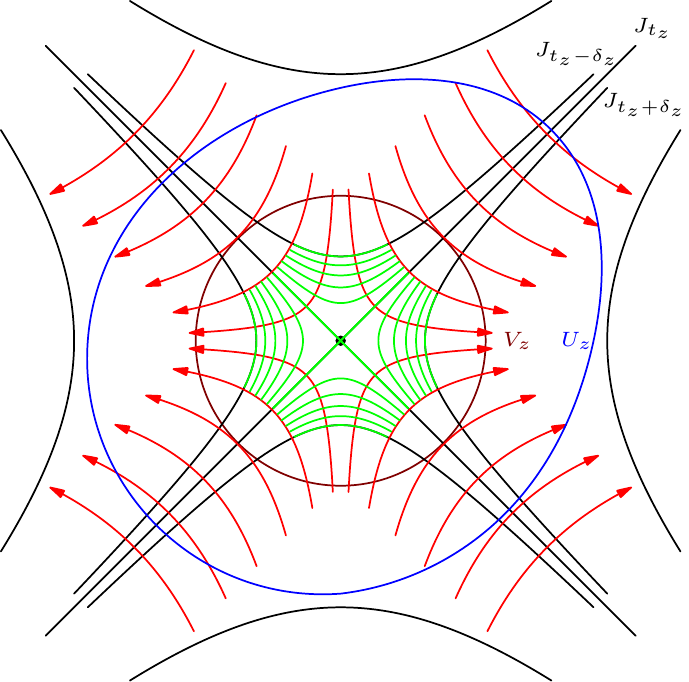}\includegraphics{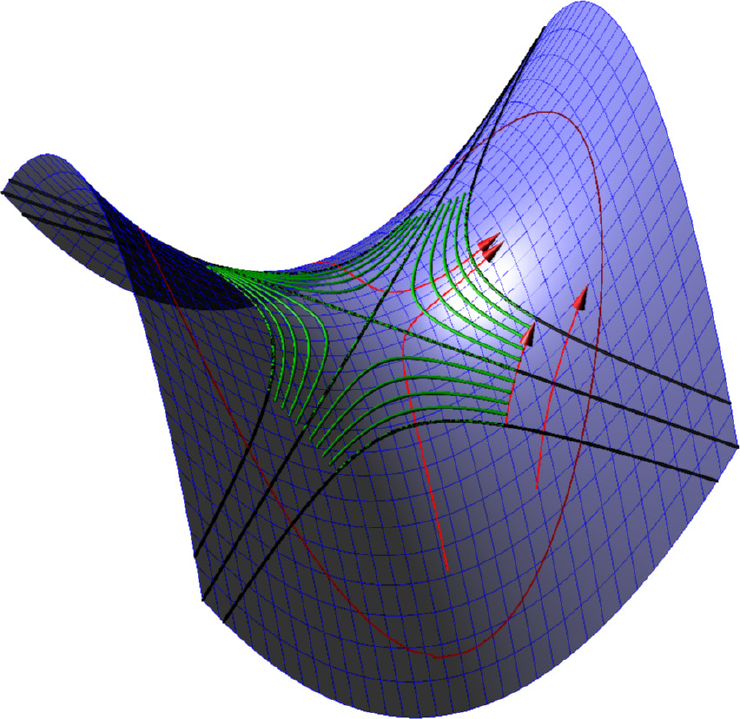}
\caption{Visualization of properties \eqref{cpl1} and \eqref{cpl2} in Theorem \ref{cpl}. Green curves represents portions of the level sets $J_t \setminus \tilde J_t$, $|t-t_z| \le \delta_z$, near the critical point $z$. The solution curves of the flow between level sets are in red.}
\end{figure}

In the proof of Theorem \ref{cpl} we shall employ the following elementary lemma about transition times of solution curves. Let $B(z,r)$ be an open ball centered at $z \in M$ and radius $r>0$ with respect to a geodesic distance $\di$ on $M$. 

\begin{lemma}\label{time}
Let $z\in M$ and $r>0$ be such that the closed ball $\bar B(z,2r)$ is compact in $M$. Let $X$ be a smooth and non-zero vector field defined on some open neighborhood of $\bar B(z,2r) \setminus B(z,r)$.
Then, there exists $\delta>0$, such that whenever $\eta: [t_0,t_1] \to M$ is a solution curve of $X$ such that
\begin{equation}\label{time1}
\begin{cases}
\eta'(t)=X(\eta(t)), \\
\eta(t_0) \in B(z,r) \quad\text{and}\quad \eta(t_1) \not \in B(z,2r),
\end{cases}
\end{equation}
then $|t_0-t_1| \ge \delta$.
\end{lemma}

\begin{proof} Since $K=\bar B(z,2r) \setminus B(z,r)$ is compact in $M$, there exists a constant $c>0$ such that
\begin{equation}\label{rest}
||X(x)||:=\sqrt{g(x)\cdot (X(x),X(x))}\le c \qquad\text{for all }x \in K.
\end{equation}
Let $[s_0,s_1] \subset [t_0,t_1]$ be a subinterval on which $\eta$ travels through $K$. That is,
\[
\di(z,\eta(s_0))=r, \qquad \di(z,\eta(s_1))=2r, \qquad\text{and } \eta(s)\in K \text{ for all }s\in [s_0,s_1].
\]
Then,
\[
r \le \di(\eta(s_0),\eta(s_1)) \le \int_{s_0}^{s_1} ||\eta'(s)|| ds =   \int_{s_0}^{s_1} ||X(\eta(s))|| ds \le (s_1-s_0)c.
\]
Hence, we conclude Lemma \ref{time} with $\delta=r/c$.
\end{proof}

\begin{proof}[Proof of Theorem  \ref{cpl}]
Choose $r>0$ such that $B(z,8r) \subset U_z$ is precompact. Let $X$ be a normalized gradient vector field of $m$ given by \eqref{ril7}. Let $\delta_i$ be the minimum transition time for the annulus $B(z,2^i r) \setminus B(z,2^{i-1}r)$, where $i=1,2$, or $3$, which is given by Lemma \ref{time}. Fix $0<\delta_z \le \frac13 \min(\delta_1,\delta_2,\delta_3)$ such that $[t_z-\delta_z,t_z+\delta_z] \subset (a,b)$.

For any $\th \in [t_z-\delta_z,t_z+\delta_z]$, let $\eta$ be the solution curve to the initial value problem 
\begin{equation}\label{ode2}
\begin{cases}
\eta'(t)=X(\eta(t)) \\
\eta(\th)=x, \qquad x\in J_\th \setminus B(z,2r),
\end{cases}
\end{equation}
We claim that the solution $\eta$ exists on the interval $I$ and it satisfies
\begin{equation}\label{ode3}
\eta(t) \in J_t \setminus B(z,r) \qquad\text{for all } t\in [t_z-\delta_z,t_z+\delta_z].
\end{equation}
On the contrary, suppose that there exists a trajectory satisfying $\eta(t) \in B(z,r)$ for some $t\in [t_z-\delta_z,t_z+\delta_z]$. Applying Lemma \ref{time} on the annulus $B(z,2 r) \setminus B(z,r)$ yields a contradiction with $2\delta_z<\delta_1$.
Hence, \eqref{ode3} is a consequence of \eqref{ril7}, which guarantees that trajectories $\eta(t)$ of the vector field $X$ travel through level submanifolds $J_t$ precisely at time $t$ as in the proof of Theorem \ref{ril}. Since the solution $\eta$ stays away from the critical point $z$, which is the only singularity of $X$ on $m^{-1}(I)$, it exists on the interval $I$.

For some fixed choice of $\th_0 \in [t_z-\delta_z,t_z+\delta_z]$, we define
\[
\tilde J_{\th_0}=J_{\th_0} \setminus \bar B(z,2r).
\]
By the above claim, for any $t\in I$, we can define
\begin{equation}\label{cpl5}
\tilde J_t = \{ \eta(t): \eta \text{ is a solution of \eqref{ode2} for some } x \in \tilde J_{\th_0} \}.
\end{equation}
By \eqref{ode3} we have
\begin{equation}\label{cpl16}
\tilde J_t \cap B(z,r) =\emptyset \qquad\text{for } t\in  [t_z-\delta_z,t_z+\delta_z].
\end{equation}
It is less obvious that
\begin{equation}\label{cpl17}
J_t \setminus \tilde J_t \subset B(z,4r)
\qquad\text{for } t\in  [t_z-\delta_z,t_z+\delta_z]. 
\end{equation}
On the contrary, suppose that for some $t$ and $y \in J_t \setminus \tilde J_t$ we have $y \not \in B(z,4r)$. Consider the solution curve $\eta$ of the vector field $X$ with the initial condition $\eta(t)=y$. Then, the solution $\eta$ exists on $I$ and by Lemma \ref{time} on the annulus $B(z,4r) \setminus B(z,2r)$, it satisfies
\[
\eta(s) \not \in B(z,2r) \qquad\text{for all } s\in [t_z-\delta_z,t_z+\delta_z].
\]
In particular, $x=\eta(\th_0) \not \in B(z,2r)$. Hence, $x\in \tilde J_{\th_0}$. By \eqref{cpl5}, this forces $y\in \tilde J_t$, which is a contradiction. This proves \eqref{cpl17}, and hence \eqref{cpl1} holds for $V_z=B(z,4r)$.

For any $\th \in I$, we define the mapping
\[
F_\th: I \times \tilde J_\th \to \tilde M_I:=\bigcup_{t\in I} \tilde J_t\subset M
\]
in the same way as in the proof of Theorem \ref{ril}. When moving along the trajectories of \eqref{ode2} between the level sets $\tilde J_t$ near the critical value $|t-t_z|\le \delta_z$, we stay away from the ball $B(z,r)$ by \eqref{cpl16}. Hence, $F_\th$ is well-defined and $F_{\th,t}: \tilde J_\th \to \tilde J_t$ given by $F_{\th,t}(x)=F_{\th}(t,x)$ satisfies \eqref{ril1} and \eqref{ril2}. Moreover, $F_\th$ is a diffeomorphism for the same reasons as in the proof of Theorem \ref{ril}. 

Note that $F_\th$ extends smoothly to the closure of $I \times \tilde J_\th$. Hence, the differential of $F_\th$ is uniformly bounded and has uniformly bounded inverse for all $(t,x) \in I \times \tilde J_\th$. As in the proof of Theorem \ref{ril} this yields \eqref{cpl4} with the weight $\psi$ given by \eqref{epsi}.

Finally, to prove \eqref{cpl2}, take any $y\in J_{t} \setminus B(z,8r)$, where $t\in  [t_z-\delta_z,t_z+\delta_z]$. Since the transition time of the vector field $X$ for annulus $B(z,8r) \setminus B(z,4r)$ is at least $3\delta_z$, we deduce that $x=F_{t,\th}(y) \not \in B(z,4r)$ for any $\th \in  [t_z-\delta_z,t_z+\delta_z]$. Hence, we have $y=F_{\th,t}(x)$ for some $x\in J_\th \setminus V_z$, where $V_z=B(z,4r)$. Since $B(z,8r) \subset U_z$, this shows \eqref{cpl2}.
\end{proof}

\begin{remark} \label{cplr}
As a corollary of \eqref{cpl1} and \eqref{cpl4}, for any $f\in C_0(M)$ such that $f$ vanishes on $U_z$, we have
\begin{equation}\label{cpl6}
\int_{m^{-1}((t_z-\delta_z,t_z+\delta_z))} f(x) d\nu_M(x) = \int_{t_z-\delta_z}^{t_z+\delta_z} \int_{\tilde J_\th} (f \circ F_\th) (t,x) \psi(t,x) dt d\nu_\th(x).
\end{equation}
Combining Theorems \ref{ril} and \ref{cpl} we can extend the above formula to a larger portion of the manifold $M$. Indeed, suppose that interval $I=(a,b)$ contains exactly one critical point $t_z$.
 Choose $\th\in I$ such that $\th>t_z$. Then, for any $t>t_z$, the diffeomorphisms $F_{\th,t}$ in Theorem \ref{cpl} are merely restrictions of diffeomorphisms in Theorem \ref{ril}. However, if $t\le t_z$, then $F_{\th}$ can be extended by 
\[
F_{\th}(t,x) = \begin{cases} F_{\th,t}(x) & t>t_z \text{ or } x\in  \tilde J_\th,
\\
z & t \le t_z \text{ and } x \in J_\th \setminus  \tilde J_\th.
\end{cases}
\]
Under this convention, $F_\th: I \times J_\th \to M$ is well-defined and its restriction to $I \times \tilde J_\th$ is a diffeomorphism onto its image. 
Likewise, we extend weight function $\psi$ to $I \times J_\th$ in any way.
Combining \eqref{ril6} and \eqref{cpl6} we have that for any $f\in C_0(M)$ that vanish on open neighborhood of $U_z$ of a critical point $z$, we have
\begin{equation}\label{cpl7}
\int_{m^{-1}((t_z-\delta_z,b))} f(x) d\nu_M(x) = \int_{t_z-\delta_z}^{b} \int_{J_\th} (f \circ F_\th) (t,x) \psi(t,x) dt d\nu_\th(x).
\end{equation}
Likewise, if $\th \in I$ is such that $\th<t_z$, then under the same assumptions we have
\begin{equation}\label{cpl8}
\int_{m^{-1}((a,t_z+\delta_z))} f(x) d\nu_M(x) = \int_{a}^{t_z+\delta_z} \int_{J_\th} (f \circ F_\th) (t,x) \psi(t,x) dt d\nu_\th(x).
\end{equation}
\end{remark}

Applying  Theorem \ref{ril} or Theorem \ref{cpl} for another regular value $s\in I$ yields another smooth function $\psi_s$ on $I \times J_s$ or respectively  on $I \times \tilde J_s$. It turns out that $\psi_\th$ and $\psi_s$ must satisfy the following relationship.

\begin{lemma}
For any $\th,s \in I$, the weight functions $\psi_\th$ and $\psi_s$ from Theorem \ref{ril} or Theorem \ref{cpl} satisfy
\begin{equation}\label{psi}
\frac{\psi_\th(t_1,F_{s,\th}(x))}{\psi_\th(t_0,F_{s,\th}(x))}=\frac{\psi_s(t_1,x)}{\psi_s(t_0,x)} \qquad\text{for all }t_0,t_1 \in I \text{and } x\in  J_s \text{ or resp. } x\in \tilde J_s.
\end{equation}
\end{lemma}

\begin{proof}
As a consequence of Theorem \ref{cpl} we have
the following commutative diagram of diffeomorphisms
\begin{equation}\label{diag}
\begin{tikzcd}[column sep=small]
(I\times \tilde J_\th, \psi_\th ({\lambda}\times \nu_\th)) \arrow{rr}{id \times F_{\th,s}} \arrow[swap]{dr}{F_\th}& &(I \times \tilde J_s,\psi_s  ({\lambda}\times \nu_s)) \\
& (\tilde M_I, \nu_d) \arrow[swap]{ur}{F_s^{-1}}&\subset M.
\end{tikzcd} 
\end{equation}
Indeed, the rule \eqref{ril2} yields
\[
F_s(t,F_{\th,s}(x))=F_\th(t,x),\qquad\text{for }(t,x)\in I\times \tilde J_\th.
\]
Moreover, the pushforward measure by these diffeomorphisms are indicated in the diagram \eqref{diag}. Hence,
\[
(F_{\th,s})_*(\psi_\th(t,\cdot)\nu_\th)=\psi_\th(t,F_{s,\th}(\cdot))
(F_{\th,s})_*(\nu_\th)=\psi_s(t,\cdot) \nu_s.
\]
Since $t\in I$ is arbitrary, \eqref{psi} follows immediately. In the setting of Theorem \ref{ril} the proof is analogous.
\end{proof}

Finally, we shall describe how results of this section behave under stretching of a Morse function.

\begin{theorem}\label{qm} Suppose $m:M \to [0,\infty)$ is a Morse function as in Theorem \ref{mfact}.  Let $q:[0,\infty) \to [0,\infty)$ be an increasing $C^\infty$ function such that $q(0)=0$. Define $\hat m: M \to [0,\infty)$ by $\hat m = q\circ m$. Then $\hat m$ also satisfies conclusions of Theorem \ref{mfact}. 

More precisely, suppose that $I=(a,b)$ contains only regular values of $m$. Then, the same is true for $\hat m$ and $\hat I=(q(a),q(b))$. Level submanifolds of $m$ and $\hat m$ are the same after the change of parameter,
\begin{equation}\label{qm1}
\hat J_s = \hat m^{-1}(s)=J_{q^{-1}(s)} \qquad\text{for all }s\in \hat I.
\end{equation}
Moreover, the diffeomorphisms $\hat F_{\theta,s}: \hat J_\theta \to \hat J_s$, $\theta,s\in I$, corresponding to $\hat m$ and $\hat I$ from Theorem \ref{ril} satisfy
\begin{equation}\label{qm2}
\hat F_{\theta,s}(x) = F_{q^{-1}(\theta),q^{-1}(s)}(x) 
\qquad \text{for }x\in \hat J_\theta.
\end{equation}

Likewise, suppose that $I$ contains only one critical value $t_z$ of $m$. Then, $\hat t_z=q(t_z)$ is a unique critical value of $\hat m$ in $\hat I$ and the above conclusions hold with $J_t$ replaced by $\tilde J_t$ from Theorem \ref{cpl}. In particular, the conclusions \eqref{cpl1} and \eqref{cpl2} for the Morse function $m$ with value $\delta_z>0$ transfer to the same conclusions for $\hat m$ with critical value $\hat t_z$ and value 
\begin{equation}\label{qm3}
\hat \delta_z=\min(q(t_z)-q(t_z-\delta_z),q(t_z+\delta_z)-q(t_z))>0.
\end{equation}
\end{theorem}

\begin{proof}
Let $\hat X$ be a renormalized gradient vector field of $\hat m$ given by
\begin{equation}\label{hX}
\hat X(x) = \frac{\grad \hat m(x)}{|\grad \hat m(x)|^2} = \frac{1}{q'(m(x))} X(x).
\end{equation}
For any $\th \in I$ and $x\in M$ consider the initial value problem \eqref{ode}. Likewise, consider the same problem for $\hat X$ given by
\begin{equation}\label{odeh}
\begin{cases}
\hat \eta'(t)=\hat X(\hat \eta(t)) \\
\hat \eta(q(\th))=x.
\end{cases}
\end{equation}
We claim that $\eta(t)=\hat \eta ( q(t))$ for all $t$ in the interval of existence of solution $\eta$. Indeed, let $\tilde \eta(t)=\hat \eta ( q(t))$. By the chain rule and \eqref{hX}, we have
\[
\tilde \eta'(t)= \hat \eta'(q(t)) q'(t) = \hat X(\hat \eta (q(t)) q'(t)= \frac{q'(t)}{q'(m(\hat \eta(q(t))))} X(\tilde \eta(t))=X(\tilde \eta(t)).
\]
In the last step we used the fact that $\hat \eta(s)$ moves through level submanifolds $\hat J_s$. That is, $\hat m(\hat \eta(s))=s$. Hence, the uniqueness of solutions of \eqref{ode} yields $\tilde \eta(t)=\eta(t)$. The identity $\eta = \hat \eta \circ q$ leads to properties \eqref{qm1} and \eqref{qm2} by the proof of Theorem \ref{ril}. We leave filling the remaining details of the proof to the reader.
\end{proof}

\section{Latitudinal projections on a manifold} \label{S4}

In this section we show the existence of a family of projections dissecting a manifold $M$ along level sets of a Morse function $m$. The standing assumption is that a Morse function $m: M \to [0,\infty)$ satisfies the conclusions of Theorem \ref{mfact}. The assumption that each critical value corresponds to only one critical point can be removed. However, it becomes essential in the next section. 

\subsection{One dimensional smooth decomposition of identity}
We start by recalling smooth projections on the real line originally introduced  by Auscher, Weiss, and Wickerhauser \cite{AWW} and Coifman and Meyer \cite{CM}, see also \cite{BD}

\begin{definition}\label{def0}
Let $\delta>0$ and $\th\in \R$. Let $s: \R \to [0,1]$ be a smooth function such that
\begin{equation}\label{supps}
\supp s\subset (-\delta,+\infty)
\end{equation}
and 
\begin{equation}\label{kwadrat}
s^2(t)+s^2(-t)=1 \qquad \text{for all } t\in \R.
\end{equation}
For the construction of such function, see \cite{HW}.
We define Auscher, Weiss, and Wickerhauser (AWW) operator  $E^\pm_\th$ acting on a function $h$ on $\R$ by 
\[
E^\pm_\th (h)(t)= 
s^2(t-\th)h(t)\pm s(t-\th)s(-t+\th) h(2\th-t),\quad t\in \R.
\]
The choice of $\pm$ is referred as the polarity of $E_\th^\pm$.   If polarity is not indicated, we shall assume it is positive, i.e., $E_\th=E_\th^+$.
\end{definition}

By \cite[Proposition 2.1]{BD} we have that
\begin{equation}
E^\pm_\th (C^\infty(\R)) \subset C^\infty(\R)
\end{equation}
and $E^\pm_\th$ is an orthogonal projection as an operator
\begin{equation}
E^\pm_\th :L^2(\R) \to L^2(\R).
\end{equation}

The following lemma shows that AWW operators are uniformly bounded on $L^p(\R)$.

\begin{lemma}\label{lp}
Let $1\leq p \le \infty$. The operator $E^\pm_\th$ extends to a bounded projection
\[
E^\pm_\th:L^p(\R) \to L^p(\R)
\]
with the norm is given by
\[
\|E^\pm_\th\|_{L^p \to L^p}= \sup_{0\le \xi \le 1} ||A^\pm_\xi||_{p \to p},
\qquad\text{where }
A_\xi^\pm := \begin{bmatrix} \xi & \pm \sqrt{\xi(1-\xi)} \\ \pm \sqrt{\xi(1-\xi)} & 1-\xi \end{bmatrix}.
\]
Furthermore, the norm of operator $E^\pm_\th: C_0(\R) \to C_0(\R)$ is given by
\[
\|E^\pm_\th\|_{C_0 \to C_0}=\sup_{0 \le \xi \le 1}\left(\xi+\sqrt{\xi(1-\xi)}\right) = \frac{1+\sqrt{2}}{2}.
\]
\end{lemma}

\begin{proof}
Without loss of generality, we can assume
$\th=0$ and we let $E^{\pm}=E^{\pm}_0$. Let $h\in L^p(\R)$ and $1\le p< \infty$. Since 
\begin{equation}\label{0c1}
E^\pm(h)(t)=
\begin{cases} h(t) & \text{for } t>\delta, 
\\
0 & \text{for }t<-\delta,
\end{cases}
\end{equation}
we need to estimate
\[
L=\int_{-\delta}^\delta |E^\pm(h)(t)|^p dt =\int_0^\delta ( |E^\pm(h)(t)|^p + |E^\pm(h)(-t)|^p )dt.
\]
Observe that
\[
\begin{aligned}
|E^\pm(h)(t)|^p + |E^\pm(h)(-t)|^p & = \bigg\| A^\pm_{s^2(t)}\begin{bmatrix} h(t) \\ h(-t) \end{bmatrix} \bigg\|^p_p \le ||A^\pm_{s^2(t)}||^p_{p\to p} (|h(t)|^p + |h(-t)|^p) 
\\
& 
\le (B_p)^p  (|h(t)|^p + |h(-t)|^p), \qquad\text{where }B_p=\sup_{0\le \xi \le 1} ||A^\pm_\xi||_{p \to p}.
\end{aligned}
\]
Integrating over the interval $[0,\delta]$ yields 
\[
L \le (B_p)^p \int_{-\delta}^\delta |h(t)|^p dt.
\]
Combining this with \eqref{0c1} yields that
\begin{equation}\label{bp}
\|E^\pm\|_{L^p \to L^p} \leq B_p.
\end{equation}
To show equality in \eqref{bp}, we find $1/2\le \xi_0\le 1$ and $(a,b) \in \R^2 $ with $||(a,b)||_p=1$ such that
\[
B_p=||A^{\pm}_{\xi_0}||_{p\to p} = \bigg\| A^{\pm}_{\xi_0} \begin{bmatrix} a \\ b \end{bmatrix} \bigg\|_p.
\]
Let $0 \le t_0<\delta$ be such that $s^2(t_0)=\xi_0$. Take any $\ve>0$. By continuity we find $\eta>0$ such that
\[
\bigg\| A^{\pm}_{s^2(t)} \begin{bmatrix} a \\ b \end{bmatrix} \bigg\|_p \ge (1-\ve) B_p \qquad\text{for }|t-t_0| \le \eta.
\]
Define a function $h$ by
\begin{equation}\label{h}
h=  a\ch_{[t_0,t_0+\eta]} + b \ch_{[-t_0-\eta,-t_0]}.
\end{equation}
Then,
\[
\begin{aligned}
\int_{-\delta}^\delta |E^\pm(h)(t)|^p dt & = \int_0^\delta \bigg\| A^\pm_{s^2(t)}\begin{bmatrix} h(t) \\ h(-t) \end{bmatrix} \bigg\|^p_p dt
= \int_{t_0}^{t_0+\eta} \bigg\| A^\pm_{s^2(t)}\begin{bmatrix} a \\ b \end{bmatrix} \bigg\|^p_p dt
\\
& \ge (1-\ve) (B_p)^p \eta = (1-\ve) (B_p)^p ||h||_p^p.
\end{aligned}
\]
This shows equality in \eqref{bp}.

The case $p=\infty$ follows by the duality $\|E^\pm\|_{L^1 \to L^1}=\|E^\pm\|_{L^\infty \to L^\infty}$ since $E^\pm:L^2(\R) \to L^2(\R)$ is self-adjoint. 
Note that for a given $0 \le \xi \le 1$
\[
 ||A^\pm_\xi||_{\infty \to \infty}=\max ( \xi+\sqrt{\xi(1-\xi)},   1-\xi+\sqrt{\xi(1-\xi)}).
\]
Hence,
\[
B_\infty= \sup_{0\le \xi \le 1} ||A^\pm_\xi||_{\infty \to \infty}=\max_{0\le \xi \le 1}  \xi+\sqrt{\xi(1-\xi)} = \frac{1+\sqrt{2}}{2}.
\]
Finally, to treat the case of $C_0(\R)$, it suffices to take a function $h\in C_0(\R)$
such that $\|h\|_\infty=1$ and $h(t_0)=1$ and $h(-t_0)=\pm 1$, where $t_0$ is as above. This proves that
\[ \|E^\pm\|_{C_0 \to C_0}= \|E^\pm\|_{L^\infty \to L^\infty}.\]
We leave the details to the reader.
\end{proof}

\begin{remark} The proof of Lemma \ref{lp} implies that
\begin{equation}\label{lps}
\begin{aligned}
\int_{\th-\delta}^{\th+\delta} |E^\pm_\th f(t)|^p dt & \le (B_p)^p \int_{\th-\delta}^{\th+\delta} |f(t)|^p dt \qquad\text{for }f\in L^p(\R),
\\
\sup_{t\in [\th-\delta, \th+\delta]}  |E^\pm_\th f(t)|  & \le B_\infty \sup_{t\in [\th-\delta, \th+\delta]}  |f(t)| \qquad\text{for }f\in C_0(\R).
\end{aligned}
\end{equation}
Since 
\begin{equation}\label{lps1}
(\mathbf I- E^+ _\th )f = R_\th E^-_\th R_\th f, \qquad\text{where } R_\th f( t) = f(2\th - t), t\in \R,
\end{equation}
we also have
\begin{equation}
\label{lps2}
\begin{aligned}
\int_{\th-\delta}^{\th+\delta} |(\mathbf I-E^+_\th) f(t)|^p dt  & \le (B_p)^p \int_{\th-\delta}^{\th+\delta} |f(t)|^p dt \qquad\text{for }f\in L^p(\R),
\\
\sup_{t\in [\th-\delta, \th+\delta]}  |(\mathbf I-E^+_\th)  f(t)|  & \le B_\infty \sup_{t\in [\th-\delta, \th+\delta]}  |f(t)| \qquad\text{for }f\in C_0(\R).
\end{aligned}
\end{equation}
\end{remark}

\begin{remark}
Since matrix $2\times 2$ matrix $A_\xi$ has rank 1, it is possible to compute its norm as a mapping $\ell^p \to \ell^p$, where $1<p<\infty $. However, the formula is rather complicated
\[
 ||A^\pm_\xi||_{p \to p}= 
 (\xi^{p/2}+(1-\xi)^{p/2})^{1/p} \bigg(\bigg(\frac{\xi}{1-\xi}\bigg)^{\frac{p}{2-2 p}}+1\bigg)^{-1/p} \bigg(\sqrt{1-\xi} \bigg(\frac{\xi}{1-\xi}\bigg)^{\frac{1}{2-2 p}}+\sqrt{\xi}\bigg).
 \]
Consequently, the norm of the projection $E^\pm_\th$
\[
||E^\pm_\th||_{L^p \to L^p} = B_p= \sup_{0\le \xi \le 1} ||A^\pm_\xi||_{p \to p}
\]
can only be computed numerically for $p\ne 2$.
\end{remark}

\begin{definition}\label{qj}
Suppose a sequence of points $\th_j\in \R$, $j\in \Z$ is such that for all $j\in \Z$, $\th_{j+1}-\th_j>2\delta$, 
\[
\lim_{j\to \infty} \th_j=+\infty, \quad
\lim_{j\to -\infty} \th_j=-\infty.
\]
Define a family of AWW operators by
\[
Q_{\th_j,\th_{j+1}}=E_{\th_j}-E_{\th_{j+1}}, \qquad j\in \Z.
\]
\end{definition}

Note that

\begin{equation}\label{qj0}
Q_{\th_j,\th_{j+1}}f(t)= \begin{cases}
0 & t \le \th_j-\delta \text{ or } t \ge \th_{j+1}+\delta,\\
E_{\th_j}f(t) & \th_j-\delta < t< \th_j+\delta,\\
f(t) & \th_j+\delta \le t \le \th_{j+1} - \delta,\\
(\mathbf I - E_{\th_{j+1}})f(t)  & \th_{j+1}-\delta < t< \th_{j+1}+\delta.
\end{cases}
\end{equation}

\begin{lemma}\label{sdir}
Let $1\le p \le \infty$. The family of operators $\{Q_{\th_j,\th_{j+1}}\}_{j\in\Z}$ is a smooth decomposition of identity in $L^p(\R)$, or $C_0(\R)$ if $p=\infty$, subordinate to an open cover $\{ (\th_j-\delta,\th_{j+1}+\delta)\}_{j\in\Z}$ in the sense of Definition \ref{sdi}. The decomposition constant in \eqref{unc0} is independent of a partition $\{\th_j\}_{j\in \Z}$ and $\delta>0$. 
\end{lemma}

\begin{proof}
By Definitions \ref{def0} and \ref{qj}, $Q_{\th_j,\th_{j+1} } \in \mathcal H(\R)$ is localized on $ (\th_j-\delta,\th_{j+1}+\delta)$, $j\in\Z$. Hence, by \eqref{lps}, \eqref{lps2}, and \eqref{qj0} we have
\begin{equation}\label{lp4}
||Q_{\th_j,\th_{j+1}}f ||_p \le B_p ||f \ch_{[\th_j-\delta,\th_{j+1}+\delta]}||_p \qquad\text{for all } f\in L^p(\R).
\end{equation}
The fact that $Q_{\th_j,\th_{j+1}}: L^2(\R) \to L^2(\R)$ is an orthogonal projection can be found in \cite{HW}. Moreover, projections $Q_{\th_j,\th_{j+1}}$, $j\in\Z$, are mutually orthogonal and they form a decomposition of identity in $L^2(\R)$. We also have pointwise equality for any $f\in L^1_{loc}(\R)$,
\begin{equation}\label{cc}
f(x) = \sum_{j\in \Z} Q_{\th_j,\th_{j+1}}f(x) \qquad\text{for a.e. }x\in \R.
\end{equation}
This implies properties (iii)--(v) in Definition \ref{sdi}. It remains to show the last property (vi). By \eqref{lp4}, for any $f\in L^p(\R)$,
\[
\sum_{j\in \Z} \|Q_{\th_{j},\th_{j+1}} f \|_p^p \leq (B_p)^p
\sum_{j\in \Z} \int_{\th_j-\delta}^{\th_{j+1}+\delta}
|f|^p\leq 2 (B_p)^p \|f\|_p^p.
\]
To show the converse inequality take any $f\in L^p(\R)$. Since for each $x\in \R$ the sum in \eqref{cc} has at most two non-zero terms, we have
\[
|f(x)|^p = \bigg|\sum_{j\in \Z} Q_{\th_j,\th_{j+1}}f(x)\bigg|^p \le 2^{p-1} \sum_{j\in \Z} |Q_{\th_j,\th_{j+1}}f(x)|^p.
\]
Integrating over $\R$ yields
\[
||f||_p^p \le 2^{p-1} \sum_{j\in \Z} \|Q_{\th_{j},\th_{j+1}} f \|_p^p.
\]
This proves (vi). We use
similar arguments for $C_0(\R)$ with the supremum norm. 
\end{proof}

\begin{definition}\label{qjs}
Suppose that $0=\th_0<\th_1<\ldots<\th_n=1$ is such that $\th_{j+1}-\th_j>2\delta$ for all $j=0,\ldots,n$, with the understanding that $\th_{n+1}=1+\th_1$. This corresponds to a partition of the unit circle $\mathbb S^1=\{z\in \C: |z|=1\}$ into arcs 
\[
U(\th_j,\th_{j+1}) = \{ e^{2\pi i t}: \th_j -\delta < t < \th_{j+1}+\delta \}
\]
Define a family of AWW operators acting on functions $f\in C(\mathbb S^1)$ by
\[
Q_{U(\th_j,\th_{j+1})}f(z)= Q_{\th_j,\th_{j+1}}(f \circ e^{2\pi i \cdot})(t), \qquad \text{where } z=e^{2\pi i t}, \ \th_j-\delta<t \le 1+\th_j -\delta.
\]
\end{definition}

Lemma \ref{sdir} has an analogue for the circle in the $L^2$ case, see \cite[Theorem 2.1]{BD}. Consequently, we have the following lemma on $\mathbb S^1$. 

\begin{lemma}\label{sdis}
Let $1\le p \le \infty$. The family of operators $\{Q_{U(\th_j,\th_{j+1})}\}_{j=0}^{n-1}$ is a smooth decomposition of identity in $L^p(\mathbb S^1)$, or $C(\mathbb S^1)$ if $p=\infty$, subordinate to an open cover $\{ U(\th_j-\delta,\th_{j+1}+\delta)\}_{j=0}^{n-1}$ in the sense of Definition \ref{sdi}. Moreover, the decomposition constant in \eqref{unc0} is independent of $\{\th_j\}_{j=0}^n$ and $\delta>0$.
\end{lemma}

The proof of Lemma \ref{sdis} is an easy adaptation from the real line.

\subsection{Smooth decomposition into latitudinal projections}
Recall that $ m : M \to [0, \infty)$ is a Morse function on a manifold $M$ satisfying the conclusions of Theorem 3.1. In what follows we shall tacitly assume that the dimension of $M$ is at least $2$.

\begin{definition}\label{def1}
Let $\th$ be a regular value of $m$ and $\delta>0$ is such that $\bar I_\delta$, where $I_\delta=(\th-\delta,\th+\delta)$, contains no critical values. 
Let $s: \R \to [0,1]$  be a smooth function satisfying \eqref{supps} and \eqref{kwadrat}.
Let $\psi$ be a smooth function on $I_\delta \times J_\th$ as in Theorem \ref{ril}. For a fixed $1\le p \le \infty$, we define Auscher, Weiss, and Wickerhauser (AWW) operator  $E_\th:C_0(M) \to C_0(M)$ with cut-off $\th$ as follows. 

Define an operator $E_{\psi,\th}$ acting on a function $h\in C_0(I_\delta \times J_\th)$ by
\begin{equation}\label{AWW0}
E_{\psi,\th} (h)(t,x)=
s^2(t-\th)h(t,x)+ s(t-\th)s(-t+\th) \bigg(\frac{\psi(2\th-t,x)}{\psi(t,x)} \bigg)^{1/p} h(2\th-t,x),
\end{equation}
where $(t,x)\in I_\delta \times J_\th$.
By the support condition \eqref{supps}, we extend the domain to $h\in C_0(\R \times J_\th)$ by setting 
\begin{equation}\label{AWW1}
E_{\psi,\th} h(t,x)=
\begin{cases} h(t,x) & t\ge \th+\delta,
\\
0 & t \le\th-\delta.
\end{cases}
\end{equation}
Finally, we define AWW operator $E_\th$ on the whole manifold $M$ by setting for $f\in C_0(M)$,
\begin{equation}\label{AWW}
E_\th(f)(y)
=\begin{cases} f(y) &   m(y) \ge \th+\delta\\
E_{\psi,\th} (f\circ F_\th)(t,x)& y=F_\th(t,x), (t,x) \in I_\delta \times J_\th,\\
0 &  m(y)\le \th-\delta. \\
\end{cases}
\end{equation}
\end{definition}

We emphasize   that operators $E_{\psi,\th}$ in  \eqref{AWW0}, \eqref{AWW1} and $E_\th$ in \eqref{AWW} depend on the fixed  value of $1\leq p\leq \infty$. But this dependence is omitted in our notation.

The following lemma is a generalization of the corresponding result on the sphere \cite[Lemma 3.3]{BD}.

\begin{lemma}\label{AWWL}
Let $m:M \to [0,\infty)$ be a Morse function as in Theorem \ref{mfact}.
Let $\th>0$ and $\delta>0$ be such that the interval $[\th-\delta,\th+\delta]$ contains only regular values of a  Morse function $m$. Fix $1\le p \le \infty$. Let $E_\th$ be an AWW operator as in Definition \ref{def1} with this value of $p$. Then, 
\begin{equation}\label{AWWL1}
E_\th:L^p(M)\to L^p(M)
\end{equation}
is a projection, which is orthogonal if $p=2$. Its norm coincides with the norm of  operator $E^{\pm}_\th$ on $L^p(\R)$ in Lemma \ref{lp}.
In the case $p=\infty$, the same holds for $E_\th: C_0(M) \to C_0(M)$.
\end{lemma}

\begin{proof} By Definition \ref{def1}, observe that $E_\th$ is a sum of two operators. By \eqref{AWW0} and \eqref{AWW}, one of these operators is a multiplication operator by the smooth function $\vp$ given by
\[
\vp(y)=\begin{cases}  1 &   m(y) \ge \th+\delta,\\
s^2(t- \th) & t = m(y) \in I_\delta,\\
0 &  m(y)\le \th-\delta.
\end{cases}
\]
By \eqref{supps}, \eqref{AWW0}, and \eqref{AWW}, the other is a simple $H$-operator localized on an open set $m^{-1}(\th-\delta,\th+\delta)$ 

To prove that $E_\th$ is a projection on $L^p(M)$, observe that $L^p(M)$ decomposes as an $\ell^p$ sum
\[
\begin{aligned}
L^p(M) &= L^p(m^{-1}[0,\th-\delta]) \oplus_p L^p(m^{-1}[\th-\delta, \th+\delta]) \oplus_p L^p(m^{-1}[\th+\delta,\infty)),
\\
\int_M |f|^p d\nu &= \int_{m^{-1}[0,\th-\delta]} |f|^p d\nu + \int_{m^{-1}[\th-\delta, \th+\delta]} |f|^p d\nu + \int_{m^{-1}[\th+\delta,\infty)} |f|^p d\nu  \qquad f\in L^p(M).
\end{aligned}
\]
Since $E_\th$ is the zero operator $\bf 0$ and the identity operator $\bf I$ on the first and the last component, respectively, we can restrict our attention to the middle subspace. We claim that
\begin{equation}\label{AWWL4}
\int_{m^{-1}(\th-\delta, \th+\delta)} |E_\th f|^p d\nu \le (B_p)^p \int_{m^{-1}(\th-\delta, \th+\delta)} |f|^p d\nu,
\end{equation}
where $B_p$ is the same constant as in the proof of Lemma \ref{lp}. By Theorem \ref{ril}, $F_{\th}: I_\delta \times J_\th \to m^{-1}(\th-\delta, \th+\delta)$ induces an isometric isomorphism between $L^p( I_\delta \times J_\th,\psi(\lambda \times \nu_\th))$ and $L^p(m^{-1}(\th-\delta, \th+\delta), \nu)$. Hence, it suffices to show that for $h\in C_0(I_\delta \times J_\th)$,
\[
\int_{I_\delta \times J_\th} |E_{\psi,\th} h(t,x)|^p \psi(t,x) dt d \nu_\th \le (B_p)^p \int_{I_\delta \times J_\th} |h(t,x)|^p \psi(t,x) dt d \nu_\th.
\]
For this, it is enough to show that
\begin{equation}\label{AWWL5}
\int_{I_\delta} |E_{\psi,\th} h(t,x)|^p \psi(t,x) dt \le (B_p)^p \int_{I_\delta} |h(t,x)|^p \psi(t,x) dt \qquad x\in J_\th.
\end{equation}
Let $\bar E_\th$ be the one dimensional AWW operator from Definition \ref{def0}. Then, by \eqref{AWW0} we have the following identity
\begin{equation}\label{AWWL6}
E_{\psi,\th}h(t,x)=[ M_{\psi^{-1/p}(\cdot,x)} \circ \bar E_\th \circ M_{\psi^{1/p}(\cdot,x)} ] (h (\cdot,x))(t) \qquad (t,x) \in I_\delta \times J_\th,
\end{equation}
where $M_\kappa$ denotes the multiplication operator by a function $\kappa$ on $I_\delta$. Combining Proposition \ref{absn}(ii), \eqref{lps}, and \eqref{AWWL6} yields \eqref{AWWL5}. Likewise, Proposition \ref{absn}(iii), Lemma \ref{lp} and  \eqref{AWWL6} implies that 
\[
(E_{\psi,\th})^2 h=E_{\psi, \th} h \qquad\text{for } h\in C_0(I_\delta\times J_\th).
\]
Hence, $E_\th: L^p(M) \to L^p(M)$. To see that $E_\th$ is an orthogonal projection when $p=2$, we apply Proposition \ref{absn}(v).
\end{proof}

\begin{remark} Using \eqref{lps2}, an analogue of \eqref{AWWL6} for $\mathbf I - E_{\psi,\th}$, and repeating the argument in the proof of \eqref{AWWL4} yields a manifold variant of \eqref{lps2}
\begin{equation}\label{AWWL7}
\int_{m^{-1}(\th-\delta, \th+\delta)} |(\mathbf I -E_\th )f|^p d\nu \le (B_p)^p \int_{m^{-1}(\th-\delta, \th+\delta)} |f|^p d\nu.
\end{equation}
\end{remark}

\begin{lemma}\label{qth} Let $m:M \to [0,\infty)$ be a Morse function as in Theorem \ref{mfact}. Suppose $\th_1,\th_2 >0$ are such that:
\begin{itemize}
\item $\th_1+\delta<\th_2-\delta$,
\item intervals $[\th_i-\delta,\th_i+\delta]$, $i=1,2$, contain only regular values of $m$.
\end{itemize}
Let $E_{\th_1}$ and $E_{\th_2}$ be AWW operators for some fixed value $1\le p \le \infty$. Then,  $E_{\th_1}$ and $E_{\th_2}$ commute
\begin{equation}\label{qth5}
E_{\th_2}E_{\th_1} = E_{\th_1} E_{\th_2} =E_{\th_2}
\end{equation}
and the operator
\begin{equation}\label{qth1}
Q_{\th_1,\th_2}:=E_{\th_1}-E_{\th_2} \in \mathcal H(M)
\end{equation}
is localized on the open subset $m^{-1}(\th_1-\delta,\th_2+\delta) \subset M$. Moreover, $Q_{\th_1,\th_2}:L^p(M) \to L^p(M)$ is a projection, which satisfies
\begin{equation}\label{qth2}
Q_{\th_1,\th_2}f(x)= \begin{cases}
0 & x \in m^{-1}([0,\th_1-\delta] \cup [\th_2+\delta,\infty)),\\
E_{\th_1}f(x) & x\in m^{-1}(\th_1-\delta,\th_1+\delta),\\
f(x) & x \in m^{-1}([\th_1+\delta,\th_2-\delta]),\\
(\mathbf I -E_{\th_2})f(x) & x\in m^{-1}(\th_2-\delta,\th_2+\delta).
\end{cases}
\end{equation}
The norm of $Q_{\th_1,\th_2}$ acting on $L^p(M)$ is the same as the norm of $E^{\pm}_\th$ on $L^p(\R)$ in Lemma \ref{lp}. In particular, if $p=2$, $Q_{\th_1,\th_2}$ is an orthogonal projection on $L^2(M)$. The analogous statement holds for $Q_{\th_1,\th_2}:C_0(M)\to C_0(M)$.
\end{lemma}

\begin{proof}
Using $\th_1+\delta<\th_2-\delta$ and \eqref{AWW}, one can show \eqref{qth5}. Hence, $Q_{\th_1,\th_2}$ is a projection (orthogonal if $p=2$) by Lemma \ref{AWWL}. The localization of $Q_{\th_1,\th_2}$ is the consequence of the first part of the proof of Lemma \ref{AWWL}. That is, $Q_{\th_1,\th_2}$ is a sum of three simple $H$-operators. Two of them are localized on sets $m^{-1}(\th_i-\delta,\th_i+\delta)$, $i=1,2$. The third simple $H$-operator is a multiplication operator $H_{\vp, id, M}$ by the smooth function 
\[
\vp(y)=\begin{cases}  0 & m(y) \in [\th_2+\delta,\infty),\\
1-s^2(\th_2-t)=s^2(\th_2-t) & t = m(y) \in (\th_2-\delta,\th_2+\delta),\\
1 & m(y) \in [\th_1+\delta,\th_2-\delta],\\
s^2(t- \th_1) & t = m(y)\in (\th_1-\delta,\th_1+\delta),\\
0 &  m(y) \in [0, \th_1-\delta].
\end{cases}
\]
 Therefore, \eqref{qth2} is an immediate consequence of this observation. Finally, the conclusion on the norm of $Q_{\th_1,\th_2}$ is spelled out by the following remark.
\end{proof}

\begin{remark} Combining \eqref{AWWL4}, \eqref{AWWL7}, and \eqref{qth2} yields for $f\in L^p(M)$,
\begin{equation}\label{lpq}
\int_{m^{-1}(\th_1-\delta,\th_2+\delta)}  |Q_{\th_1,\th_2} f (x)|^p d\nu(x)  \le (B_p)^p \int_{m^{-1}(\th_1-\delta,\th_2+\delta)} |f(x)|^p d\nu(x).
\end{equation}
We also have an analogue for $p=\infty$,
\begin{equation}
\sup_{x\in m^{-1}[\th_1-\delta,\th_2+\delta]}  |Q_{\th_1,\th_2} f(x)|   \le B_\infty \sup_{x\in m^{-1}[\th_1-\delta,\th_2+\delta]}   |f(x)| \qquad\text{for }f\in C_0(M).
\end{equation}
In addition, Lemma \ref{qth} holds if $\th_1=0$ under the convention that $E_0 = \mathbf I$ and $Q_{0,\th} = \mathbf I - E_{\th}$.
\end{remark}

We shall refer to $Q_{\th_1,\th_2}$ as a {\it latitudinal projection}.  As a consequence of Lemma \ref{qth} and the telescoping argument we obtain a generalization of \cite[Lemma 3.4]{BD}.

\begin{corollary}\label{ncq}
Suppose that $M$ is a non-compact Riemannian manifold and $m:M \to [0,\infty)$ is a surjective Morse function. 
Suppose that $0=\th_0<\th_1<\th_2<\ldots$ is a partition of $[0,\infty)$ such that:
\begin{itemize}
\item $\th_i+\delta<\th_{i+1}-\delta$ for all $i\ge 0$,
\item each interval $[\th_i-\delta,\th_i+\delta]$, $i\ge 1$, contains only regular points of $m$.
\end{itemize}
Fix $1\le p\le \infty$. Then, the family of operators $\{Q_{\th_{i},\th_{i+1}}\}_{i=0}^{\infty}$ as in Lemma \ref{qth} is a smooth decomposition of identity in $L^p(M)$, $1\leq p<\infty$,  subordinate to the open cover
\[
\mathcal U = \{ m^{-1}(\th_i-\delta,\theta_{i+1}+\delta): i\in \Z\}.
\]
In particular, 
\begin{equation}\label{ncq1}
\sum_{i=0}^\infty Q_{\th_{i},\th_{i+1}} = \mathbf I,
\end{equation}
where $\mathbf I$ is the identity on $L^p(M)$ and the convergence is unconditional in strong operator topology. The decomposition in \eqref{unc0} satisfies
\begin{equation}\label{ncq2}
2^{1/p-1} ||f||_p \le \bigg(\sum_{j=0}^\infty  \|Q_{\th_{j},\th_{j+1}} f \|_p^p \bigg)^{1/p} \le 2^{1/p} B_p ||f||_p
\qquad\text{for all }f\in L^p(M).
\end{equation}
 Moreover, if $p=2$, $\{Q_{\th_{i},\th_{i+1}}\}_{i=0}^{\infty}$ forms an orthogonal decomposition of the identity operator $\mathbf I$ on $L^2(M)$. In the case of $p=\infty$ the same conclusion holds for $C_0(M)$.
\end{corollary}

\begin{proof}
Lemma \ref{qth} shows properties (ii) and (iii) in Definition \ref{sdi}. Property (iv) 
\[
Q_{\th_{i},\th_{i+1}} \circ Q_{\th_{j},\th_{j+1}}=\mathbf 0 \qquad i \ne j,
\]
is a consequence of \eqref{qth5}. By the telescoping argument and Lemma \ref{skon} we have pointwise equality for any $f\in L^1_{loc}(M)$,
\begin{equation}\label{ccc}
f(x) = \sum_{j=0}^\infty Q_{\th_j,\th_{j+1}}f(x) \qquad\text{for a.e. }x\in M,
\end{equation}
with at most two non-zero terms for each $x$. It remains to show property (vi).

By \eqref{lpq}, for any $f\in L^p(M)$,
\[
\sum_{j=0}^\infty  \|Q_{\th_{j},\th_{j+1}} f \|_p^p \leq
(B_p)^p \sum_{j=0}^\infty \int_{m^{-1}(\th_j-\delta,\th_{j+1}+\delta)}
|f|^p d\nu \leq 2 (B_p)^p \|f\|^p_p.
\]
To show the converse inequality take any $f\in L^p(M)$, we apply \eqref{ccc} to get
\[
|f(x)|^p = \bigg|\sum_{j=0}^\infty Q_{\th_j,\th_{j+1}}f(x)\bigg|^p \le 2^{p-1} \sum_{j=0}^\infty |Q_{\th_j,\th_{j+1}}f(x)|^p.
\]
Integrating over $M$ yields
\[
||f||_p^p \le 2^{p-1} \sum_{j=0}^\infty \|Q_{\th_{j},\th_{j+1}} f \|_p^p.
\]
This proves (vi). We use
similar arguments for $C_0(M)$ with the supremum norm. 
\end{proof}

For compact manifolds we have the following variant of Corollary \ref{ncq}, where we use the convention that $E_0=\mathbf I$ and $E_1=\mathbf 0$, which implies that $Q_{0,\th_1} = \mathbf I-E_{\th_1}$ and $Q_{\th_{n-1},1}=E_{\th_{n-1}}$.

\begin{corollary}\label{cq}
Suppose that $M$ is a compact Riemannian manifold and $m:M \to [0,1]$ is a surjective Morse function. 
Suppose that $0=\th_0<\th_1<\ldots<\th_n=1$ is a partition of the interval $[0,1]$ such that:
\begin{itemize}
\item $\th_i+\delta<\th_{i+1}-\delta$ for all $i=0,\ldots,n-1$,
\item each interval $[\th_i-\delta,\th_i+\delta]$, $i=1,\ldots,n-1$, contains only regular points of $m$.
\end{itemize}
Fix $1\le p\le \infty$. Then, the family of operators $\{Q_{\th_{i},\th_{i+1}}\}_{i=0}^{n-1}$ as in Lemma \ref{qth} is a smooth decomposition of identity in $L^p(M)$, $1\leq p<\infty$,  subordinate to the open cover
\[
\mathcal U = \{ m^{-1}(\th_i-\delta,\theta_{i+1}+\delta): i=0,\ldots,n-1\}.
\]
The decomposition constant in \eqref{unc0} is universal. 
Moreover, if $p=2$, $\{Q_{\th_{i},\th_{i+1}}\}_{i=0}^{n-1}$ forms an orthogonal decomposition of the identity operator $\mathbf I$ on $L^2(M)$. 
In case of $p=\infty$ the same conclusion holds for $C(M)$.
\end{corollary}

\section{Lifting of $H$-operators from a level submanifold}\label{S5}

In this section we introduce the method of lifting an operator acting on a level submanifold to an operator on the whole manifold. To achieve this we shall rely heavily on local parameterizations as in Theorems \ref{ril} and \ref{cpl}. Again the standing assumption is that a Morse function $m: M \to [0,\infty)$ satisfies the conclusions of Theorem \ref{mfact}. Throughout this section we fix $1\le p \leq \infty$.

\begin{definition}\label{lift}
Let $I=(a,b)$ be an interval and let 
$\th\in I$ be a regular value of $m$. Let $\tilde J_\th \subset J_\th$ be an open subset of a level submanifold manifold $J_\th = m^{-1}(\th)$. Let $\psi=\psi_\th:I \times \tilde J_\th \to (0,\infty)$ be a smooth function satisfying \eqref{ril5}.
Suppose that $P \in \mathcal H(J_\th)$ is an $H$-operator that is localized on $\tilde J_\th$. For any $t\in I$, define an operator $P_t$ by
\begin{equation}\label{lift1}
P_{t}(f)(x)=
\begin{cases}
P\big(f(\cdot)\big( \frac{\psi(t,\cdot)}{\psi(\th,\cdot)} \big)^{1/p} \big)(x) \big(\frac{\psi(\th,x)}{\psi(t,x)} \big)^{1/p}
& x\in \tilde J_\th,\\
0 & x\in J_\th \setminus \tilde J_\th,
\end{cases}
\qquad \text{for }f\in C(J_\th).
\end{equation}
Define the corresponding {\it local lifting operator} $\Pi$ by
\begin{equation}\label{lift2}
\Pi(h)(t,x)=
P_{t}(h(t,\cdot))(x),\quad h\in C_0(I\times J_\th), (t,x) \in I \times J_\th.
\end{equation}
\end{definition}

The operators $P_t$ in  \eqref{lift1}  and $\Pi$ in \eqref{lift2} depend on the fixed  value of $1\leq p\leq \infty$, but this dependence is omitted in our notation. In the sequel, when we consider simultaneously  the operators  $\Pi$ from Definition \ref{lift} and $E_{\psi,\th}$ from   Definition \ref{def1}, then they correspond to the same value of $p$.

\begin{remark} Despite that the function $(\frac{\psi(t,\cdot)}{\psi(\th,\cdot)})^{1/p}$ is defined only on $\tilde J_\th$, the formula \eqref{lift1} is well-defined since the operator $P$ is assumed to be localized on $\tilde J_\th$. Indeed, by Definition \ref{localized} and \eqref{H2}, the values of an input outside of $\tilde J_\th$ are completely irrelevant and $P_t \in\mathcal H(J_\th)$. Moreover, we have the following useful formula 
\begin{equation}\label{lift3}
\Pi(h)(t,x)=P\bigg(h(t,\cdot)\bigg(\frac{\psi(t,\cdot)}{\psi(\th,\cdot)}\bigg)^{1/p} \bigg)(x)\bigg(\frac{{\psi(\th,x)}}{\psi(t,x)}\bigg)^{1/p} \qquad\text{for }(t,x) \in I \times \tilde J_\th.
\end{equation}
Naturally, $\Pi(h)(t,x)=0$ for all $(t,x) \in I\times (J_\th \setminus \tilde J_\th)$.
\end{remark}

First we shall establish basic properties of a local lifting operator $\Pi$.

\begin{lemma} \label{PH} Let $\Pi$ be a local lifting operator as in Definition \ref{lift} corresponding to $P \in\mathcal H(J_\th)$, which is localized on $\tilde J_\th$.  The following holds:
\begin{enumerate}[(i)]
\item For $\eta\in C_c^\infty(I\times J_\th)$ we define
\[
\Pi_\eta(f)=\eta\Pi(f),\qquad \text{for } f \in C_0(I\times J_\th).
\]
Then $\Pi_\eta$ belongs to  $\mathcal H(I \times J_\th)$ and is localized on $I \times \tilde J_\th$.
\item
If $P$ induces a projection 
\begin{equation}\label{PH1}
P:L^p(\tilde J_\th,\psi(\th,\cdot)\nu_{\th}) \to L^p(\tilde J_\th,\psi(\th,\cdot)\nu_{\th}), \qquad 1\leq p<\infty,
\end{equation}
then $\Pi$ is also a projection 
\begin{equation}\label{PH2}
\Pi: L^p( I\times \tilde J_\th, \psi \lambda\times\nu_{\th}) \to  L^p( I\times \tilde J_\th, \psi \lambda\times\nu_{\th}).
\end{equation}
Here, $\lambda$ is a Lebesgue measure on $I$ and $\nu_\th$ is the Riemannian measure on $J_\th$. 
Moreover, the norms of $P$ and $\Pi$ are the same. In particular, if $p=2$ and $P$ is an orthogonal projection, then so is $\Pi$. 
\item
If $0<\delta<\min(|\th-a|, |b-\th|)$ and interval $[\th-\delta,\th+\delta]$ contains only regular values of $m$, then
 $\Pi$ commutes with AWW operators $E_{\psi,\th}$ as in Definition \ref{def1}, i.e., $\Pi E_{\psi,\th}=E_{\psi,\th} \Pi$. 
\end{enumerate}
\end{lemma}

Lemma \ref{PH}(ii) also holds for $p=\infty$. That is, if
$
P:C_0(J_\th) \to C(J_\th) 
$
is a projection, 
then $
\Pi: C_0(I\times J_\th)\to C_0(I\times J_\th)
$
is also a projection and the norms of $P$ and $\Pi$ are the same.

\begin{proof} Since $P\in \mathcal H(J_\th)$ and it is localized on $\tilde J_\th$ in the sense of Definition \ref{localized}, it suffices to consider the case where $P=H_{\vp,\Phi,V}$ is as in Definition \ref{H} and $V \subset \tilde J_{\th}$ is an open subset with $\Phi(V)\subset \tilde J_\th$. A simple calculation using \eqref{lift1} and \eqref{lift2} shows that $\Pi_\eta=H_{\tilde \vp,\tilde \Phi,\tilde V}$, where
\[
\begin{aligned}
\tilde \vp(t,x)& =
\begin{cases} \eta(t,x)\vp(x) \left(\frac{\psi(\th,x)}{\psi(\th,\Phi(x))}\right)^{1/p}
\left(\frac{\psi(t,\Phi(x))}{\psi(t,x)}\right)^{1/p} & x\in V,
\\
0 & x\in J_\th \setminus V,
\end{cases}
\\
\tilde \Phi(t,x) &=(t,\Phi(x)) \qquad\text{for }(t,x)\in \tilde V=I \times V.
\end{aligned}
\]
Consequently, $\Pi_\eta \in \mathcal H(I\times J_\th)$ and $\Pi_\eta$ is localized on $I \times \tilde J_\th$. Note the presence of $\eta$ is necessary to guarantee that $\supp \tilde \vp \subset I \times \tilde J_\th$ is compact. This proves (i).

To prove (ii), observe that an operator $P$ in \eqref{PH1} has the same norm as
\begin{equation}\label{PH3}
P_t:L^p(\tilde J_\th,\psi(t,\cdot)\nu_{\th}) \to L^p(\tilde J_\th,\psi(t,\cdot)\nu_{\th}).
\end{equation}
This is a consequence of Proposition \ref{absn}(ii) for a measure $\mu$ and a weight $\kappa_t$ on $\tilde J_\th$ given by
\[
d\mu(x)= \psi(\th,x) d\nu_\th(x), \quad \kappa_t(x) =\frac{\psi(t,x)}{\psi(\th,x)} \qquad\text{for } x\in \tilde J_\th.
\]
Hence, by \eqref{lift2} for any $h\in C_0(I \times J_\th)$ we have
\begin{equation}\label{PH4}
\int_{\tilde J_\th} |\Pi h(t,x)|^p \psi(t,x) d\nu_\th(x) \le ||P||^p \int_{\tilde J_\th} |h(t,x)|^p \psi(t,x) d\nu_\th(x) \qquad t\in I.
\end{equation}
Integrating over $t\in I$ shows that the norm of $\Pi$ in \eqref{PH2} is the same as the norm of $P$ in \eqref{PH1}.

In addition, suppose that $P$ acting as in \eqref{PH1} is a projection. We shall use similar methods as in the proof of Lemma \ref{weight1} to show that $\Pi$ is a projection.
Namely, let $h\in C_0(I\times J_\th)$ and $(t,x) \in I \times \tilde J_\th$. By \eqref{lift3}
\[
\begin{aligned}
\Pi^2(h)(t,x)&=P\left(\Pi(h(t,\cdot))\left(\frac{\psi(t,\cdot)}{\psi(\th,\cdot)}\right)^{1/p}\right)(x)\left(\frac{\psi(\th,x)}{\psi(t,x)}\right)^{1/p}
\\
&=P\left(P\left(h(t,\cdot)\left(\frac{\psi(t,\cdot)}{\psi(\th,\cdot)}\right)^{1/p}\right)\right)(x)\left(\frac{\psi(\th,x)}{\psi(t,x)}\right)^{1/p}
\\
&=P\left(h(t,\cdot)\left(\frac{\psi(t,\cdot)}{\psi(\th,\cdot)}\right)^{1/p}\right)(x)\left(\frac{\psi(\th,x)}{\psi(t,x)}\right)^{1/p}=\Pi(h)(t,x).
\end{aligned}
\]
The same holds trivially for $(t,x) \in I \times (J_\th \setminus \tilde J_\th)$. 

Now, if $p=2$ and $P$ is an orthogonal projection, then by Proposition \ref{absn}(v) operators $P_{t}: L^2(\tilde J_\th,\psi(t,\cdot)\nu_{\th}) \to L^2(\tilde J_\th,\psi(t,\cdot)\nu_{\th})$ are orthogonal projections. Consequently, by Fubini's Theorem for any $f,h\in C_0(I \times J_\th)$
\[
\begin{aligned}
\lan \Pi f,h \ran& 
=\int_{I}\int_{\tilde J_\th}P_{t}(f(t,\cdot))(x)h(t,x) \psi(t,x) d\nu_\th(x)dt
\\
&=\int_{I}\int_{\tilde J_\th}f(t,x)P_t(h(t,\cdot))(t,x) \psi(t,x) d\nu_\th(x)dt=\lan f,\Pi h \ran.
\end{aligned}
\]
This shows that 
\[
\Pi: L^2( I\times \tilde J_\th, \psi \lambda\times\nu_{\th}) \to  L^2( I\times \tilde J_\th, \psi \lambda\times\nu_{\th}) 
\]
is self-adjoint, hence an orthogonal projection. 

In the case $p=\infty$ observe that
\[
|\Pi h(t,x)|\leq \|P\| \sup_{y\in J_\th} |h(t,y)|.
\]
This implies that if $h\in C_0(I\times J_\th)$, then $\Pi h\in C_0(I\times J_\th).$

It remains to show that that the operators $E_{\psi,\th}$ and $\Pi$ commute. The key part lies in the following calculation for $(t,x) \in (\th-\delta,\th+\delta) \times \tilde J_\th$,
\[
\begin{aligned}
&\Pi(E_{\psi,\th} h)(t,x)=P_{t}\left((E_{\psi,\th} h)(t,\cdot)\right)(x)=\left(\frac{\psi(\th,x)}{\psi(t,x)}\right)^{1/p} P\left((E_{\psi,\th} h)(t,\cdot)\left(\frac{\psi(t,\cdot)}{\psi(\th,\cdot)}\right)^{1/p}\right)(x)
\\
&=\left(\frac{\psi(\th,x)}{\psi(t,x)}\right)^{1/p} P\left(s^2(t-\th)h(t,\cdot) \left(\frac{\psi(t,\cdot)}{\psi(\th,\cdot)}\right)^{1/p} \right.
\\
&\hskip 2in \left. +s(t-\th)s(-t+\th)\left(\frac{\psi(2\th-t,\cdot)}{\psi(\th,\cdot)}\right)^{1/p} h(2\th -t,\cdot) \right)(x)
\\
&=s^2(t-\th)P_{t}\left(h(t,\cdot) \right)(x) +s(t-\th)s(-t+\th)P_{2\th-t}\left( h(2\th-t,\cdot) \right)(x)\left(\frac{\psi(2\th-t,x)}{\psi(t,x)}\right)^{1/p}
\\
&=s^2(t-\th)\Pi h(t,x) +s(t-\th)s(-t+\th)\Pi h(2\th-t,x) \left(\frac{\psi(2\th-t,x)}{\psi(t,x)}\right)^{1/p}=E_{\psi,\th} \Pi h(t,x).
\end{aligned}
\]
Since operator $P$ is localized on $\tilde J_\th$, we automatically have
\[
\Pi E_{\psi,\th} h(t,x) = E_{\psi,\th} \Pi h(t,x)=0
\qquad\text{for }
(t,x) \in (\th-\delta,\th+\delta) \times (J_\th \setminus \tilde J_\th).
\]
Finally, the case when $t\ge \th+\delta$ or $t\le \th-\delta$ follows from \eqref{AWW1} and is left to the reader.
\end{proof}

We are now ready to give a global definition of a lifting operator by specializing Definition \ref{lift}  to that of Theorem \ref{ril} or \ref{cpl}.

\begin{definition}\label{glift} Let $m:M \to [0,\infty)$ be a Morse function as in Theorem \ref{mfact}. Let $I=(a,b) \subset [0,\infty)$ be an interval such that $\bar I =[a,b]$ contains at most one critical value of $m$; if it exists, then we assume this critical value corresponds to a single critical point. For a regular value $\th\in I$, let $J_\th$ be the corresponding level submanifold of $M$ and  let $\tilde J_\th$ be as in Theorem \ref{cpl} or simply $\tilde J_\th=J_\th$ if $m$ has no critical values in $\bar I$. Let $\psi=\psi_\th$ be a function on $I \times \tilde J_\th$ and $F_\th$ be a diffeomorphism as in Theorem \ref{ril} or \ref{cpl}, if $m$ has either zero or one critical point, resp. 

Suppose that $P \in \mathcal H(J_\th)$ is an $H$-operator that is localized on $\tilde J_\th$. Let $\Pi$ be the local lifting operator as in Definition \ref{lift}. Define the corresponding {\it global lifting operator} $\Pi^M$ acting on a function $f: M \to \R$ by
\begin{equation}\label{glift1}
\Pi^M(f)(y)=
\begin{cases}
\Pi(f\circ F_\th)(t,x) & y=F_\th(t,x), \ (t,x)\in I\times \tilde J_\th,\\
0 & y\in M \setminus M_I, \text{where }M_I=\bigcup_{t\in I} \tilde J_t.
\end{cases}
\end{equation}
\end{definition}

Note that in general $\Pi^M$ is not an $H$-operator due to sharp cut-off at level submanifolds $J_t$, $t=a,b$. However, if $P$ is an orthogonal projection, then so is $\Pi^M$ in light of Lemma \ref{com}.
Moreover, $\Pi^M$ becomes an $H$-operator after we compose it with appropriate latitudinal projections from Lemma \ref{qth}. To prove this we need to calculate the operator $\Pi^M$ from another parametrization $F_s$. 

For a fixed regular value $s\in I$, we define $\tilde P \in \mathcal H(J_s)$ by
\begin{equation}\label{pt}
\tilde P(f)(x)=
\begin{cases}
P_{s}(f\circ F_{\th,s})(F_{s,\th}(x)) & x\in \tilde J_s,
\\
0 & x\in J_s \setminus \tilde J_s,
\end{cases}
\end{equation}
where $P_s \in \mathcal H(J_\th)$ is given by \eqref{lift1}. It turns out that the global lifting operators corresponding to $P$ and $\tilde P$ coincide.

\begin{lemma}\label{coin} Let $\th, s\in I$ be two regular values and let $P\in \mathcal H(J_\th)$ be localized on $\tilde J_\th$. Then, the operator $\tilde P$ given by \eqref{pt} belongs to $\mathcal H(J_s)$ and $\tilde P$ is localized on $\tilde J_s$. Moreover, the global lifting operators $\Pi^M$ and $\tilde \Pi^M$ corresponding to $P$ and $\tilde P$ are the same.
\end{lemma}

\begin{proof} 
The property that $\tilde P\in \mathcal H(J_\th)$ is localized on $\tilde J_s$ is an immediate consequence of Lemma \ref{weight2} and the fact that $F_{\th,s}=(F_{s,\th})^{-1}: \tilde J_\th \to \tilde J_s$ is a diffeomorphism. 
By \eqref{lift3} the local lifting operator $\tilde \Pi$ of $\tilde P$ satisfies
\begin{equation}\label{reprez}
\tilde \Pi(h)(t,x)=\tilde P\left(h(t,\cdot)\left(\frac{\psi_s(t,\cdot)}{\psi_s(s,\cdot)}\right)^{1/p}\right)(x)\left(\frac{\psi_s(s,x)}{\psi_s(t,x)}\right)^{1/p} \qquad\text{for }
(t,x)\in I\times \tilde J_s,
\end{equation}
where $h \in C_0(I\times J_s)$. Also recall that $ \tilde \Pi(h)(t,x)=0$ for $(t,x)\in I\times J_s \setminus \tilde J_s$.

Our goal is to show that for any $f\in C_0(M)$,
\begin{equation}\label{coin2}
\Pi^Mf(F_s(t,x))=\tilde \Pi^M f(F_s(t,x)) \qquad\text{for }(t,x)\in I \times \tilde J_s.
\end{equation}
Since $\Pi^Mf(y)=\tilde \Pi^Mf(y)=0$ for all $y \in M \setminus M_I$, \eqref{coin2} implies that $\Pi^M$ and $\tilde \Pi^M$ coincide. By considering $h=f \circ F_s$, \eqref{glift1} implies that it suffices to show that for any $h\in C_0(I \times \tilde J_s)$ we have
\begin{equation}\label{coin4}
\Pi(h \circ  F_s^{-1}\circ F_\th)(F_\th^{-1}(F_s(t,x)))=\tilde \Pi(h)(t,x) \qquad\text{for }(t,x) \in I \times \tilde J_s.
\end{equation}
By the diagram \eqref{diag}, \eqref{lift1}, and \eqref{lift2} we have
\[
\begin{aligned}
&\Pi(h\circ F_s^{-1}\circ F_\th)(F_\th^{-1}(F_s(t,x)))
=P_{t}(h(t,F_{\th,s}(\cdot))(F_{s,\th}(x))
\\
&=P\left(h(t,F_{\th,s}(\cdot))\left(\frac{\psi_\th(t,\cdot)}{\psi_\th(\th,\cdot)}\right)^{1/p}\right)(F_{s,\th}(x))\left(\frac{\psi_\th(\th,F_{s,\th}(x))}{\psi_\th(t,F_{s,\th}(x))}\right)^{1/p}
\\
&=P_{s} \left(h(t,F_{\th,s}(\cdot))\left(\frac{\psi_\th(t,\cdot)}{\psi_\th(s,\cdot)}\right)^{1/p}\right)(F_{s,\th}(x))\left(\frac{\psi_\th(s,F_{s,\th}(x))}{\psi_\th(t,F_{s,\th}(x))}\right)^{1/p}
\\
&=P_{s}\left(h(t,F_{\th,s}(\cdot))\left(\frac{\psi_\th(t,F_{s,\th}\circ F_{\th,s}(\cdot))}{\psi_\th(s,F_{s,\th}\circ F_{\th,s}(\cdot))}\right)^{1/p}\right)(F_{s,\th}(x))\left(\frac{\psi_\th(s,F_{s,\th}(x))}{\psi_\th(t,F_{s,\th}(x))}\right)^{1/p}.
\end{aligned}
\]
Now applying \eqref{psi}, \eqref{pt}, and \eqref{reprez} yields \eqref{coin4}.
\end{proof}

The following lemma establishes the main properties of the global lifting operator $\Pi_M$. Of particular importance is the commutation of $\Pi_M$ with latitudinal projections $Q_{\th_1,\th_2}$.

\begin{lemma}\label{com} Suppose that $P \in \mathcal H(J_\th)$ is an $H$-operator that is localized on open subset $U \subset \tilde J_\th$. A global lifting operator $\Pi^M$ as in Definition \ref{glift} satisfies the following properties.
\begin{enumerate}[(i)]
\item  $\Pi^M:L^p(M,\nu_M) \to L^p(M,\nu_M)$ is a bounded linear operator with the same norm as $P$ acting as in \eqref{PH1},
\item If $P\in \mathcal H(J_\th)$ acting as in \eqref{PH1} is a projection (orthogonal if $p=2$), then so is $\Pi^M$.
\item
$\Pi^M$ commutes with all AWW projections $E_s$ as in Lemma \ref{AWWL} for any $s\in I$ such that $[s-\delta,s+\delta] \subset I$ contains only regular values of $m$. 
\item
For any two regular values $\th_1<\th_2 \in I$ such that:
\begin{itemize}
\item $\th_1+\delta<\th_2-\delta$,
\item $[\th_i-\delta,\th_i+\delta] \subset I$, $i=1,2$, contains only regular values of $m$,
\end{itemize}
the operator $\Pi^M$ commutes with $Q_{\th_1,\th_2}=E_{\th_1}-E_{\th_2} \in \mathcal H(M)$.
\item
Their composition is an $H$-operator, i.e.,
\begin{equation}\label{com0}
\Pi^M \circ Q_{\th_1,\th_2} = Q_{\th_1,\th_2} \circ \Pi^M \in \mathcal H(M),
\end{equation}
\item
The operator $\Pi^M \circ Q_{\th_1,\th_2}$ is localized on an open set
\begin{equation}\label{prop2}
O(U)=O(U,\th_1,\th_2):=  \bigcup_{s\in (\th_1-\delta,\th_2+\delta)} F_{\th,s}(U).
\end{equation}
\item Finally, $\Pi^M \circ Q_{\th_1,\th_2}: L^p(M) \to L^p(M)$ has norm bounded by $B_p ||P||$, where $B_p$ is the same as in \eqref{bp} and $P$ acts as in \eqref{PH1}. 
\end{enumerate}
\end{lemma}

Lemma \ref{com} also hold for $p=\infty$ with the understanding that
$\Pi^M:C_0(M) \to L^\infty(M)$ in (i) and 
 $\Pi^M \circ Q_{\th_1,\th_2}: C_0(M) \to C_0(M)$ in  (vii).

\begin{proof}
Depending whether $m$ has zero or one critical point in $I$, we apply
Theorem \ref{ril} or \ref{cpl}, resp. The pushforward under the diffeomorphism $F_{\th}^{-1}:M_I \to I \times \tilde J_{\th}$ of the Riemannian measure $\mu_M$ on $M_I$ is the weighted product measure $\psi_\th (\lambda \times \nu_\th)$. The pushforward of the measures induces an isometric isomorphism 
\begin{equation}\label{com2}
T:L^p(M_I,\mu_M) \to L^p(I \times \tilde J_{\th}, \psi_\th (\lambda \times \nu_\th)), \qquad Tf=f\circ F_\th. 
\end{equation}
By \eqref{glift1} the restriction $\Pi^M|_{L^p(M_I)} =T^{-1} \circ \Pi \circ T$, where $\Pi$ is the local lifting operator of $P$.
Thus, $\Pi^M$ acts on on $L^p(M_I)$ as a conjugate of the local lifting operator $\Pi$, and otherwise it is a zero operator on $L^p(M \setminus M_I)$. Lemma \ref{PH}(ii) and $\ell^p$ decomposition
\[
L^p(M)=L^p(M_I) \oplus_p L^p(M \setminus M_I)
\]
shows that $\Pi^M$ is a bounded linear operator with the same norm as $P$. In addition, if $P$ is a projection, then so is $\Pi^M$. This shows (i) and (ii).

Assume momentarily that $[\th-\delta,\th+\delta]$ contains only regular values of $m$. Under this assumption, we shall show that $\Pi^M$ commutes with $E_\th$. Observe that $L^p(I \times \tilde J_\th)$, which is identified with a subspace of functions in $L^p(I \times J_\th)$ vanishing outside $I \times \tilde J_\th$, is an invariant subspace for both $\Pi$ and $E_{\psi,\th}$. Thus, the operators
\[
T^{-1}\circ \Pi \circ T:L^p(M_I) \to L^p(M_I),
\qquad
 T^{-1} \circ E_{\psi,\th} \circ T: L^p(M_I) \to L^p(M_I)
\]
are well-defined and commuting by Lemma \ref{PH}(iii). This shows that  $\Pi^M$ and $E_\th$ commute since $\Pi^ME_\th=E_\th \Pi^M= \mathbf 0$ on $L^p(M\setminus M_I)$. 

Let $\tilde \Pi$ and $\tilde \Pi^M$ be local and global lifting operators as in Lemma \ref{coin}. Repeating the above argument for $\tilde \Pi$ and $E_{\psi_s,s}$ shows that $\tilde \Pi^M$ and $E_s$ also commute. By Lemma \ref{coin} we have $\Pi^M=\tilde \Pi^M$, which shows (iii). (iv) follows automatically from (iii).
Next, we need to show (v).

By the support condition \eqref{supps}, there exists $\ve>0$ such that $\supp s \subset [-\delta+\ve,\infty)$. Hence, by Lemma \ref{qth} the operator $Q_{\th_1,\th_2}$ is localized on $V=m^{-1}(\th_1-\delta+\ve,\th_2+\delta-\ve)$. Let $\eta:M\to [0,1]$ be a smooth function such that $\eta(y)=1$ for $y\in V$ and $\supp \eta \subset m^{-1}(I)$. Define an operator $\bar \Pi^M$ given by $\bar \Pi^M f = \eta \Pi^M f$ for $f\in C_0(M)$. By Lemma \ref{PH} and \eqref{glift1}, we can show that $\bar \Pi^M \in \mathcal H(M)$ due to the presence of the smooth cut-off function $\eta$. In addition, the localization of $Q_{\th_1,\th_2}$ implies that for all $\text{for }f\in C_0(M)$,
\begin{equation}\label{com7}
(\bar \Pi^M \circ Q_{\th_1,\th_2}) f = (\Pi^M \circ Q_{\th_1,\th_2}) f 
= ( Q_{\th_1,\th_2}\circ \Pi^M) f 
= ( Q_{\th_1,\th_2}\circ \bar \Pi^M) f.
\end{equation}
Since the composition of two $H$-operators is again an $H$-operator, we have shown (v).

By \eqref{ril3}, $F_\th$ is a diffeomorphism and $O(U)= F_\th ((\th_1-\delta,\th_2+\delta) \times U)$ is open.
Since the local lifting operator $\Pi$ of $P$ is localized on $I\times U$, one can show that the operator $\bar \Pi^M$ is localized on $O(U)$. Recall that by Lemma \ref{qth} the latitudinal projection operator $Q_{\th_1,\th_2}$ is localized on $m^{-1}(\th_1-\delta,\th_2+\delta)$. Using (v) and \eqref{com7}, Lemma \ref{lco} shows that the composition operator \eqref{com0} is  localized $O(U) \subset m^{-1}(\th_1-\delta,\th_2+\delta)$. This proves (vi).
Finally, (vii) is a consequence of (i) and \eqref{lpq}.
\end{proof}

If there are no critical values between $\th_1$ and $\th_2$, then Lemma \ref{com} yields the following result.
 
\begin{theorem}\label{prop}
Let $m$ be a Morse function as in Theorem \ref{mfact}. Let $\th_1< \th_2$ and $\delta>0$ be such that $m$ takes only regular values in $[\th_1-\delta,\th_2+\delta]$. Let $\th\in [\th_1,\th_2]$. Then the following holds  for any $1\le p<\infty$.

\begin{enumerate}[(i)]
\item
Let $U$ be an open subset of the level submanifold $J_\th$. 
Let $P_U\in \mathcal H(J_\th)$ be an $H$-operator localized on $U$ such that the induced operator 
\begin{equation}\label{prop0}
P_U:L^p(J_\th,\psi(\th,\cdot)\nu_{\th}) \to L^p(J_\th,\psi(\th,\cdot)\nu_{\th})
\end{equation}
is a projection.
Let $\Pi_U^M$ be the corresponding global lifting operator as in Definition \ref{glift}.
Define an operator
\begin{equation}\label{prop1}
P_{O(U)}:=\Pi^M_U\circ Q_{\th_1,\th_2}=Q_{\th_1,\th_2}\circ \Pi^M_U.
\end{equation}
Then, $P_{O(U)} \in \mathcal H(M)$ is localized on \eqref{prop2}.
Moreover, $
P_{O(U)}: L^p(M) \to L^p(M)$ is a projection with the norm bounded by $B_p||P_U||$, where $B_p$ is the same as in \eqref{bp} and $||P_U||$ is the norm of \eqref{prop0}. 

\item
Let $\mathcal U$ be a finite open cover of $J_\th$ and $\{P_{U}\}_{U\in\mathcal U}$ be a smooth decomposition of identity in $L^p(J_\th,\psi(\th,\cdot)\nu_{\th})$ subordinate to $\mathcal U$. Then, $\{P_{O(U)}\}_{U\in\mathcal U}$ is a smooth decomposition of the latitudinal projection $Q_{\th_1,\th_2}: L^p(M) \to L^p(M)$ subordinate to an open cover $ \{O(U)\}_{U\in\mathcal U}$ of $m^{-1}(\th_1-\delta,\th_2+\delta)$.
 That is, $\{P_{O(U)}\}_{U\in\mathcal U}$ satisfies properties (i)--(iv) of Definition \ref{sdi} with (v) and (vi) replaced by
\begin{equation}\label{prop3}
 \sum_{U\in \mathcal U} P_{O(U)}=Q_{\th_1,\th_2},
\end{equation}
and
\begin{equation}\label{uncq}
\frac 1{C} ||Q_{\th_1,\th_2} f||_{p}  \le \bigg( \sum_{U \in\mathcal U} ||P_{O(U)} f||_p^p \bigg)^{1/p} \le C ||Q_{\th_1,\th_2} f||_p \qquad\text{for all }f\in L^p(M),
\end{equation}
where $C>0$ is the decomposition constant of $\{P_{U}\}_{U\in\mathcal U}$ in \eqref{unc0}.
In addition, if $p=2$ and projections $\{P_{U}\}_{U\in\mathcal U}$ acting on $L^2(J_\th,\psi(\th,\cdot)\nu_{\th})$ are orthogonal, then so are projections $\{P_{O(U)}\}_{U\in\mathcal U}$ on $L^2(M)$.
\end{enumerate}
\end{theorem}

\begin{proof}
Let $I=(\th_1-\delta,\th_2+\delta)$. Since $m$ takes no critical values in $\bar I$, we let $\tilde J_\th=J_\th$. Then the conclusion (i) follows by applying Lemma \ref{com}. This implies that $\{P_{O(U)}\}_{U\in\mathcal U}$ satisfies properties (i)--(iii) of Definition \ref{sdi}. To prove that
\begin{equation}\label{prop5}
P_{O(U)} \circ P_{O(U')} =0 \qquad\text{for } U \neq U'
\end{equation}
it suffices to show by Lemma \ref{com}(iii) and \eqref{prop1} that for global lifting operators
\begin{equation}\label{prop6}
\Pi^M_U \circ \Pi^M_{U'} = 0  \qquad\text{for } U \neq U'.
\end{equation}
The above is a consequence of the same property for local lifting operators $\Pi_U \circ \Pi_{U'}=0$ and Definition \ref{glift}. This, in turn, follows from $P_U \circ P_{U'}=0$ by  Proposition \ref{absn}(iii) and \eqref{lift3}.

Since
\[
\sum_{U\in\mathcal U} P_U f= f \qquad\text{for }f\in C(J_\th),
\]
then Definition \ref{lift} implies that
\[
\sum_{U\in\mathcal U} \Pi_U h= h \qquad\text{for }h\in C_0(I\times J_\th).
\]
Consequently, by \eqref{glift1} we have
\[
\sum_{U\in\mathcal U} \Pi^M_U f= f \qquad\text{for }f\in C_0(M_I),
\]
which implies \eqref{prop3}.

Finally, to prove \eqref{uncq}, observe that  Proposition \ref{absn}(iv) and \eqref{lift3} implies an analogue of \eqref{unc0} for local lifting operators
\[
\frac 1{C} ||h ||_{p}  \le \bigg( \sum_{U \in\mathcal U} ||\Pi_U h||_p^p \bigg)^{1/p} \le C ||h||_p \qquad\text{for all }h \in C_0(I \times J_\th).
\]
Using the isometric isomorphism in \eqref{com2} yields
\[
\frac 1{C} ||f||_{p}  \le \bigg( \sum_{U \in\mathcal U} ||\Pi^M_U f||_p^p \bigg)^{1/p} \le C || f||_p \qquad\text{for all }f \in C_0(M_I).
\]
By the density argument and \eqref{prop1}, this proves \eqref{uncq}.
\end{proof}

We also need a variant of Theorem \ref{prop} that deals with critical points. In the case the closed interval $\bar I=[a,b]$ from Definition \ref{glift} contains one critical value of $m$, by Theorem \ref{cpl}, a function $\psi$ is defined initially on $I \times \tilde J_\th$ . Since $\th \in I$ is not a critical value of $m$, by Theorem \ref{ril} we can extend $\psi$ to $(\th-\ve,\th+\ve) \times J_\th$ for some $\ve>0$ such that $[\th-\ve,\th+\ve]$ contains only regular values of $m$, see Remark \ref{cplr}. Hence, it is meaningful to talk about the space $L^p(J_\th,\psi(\th,\cdot)\nu_{\th})$ below.

\begin{theorem}\label{crit} 
Let $m$ be a Morse function as in Theorem \ref{mfact}. Let $t_z \in (0,\sup m)$ be a critical value of $m$, which corresponds to a single critical point $z\in M$, where $\sup m =1$ or $\sup m= \infty$ if $M$ is compact or non-compact, resp. Let $U_z$ be an open neighborhood of $z\in M$. Let  $\delta_z>0$ and $V_z \subset U_z$ be an open neighborhood from Theorem \ref{cpl}. In particular, $t_z$ is the only critical value of $m$ in the interval $[t_z-\delta_z,t_z+\delta_z]$. Let $\th, \th_1, \th_2\in (0,\sup m)$ and $\delta>0$ be such that:
\begin{itemize}
\item $t_z-\delta_z<\th_1<t_z<\th_2<t_z+\delta_z$, 
\item $\th\in [\th_1,\th_2] \setminus\{t_z\}$ and $\delta>0$ satisfies
\begin{equation}\label{crit1}
\delta<\min(|t_z-\th_1|,|t_z-\delta_z -\th_1|, |t_z-\th_2|, |t_z+\delta_z-\th_2|).
\end{equation}
\end{itemize}
Then the following holds for any $1\le p<\infty$.

\begin{enumerate}[(i)]
\item
Let $U$ be an open subset of $\tilde J_\th$. 
Let $P_U\in \mathcal H(J_\th)$ be an $H$-operator localized on $U$ such that the induced operator 
\eqref{prop0} is a projection.
Define an operator $P_{O(U)}$ by \eqref{prop1} and an open set $O(U)$ by \eqref{prop2}.
Then, $P_{O(U)} \in \mathcal H(M)$ is localized on \eqref{prop2}.
Moreover, $
P_{O(U)}: L^p(M) \to L^p(M)$ is a projection with the norm bounded by $B_p||P_U||$, where $B_p$ is the same as in \eqref{bp} and $||P_U||$ is the norm of \eqref{prop0}. 

\item
Let $\mathcal U$ be a finite open cover of $J_\th$ such that $U_0:=J_\th \cap V_z \in \mathcal U$ and for all $U_0 \ne U \in \mathcal U$ we have $U \subset \tilde J_\th$. Let $\{P_{U}\}_{U\in\mathcal U}$ be a smooth decomposition of identity in $L^p(J_\th,\psi(\th,\cdot)\nu_{\th})$ subordinate to $\mathcal U$.  Define an operator
\begin{equation}\label{crit2}
P_{O(U_0)}=Q_{\th_1,\th_2}-\sum_{U_0\neq U \in\mathcal U} P_{O(U)}.
\end{equation}
Then, $P_{O(U_0)} \in \mathcal H(M)$ is localized on an open set
\begin{equation}\label{crit3}
O(U_0)=O(U_0,\th_1,\th_2):=  \bigcup_{s\in (\th_1-\delta,\th_2+\delta)} J_s \setminus F_{\th,s}(J_\th \setminus U_0) \subset U_z.
\end{equation}
The operator 
\begin{equation}\label{crit4}
P_{O(U_0)}: L^p(M) \to L^p(M)
\end{equation}
is a projection with the norm bounded by $B_p(C^2+1)$, where $B_p$ is the same as in \eqref{bp} and  $C>0$ is the decomposition constant of $\{P_{U}\}_{U\in\mathcal U}$ in \eqref{unc0}.

\item
Moreover, $\{P_{O(U)}\}_{U\in\mathcal U}$ is a smooth decomposition of the latitudinal projection $Q_{\th_1,\th_2}: L^p(M) \to L^p(M)$ subordinate to an open cover $ \{O(U)\}_{U\in\mathcal U}$ of $m^{-1}(\th_1-\delta,\th_2+\delta)$.
 That is, $\{P_{O(U)}\}_{U\in\mathcal U}$ satisfies properties (i)--(iv) of Definition \ref{sdi} with (v) and (vi) replaced by \eqref{prop3}
and
\begin{equation}\label{uncqq}
\frac 1{2C} ||Q_{\th_1,\th_2} f||_{p}  \le \bigg( \sum_{U \in\mathcal U} ||P_{O(U)} f||_p^p \bigg)^{1/p} \le B_p (C^{3p}+(C^{2}+1)^p)^{1/p} ||Q_{\th_1,\th_2} f||_p
\end{equation}
for all $f\in L^p(M)$,
where $C>0$ is the decomposition constant of $\{P_{U}\}_{U\in\mathcal U}$ in \eqref{unc0}. In addition, if $p=2$ and projections $\{P_{U}\}_{U\in\mathcal U}$ acting on $L^2(J_\th,\psi(\th,\cdot)\nu_{\th})$ are orthogonal, then so are projections $\{P_{O(U)}\}_{U\in\mathcal U}$ on $L^2(M)$.
\end{enumerate}
\end{theorem}

\begin{proof} 
We shall proceed as in the proof of Theorem \ref{prop} albeit for the interval $I=(t_z-\delta_z,t_z+\delta_z)$. Then the conclusion (i) follows by applying Lemma \ref{com}. This implies that $\{P_{O(U)}\}_{U\in\mathcal U \setminus \{U_0\}}$ satisfies properties (i)--(iii) of Definition \ref{sdi}. 

Next we claim that
\begin{equation}\label{prop10}
P_{O(U)} \circ P_{O(U')} =0 \qquad\text{for } U \neq U' \in\mathcal U \setminus \{U_0\}.
\end{equation}
Repeating the argument in the proof of Theorem \ref{prop}, the property $P_U \circ P_{U'}=0$ implies the same for local lifting operators $\Pi_U \circ \Pi_{U'}=0$, and consequently for global lifting operators $\Pi^M_U \circ \Pi^M_{U'} = 0$. Using \eqref{prop10} and the property
\[
P_{O(U)} \circ Q_{\th_1,\th_2} = Q_{\th_1,\th_2} \circ P_{O(U)} = P_{O(U)} \qquad\text{for } U  \in\mathcal U \setminus \{U_0\},
\]
we find that 
\[
\begin{aligned}
(P_{O(U_0)})^2 & = P_{O(U_0)},
\\
 P_{O(U)} \circ P_{O(U_0)} & =P_{O(U_0)} \circ P_{O(U)}=0 \qquad\text{for } U  \in\mathcal U \setminus \{U_0\}.
 \end{aligned}
\]
This shows that the operator \eqref{crit4} is a projection. To estimate its norm, observe that $\sum_{U \ne U_0} P_U$ is an $H$-operator localized on $\bigcup_{U \ne U_0} U \subset \tilde J_\th$. Applying \eqref{unc0} for $\sum_{U \ne U_0} P_U f$, where $f\in L^p(J_\th, \psi(\th, \cdot)\nu_\th)$, yields
\begin{equation}\label{crit5}
\frac1C \bigg\| \sum_{U \ne U_0} P_U f \bigg\|_p  \le \bigg(\sum_{U \ne U_0} ||P_U f||^p_p \bigg)^{1/p} \le 
C \bigg\| \sum_{U \ne U_0} P_U f \bigg\|_p.
\end{equation}
Thus,
\begin{equation}\label{crit6}
\frac1C \bigg\| \sum_{U \ne U_0} P_U f \bigg\|_p  \le \bigg(\sum_{U \in\mathcal U} ||P_U f||^p_p \bigg)^{1/p} \le 
C ||f||_p.
\end{equation}
Applying conclusion (i) to the operator $\sum_{U \ne U_0} P_U$, we deduce that $\sum_{U \ne U_0} P_{O(U)}$ is a projection on $L^p(M)$ with the norm bounded by
\begin{equation}\label{crit7a}
\bigg\| \sum_{U \ne U_0} P_{O(U)} \bigg\| \le B_p C^2.
\end{equation}
 Hence, we estimate the operator norm of \eqref{crit4} by
\begin{equation}\label{crit7b}
||P_{O(U_0)}|| \le ||Q_{\th_1,\th_2} || + \bigg\| \sum_{U \ne U_0} P_{O(U)}\bigg\| \le B_p(1+C^2).
\end{equation}

By Theorem \ref{cpl}, $F_\th: I \times \tilde J_\th \to M_I$ is a diffeomorphism. Hence, the set 
\[
W=F_{\th}([\th_1-\delta,\th_2+\delta] \times (\tilde J_\th \setminus U_0))
\]
is closed in $M$. Consequently,
\[
O(U_0) = m^{-1}((\th_1-\delta,\th_2+\delta)) \setminus W 
\]
is open.
To complete the proof of conclusion (ii), we need to show that $P_{O(U_0)} \in \mathcal H(M)$ is localized on $O(U_0)$. This is a highly non-trivial statement, since $P_{O(U_0)}$ is a combination of several $H$-operators, which, in general, are {\bf not} localized on $O(U_0)$. Hence, we need to take an advantage of cancelations occurring in formula \eqref{crit2}. 

To achieve this goal, choose an open set $U_1 \subset J_\th$ such that
\begin{equation}\label{crit10}
J_\th \setminus U_0 \subset U_1 \subset \overline{U_1} \subset U_2:=\bigcup_{U_0 \ne U \in\mathcal U} U.
\end{equation}
Take $\vp_\th \in C^\infty(J_\th)$ such that
\begin{equation}\label{crit11}
\supp \vp_\th \subset U_2
\qquad\text{and}\qquad \vp_\th(x)=1 \quad\text{for all }x\in \overline{U_1}.
\end{equation}
Clearly, the corresponding multiplication operator $T_\th \in \mathcal H(J_\th)$, which is given by $T_\th f=\vp_\th f$, is localized on $U_2$. So is the operator 
\begin{equation}\label{crit12}
T:=T_\th - \sum_{U_0\neq U \in\mathcal U} P_{U}=P_{U_0}-(\mathbf I-T_\th).
\end{equation}
Here, $\mathbf I$ is the identity operator on $C_0(J_\th)$. Observe that $\mathbf I - T_\th$ acts as a multiplication operator by $1-\vp_\th$ and is localized on $U_0$. Indeed,
\[
\vp_\th(x) \ne 1 \implies x\in J_\th \setminus \overline{U_1} \subset J_\th \setminus U_1 \subset U_0.
\]
Thus, the operator $T$ defined by \eqref{crit12} is localized both on $U_2$ and $U_0$. By Lemma \ref{il}, $T$ is localized on the intersection \
\begin{equation}\label{crit14}
U_0 \cap U_2=\bigcup_{U_0 \ne U \in\mathcal U} (U_0 \cap U).
\end{equation}

Let $T_0$ be the composition of the global lifting operator of $T$ with $Q_{\th_1,\th_2}$. By Lemma \ref{com}(vi), $T_0$ is localized on 
\[
\begin{aligned}
\bigcup_{s\in (\th_1-\delta,\th_2+\delta)} F_{\th,s}(U_0 \cap U_2) & = \bigcup_{s\in (\th_1-\delta,\th_2+\delta)} \tilde J_s \setminus F_{\th,s}(\tilde J_\th \setminus (U_0\cap U_2)) \\
&
\subset \bigcup_{s\in (\th_1-\delta,\th_2+\delta)} \tilde J_s \setminus F_{\th,s}(J_\th \setminus U_0) \subset O(U_0).
\end{aligned}
\]

Let $\Pi^M$ be a global lifting operator of $T_\th$ and define
\begin{equation}\label{crit13}
T_1=\Pi^M \circ Q_{\th_1,\th_2}=Q_{\th_1,\th_2} \circ \Pi^M.
\end{equation}
Now, \eqref{crit2} implies that 
\[
P_{O(U_0)}=T_0+ Q_{\th_1,\th_2} - T_1.
\]
 Hence, the proof will be complete if we show that $Q_{\th_1,\th_2} - T_1$ is also localized on $O(U_0)$.

Let $\eta \in C^\infty(M)$ be a smooth cut-off function as in the proof of Lemma \ref{com}(v). Let $T_\eta \in \mathcal H(M)$ be the corresponding multiplication operator by $\eta$. Let $\bar \Pi^M \in \mathcal H(M)$ be given by $\bar \Pi^M f = \eta \Pi^M f$ for $f\in C_0(M)$.
Then, by \eqref{com7} and \eqref{crit13} we have
\[
T_1 =\bar \Pi^M \circ Q_{\th_1,\th_2} =Q_{\th_1,\th_2} \circ \bar \Pi^M.
\]
Combing this with $T_\eta \circ Q_{\th_1,\th_2}= Q_{\th_1,\th_2} \circ T_\eta = Q_{\th_1,\th_2}$ yields
\begin{equation}\label{crit15}
Q_{\th_1,\th_2} - T_1 = Q_{\th_1,\th_2} \circ (T_\eta - \bar \Pi^M)=(T_\eta - \bar \Pi^M) \circ Q_{\th_1,\th_2}.
\end{equation}
By Definition \ref{glift}, $\bar \Pi^M$ is a multiplication operator by a function
\begin{equation}\label{crit16}
\vp^M(y) = 
\begin{cases} \eta(y) \vp_\th(x) & y=F_\th(t,x), \ (t,x)\in I\times \tilde J_\th,\\
0 & y\in M \setminus M_I, \text{where }M_I=\bigcup_{t\in I} \tilde J_t.
\end{cases}
\end{equation}
Thus, $T_\eta - \bar \Pi^M\in \mathcal H(M)$ is a multiplication operator by a function $\eta-\vp^M$,
which satisfies
\begin{equation}\label{crit17}
\eta(y)-\vp^M(y) = \eta(y)(1-\vp_\th(x)) \qquad\text{for }y=F_\th(t,x), \ (t,x)\in (\th_1-\delta,\th_2+\delta) \times \tilde J_\th.
\end{equation}
Recall also that $Q_{\th_1,\th_2}$ is localized on $V=m^{-1}(\th_1-\delta+\ve,\th_2+\delta-\ve)$, where $\ve>0$ is sufficiently small. 

Choose $0<\ve'<\ve$ and define a closed set
\begin{equation}\label{crit18}
\supp \eta \subset W=m^{-1}[\th_1-\delta+\ve',\th_2+\delta-\ve'].
\end{equation}
By \eqref{crit11} and \eqref{crit17},
\[
\{y\in W: \vp^M(y)=\eta(y)\} \supset \bigcup_{s \in [\th_1-\delta+\ve',\th_2+\delta-\ve']} F_{\th,s}(U_1).
\]
Hence, by \eqref{crit10}
\begin{equation}\label{crit19}
\begin{aligned}
\{y\in W: \vp^M(y)\ne \eta(y) \} & \subset
\bigcup_{s \in [\th_1-\delta+\ve',\th_2+\delta-\ve']} J_s \setminus F_{\th,s}(U_1)
\\
&\subset
\bigcup_{s \in [\th_1-\delta+\ve',\th_2+\delta-\ve']} J_s \setminus F_{\th,s}(J_\th \setminus U_0) \subset O(U_0).
\end{aligned}
\end{equation}
Since $\bigcup_{s \in [\th_1-\delta+\ve',\th_2+\delta-\ve']} J_s \setminus F_{\th,s}(U_1)$ is closed, \eqref{crit17}, \eqref{crit18}, and \eqref{crit19} yield
\begin{equation}\label{crit20}
\supp(\eta - \vp^M) \subset O(U_0).
\end{equation}
Using \eqref{crit15}, Lemma \ref{lco} implies that the operator $Q_{\th_1,\th_2} - T_1$ is localized on the intersection of $V$ with an open neighborhood of $\supp(\eta-\vp^M)$. 
Hence, \eqref{crit20} yields the required localization of $P_{O(U_0)}$ on $O(U_0)$.
Finally, the inclusion $O(U_0) \subset U_z$ is a consequence of \eqref{cpl2}.

It remains to prove the conclusion (iii). We have already shown that $\{P_{O(U)}\}_{U\in\mathcal U}$ satisfies properties (i)--(iv) of Definition \ref{sdi}. The property \eqref{prop3} is immediate from the definition \eqref{crit2}. In order to show \eqref{uncqq}, note that Proposition \ref{absn}(iv) and \eqref{lift3} implies an analogue of  \eqref{crit5} for local lifting operators
\[
\frac 1{C} \bigg\| \sum_{U \ne U_0 } \Pi_U h  \bigg\|_p \le \bigg( \sum_{U \ne U_0 } ||\Pi_U h||_p^p \bigg)^{1/p} \le C \bigg\| \sum_{U \ne U_0 } \Pi_U h  \bigg\|_p  \qquad\text{for all }h \in C_0(I \times J_\th).
\]
Using isometric isomorphism in \eqref{com2} yields
\[
\frac 1{C} \bigg\| \sum_{U \ne U_0 } \Pi^M_U f \bigg\|_{p}  \le \bigg( \sum_{U \ne U_0 } ||\Pi^M_U f||_p^p \bigg)^{1/p} \le C  \bigg\| \sum_{U \ne U_0 } \Pi^M_U f \bigg\|_{p} \qquad\text{for all }f \in C_0(M_I).
\]
By the density argument and \eqref{prop1} we have
\begin{equation}\label{crit8}
\frac 1{C} \bigg\| \sum_{U \ne U_0 } P_{O(U)} f \bigg\|_{p}  \le \bigg( \sum_{U \ne U_0 } ||P_{O(U)} f||_p^p \bigg)^{1/p} \le C  \bigg\| \sum_{U \ne U_0 } P_{O(U)} f \bigg\|_{p} \qquad\text{for all }f \in L^p(M).
\end{equation}
Hence,
\[
\begin{aligned}
||Q_{\th_1,\th_2} f||_p \le ||P_{O(U_0)}f||_p + \bigg\| \sum_{U \ne U_0 } P_{O(U)} f \bigg\|_{p}  
& \le ||P_{O(U_0)}f||_p+ C \bigg( \sum_{U \ne U_0 } ||P_{O(U)} f||_p^p \bigg)^{1/p}
\\
&\le 2C \bigg( \sum_{U \in \mathcal U } ||P_{O(U)} f||_p^p \bigg)^{1/p}.
\end{aligned}
\]
Applying first \eqref{crit8} and then \eqref{crit7a}, \eqref{crit7b}, and $P_{O(U)} = P_{O(U)}Q_{\th_1,\th_2}$ yields
\[
\sum_{U \in \mathcal U } ||P_{O(U)} f||_p^p \le  ||P_{O(U_0)}f||_p^p  + C^p  \bigg\| \sum_{U \ne U_0 } P_{O(U)} f \bigg\|_{p}^p \le ( B_p)^p ((1+C^{2})^p+ C^{3p})||Q_{\th_1,\th_2}f ||_p^p.
\]
Combining the last two estimates yields \eqref{uncqq}. 

Finally, consider the case when $p=2$ and $\{ P_U \}_{U \in {\mathcal U}}$ is a smooth orthogonal decomposition of identity in $L^2(J_\th, \psi(\th, \cdot) \nu_\th)$. By part (i), projections  $P_{O(U)}: L^2(M) \to L^2(M)$, $U \neq U_0$, have norm $1$, so they are orthogonal. Hence, by \eqref{crit2} the projection $P_{O(U_0)}: L^2(M) \to L^2(M)$ is orthogonal as well. Therefore, $\{P_{O(U)}\}_{U\in\mathcal U}$ is a smooth orthogonal decomposition of the latitudinal projection $Q_{\th_1,\th_2}: L^2(M) \to L^2(M)$.
\end{proof}

\begin{remark} Theorems \ref{prop} and \ref{crit} also hold for $p=\infty$ when $L^p(J_\th)$ is replaced by $C(J_\th)$ and $L^p(M)$ is replaced by $C_0(M)$. We leave the details to the reader.
\end{remark}

\section{Smooth decomposition of identity in $L^p(M)$}\label{S6}

The goal of this section is to prove Theorem \ref{Main}. This is a consequence of the following two more general theorems about existence of smooth decompositions of identity in $L^p(M)$ spaces for $1\le p<\infty$ or $C_0(M)$ in the case $p=\infty$. Theorem \ref{Main1} shows the existence of smooth decompositions for some particular open and precompact cover of $M$. Then, Theorem \ref{Main2} generalizes this result to arbitrary open and precompact covers of $M$.

\begin{theorem}\label{Main1}
Let $M$ be a smooth connected Riemannian manifold (without boundary) of dimension $d$ and suppose $1\le p \le \infty$. Let $\{K_i\}_{i=1}^\infty$ be a nested sequence of compact subsets of $M$ such that
\begin{equation}\label{sk}
K_1=\emptyset,\qquad K_i\subset \interior(K_{i+1})\quad\text{ for all }\in\N, \qquad\text{and} \qquad M= \bigcup_{i=1}^\infty K_i.
\end{equation}

Then, for  any sequence $\{\varepsilon_i\}_{i=1}^\infty$  of positive real numbers, there is at most countable, open and precompact cover $\mathcal U$ of $M$ such that:
\begin{enumerate}[(i)]
\item for all $U\in \mathcal U$ such that $U \cap K_i= \emptyset$, we have $\diam( U)<\varepsilon_i$, 
\item the cover $\mathcal U$ is locally uniformly finite, i.e., there exists a constant $N=N(d)$ depending only on dimension $d$ such that
\begin{equation}\label{luf}
\forall x\in M \ \exists \text{open } V \ni x \qquad
\# |\{ U \in \mathcal U: U \cap V \neq \emptyset \} |\le N,
\end{equation}
\end{enumerate}
Moreover, for any $1\le p \le \infty$, there exists $\{P_U\}_{U\in \mathcal U}$ a smooth decomposition of identity in $L^p(M)$ if $p<\infty$, or $C_0(M)$ if $p=\infty$, which is subordinate to $\mathcal U$. That is,
\begin{enumerate}[(i)]
\setcounter{enumi}{2}
\item each $P_U \in \mathcal H(M)$ is localized on an open set $U \in\mathcal U$,
\item each $P_U:L^p(M) \to L^p(M)$ if  $1\le p<\infty$ or $P_U: C_0(M) \to C_0(M)$ if $p=\infty$, is a projection ,
\item the projections $\{P_U\}_{U\in \mathcal U}$ satisfy
\begin{equation}\label{compo}
P_U \circ P_{U'} =0, \ U\neq U'\in\mathcal U,
\end{equation}
\item in the case $p<\infty$, there exists a constant $C=C(d)$ depending only on dimension $d$ such that
\begin{equation}\label{unc}
\frac 1{C} ||f||_{p} \le \bigg( \sum_{U \in\mathcal U} ||P_U f||_p^p \bigg)^{1/p} \le C ||f||_p \qquad\text{for all }f\in L^p(M),
\end{equation} 
in the case $p=\infty$ we have for all $f\in C_0(M)$,
\[
(||P_U f||_\infty)_{U\in\mathcal U} \in c_0(\mathcal U) \qquad\text{and}\qquad
\frac 1{C} ||f||_{\infty}  \le  \sup_{U \in\mathcal U} ||P_U f||_\infty \le C ||f||_\infty ,
\]
\item the projections $\{P_U\}_{U\in \mathcal U}$ satisfy
\begin{equation}\label{sumo}
\sum_{U \in \mathcal U} P_U = \mathbf I,
\end{equation}
where the convergence in \eqref{sumo} is unconditional in strong operator topology on $L^p(M)$ if $1\le p<\infty$, or on $C_0(M)$ if $p=\infty$. 
\end{enumerate}
\end{theorem}

In the case $p=2$, the decomposition constant in \eqref{unc} equals 1. Hence, each $P_U:L^2(M)\to L^2(M)$ is an orthogonal projection and
$\{P_U\}_{U\in \mathcal U}$ is a smooth orthogonal decomposition in $L^2(M)$.
As a corollary of Theorem \ref{Main1} we deduce a more general result for arbitrary precompact covers of $M$. 

\begin{theorem}\label{Main2}
Let $M$ be a smooth connected Riemannian manifold (without boundary) and let $1\le p \le \infty$.
Suppose $\mathcal W$ is an open and precompact cover of $M$. Then, there exists $\{Q_W\}_{W\in\mathcal W}$ a smooth decomposition of identity in $L^p(M)$ if $p<\infty$, or $C_0(M)$ if $p=\infty$, which is subordinate to $\mathcal W$. In particular, if $p<\infty$ we have
\begin{equation}\label{mc1}
\frac 1{C^2} ||f||_{p} \le \bigg( \sum_{W \in\mathcal W} ||Q_W f||_p^p \bigg)^{1/p} \le C^2 ||f||_p \qquad\text{for all }f\in L^p(M),
\end{equation} 
where $C$ is the the constant in \eqref{unc}. In case $p=\infty$ we have for all $f\in C_0(M)$,
\[
(||Q_W f||_\infty)_{W\in\mathcal W} \in c_0(\mathcal W) \qquad\text{and}\qquad
\frac 1{C^2} ||f||_{\infty}  \le  \sup_{W \in\mathcal W} ||Q_W f||_\infty \le C^2 ||f||_\infty.
\]

Moreover, the localizing sets $K(Q_W)$ of operators $Q_W$ satisfy
\begin{equation}\label{mc2}
\forall x\in M \ \exists \text{open } V \ni x \qquad
\# |\{ W \in \mathcal W: K(Q_W) \cap V \neq \emptyset \} |\le N,
\end{equation}
where $N$ is the same as in \eqref{luf}.
In the case $p=2$, the decomposition constant $C=1$ and $\{Q_W\}_{W\in\mathcal W}$ is a smooth orthogonal decomposition in $L^2(M)$.
\end{theorem}

The scheme of the proofs of Theorems \ref{Main1} and \ref{Main2} is as follows. First, we show that if Theorem \ref{Main1} holds in some dimension $d$, then Theorem \ref{Main2} also holds in the same dimension. Next, we prove Theorem \ref{Main1} by induction with respect to $d$. We start with the base case $d=1$. The inductive argument uses Theorem \ref{Main2} in dimension $d-1$, which is a consequence of the inductive hypothesis that Theorem \ref{Main1} holds in dimension $d-1$.

\begin{proof}[Proof of Theorem \ref{Main2}]
In the proof we shall employ Lebesgue's number lemma. For an open cover of  a compact metric space $(X,\di)$, there exists a number $\eta>0$ such that every subset of $X$ with diameter less than $\eta$ is contained in some member of that cover. In our case, $X$ is a compact subset of $M$ and $\di$ is a geodesic distance on $M$.

Suppose $\mathcal W$ is an open and precompact cover of $M$. Let $\{K_i\}_{i=1}^\infty$ be a sequence of compact sets as in \eqref{sk}. For any $i\ge 1$, we find a sufficiently small $\delta_i>0$ such that the following set is compact
\[
\tilde K_{i+1} := \{x\in M: \di(x, K_{i+1}) \le \delta_i \}.
\]
Let $0<\varepsilon_i<\min(\delta_i,\eta_i)$, where $\eta_i$ is Lebesgue's number for the cover $\mathcal W$ of $\tilde K_{i+1}$. Then, any subset $U \subset M$ with diameter $<\ve_i$ and intersecting $K_{i+1}$ is contained in $\tilde K_{i+1}$. By Lebesgue's number lemma we have $U \subset W$ for some open set $W \in\mathcal W$.

Now let $\mathcal U$ be an open cover and let $\{P_U\}_{U\in \mathcal U}$ be a smooth decomposition of identity as in the conclusions of Theorem \ref{Main1}. In particular, if $U \cap K_i = \emptyset$, then $\diam(U)<\ve_i$. We claim that for any $U\in \mathcal U$, there exists $W \in \mathcal W$ such that $U \subset W$. Indeed, take any $U\in \mathcal U$ and find minimal $i\ge 1$ such that $U \cap K_{i+1} \ne \emptyset$. Since $U \cap K_i =\emptyset$, by Theorem \ref{Main1}(i) we deduce that $\diam(U)<\ve_i$, and hence there exists $W \in \mathcal W$ such that $U \subset W$. For any $U\in\mathcal U$ choose $W=W(U) \in\mathcal W$ such that $U \subset W$. Since $\mathcal U$ is a locally finite cover, only finitely many sets $U\in \mathcal U$ can be assigned to the same (open and precompact) subset $W \in \mathcal W$. 

For any $W \in \mathcal W$, define 
\begin{equation}\label{puq}
Q_W = \sum_{U\in \mathcal U: \ W(U) = W} P_U.
\end{equation}
Since cover $\mathcal U$ is locally uniformly finite and
\[
K(Q_W) \subset \bigcup_{U\in \mathcal U: \ W(U) = W} K(P_U),
\]
hence \eqref{mc2} holds for the same value $N$ as \eqref{luf}. Thus, a sum in \eqref{puq} has only finitely many terms and each $Q_W \in \mathcal H(M)$ is localized on $W$. Observe that
\[
P_U \circ Q_W=Q_W\circ P_U=
\begin{cases}
P_U & \text{if } W=W(U),\\
0& \text{if } W\neq W(U).
\end{cases}
\] Hence applying \eqref{unc} twice we have for all $f\in L^p(M)$,
\[
\bigg( \sum_{W \in\mathcal W} ||Q_W f||_p^p \bigg)^{1/p} \le C \bigg( \sum_{W \in\mathcal W} \sum_{U\in \mathcal U \atop W(U) = W} ||P_U f||_p^p \bigg)^{1/p} = C \bigg( \sum_{U \in\mathcal U} ||P_U f||_p^p \bigg)^{1/p} \le C^2 ||f||_p.
\] 
The converse inequality is shown the same way. A routine verification shows the remaining properties showing that $\{Q_W\}_{W\in \mathcal W}$ is a smooth decomposition of identity in $L^p(M)$. The case $p=\infty$ is an easy modification of the above argument. This completes the proof of Theorem \ref{Main2}.
\end{proof}

We are now ready to prove Theorem \ref{Main1}.

\begin{proof}[Proof of Theorem \ref{Main1}]
Suppose that $m$ is a Morse function on $M$ as in Theorem \ref{mfact}. We claim that it suffices to show Theorem \ref{Main1} for a specific sequence of compact sets $\{K_i\}_{i=1}^\infty$ satisfying \eqref{sk}. Indeed, suppose $\{\tilde K_i\}_{i=1}^\infty$ is another sequence of compact sets of $M$ as in \eqref{sk}. Let $\{\tilde \varepsilon_i\}_{i=1}^\infty$  be any sequence of positive real numbers. Without loss of generality we can assume that $\{\tilde \varepsilon_i\}_{i=1}^\infty$ is nonincreasing. 
By compactness argument for any $i\in \N$, there exists $j=j(i)\in \N$ such that $K_i \subset \tilde  K_j$. Then, choose $j(i)$ such that $j(1)=1$ and $j(i) \le j(i+1)$ for all $i\ge1$. Define a sequence $\{\varepsilon_i\}_{i=1}^\infty$ by $\varepsilon_i=\tilde \varepsilon_{j(i+1)}$. Assuming that the conclusions of Theorem \ref{Main1} hold for  $\{K_i\}_{i=1}^\infty$ and $\{\varepsilon_i\}_{i=1}^\infty$, one can show that the same holds for $\{\tilde K_j\}_{j=1}^\infty$ and $\{\tilde \varepsilon_j\}_{j=1}^\infty$. The only non-trivial is property (i), which follows from the fact that (i) holds for $\{\tilde K_{j(i)}\}_{i=1}^\infty$ and $\{\tilde \varepsilon_{j(i+1)}\}_{i=1}^\infty$

We prove Theorem \ref{Main1} by induction on the dimension $d$ of a manifold $M$. Note that a connected Riemannian manifold $M$ of dimension $d=1$ is a diffeomorphic to a circle if $M$ is compact or a line $\R$ if $M$ is not compact. If $M$ is a circle, the result follows from Lemma \ref{sdis} and Lemmas \ref{weight1} and \ref{weight2}, which enable us to change weights and Riemannian structure. If $M$ is a real line, then we use Lemma \ref{sdir} instead.

Now assume that $M$ is non-compact connected Riemannian manifold with dimension $d\geq 2$. The case of compact $M$ is an easy modification of more complicated non-compact case and is left to the reader. Let $m:M \to [0,\infty)$ be a Morse function as in Theorem \ref{mfact}. In the sequel we shall assume that $m$ has infinitely many critical points; the finite case is again an easy modification. Let $\{z_i\}_{i=1}^\infty$ be the sequence of critical points arranged so that the sequence of critical values $\{t_{z_i}=m(z_i)\}_{i=1}^\infty$ is increasing. The smallest critical value is $t_{z_1}=0$. Define a sequence of compact sets $K_i=m^{-1}([0,t_{z_i}-2018])$, $i\ge 1$. Let $\{\varepsilon_i\}_{i=1}^\infty$ be a decreasing sequence of positive real numbers.

For any $k\ge 2$, consider an open neighborhood $U_{z_k}=B(z_k,\ve_k)$. Let $\delta_{z_k}>0$ be the value and $V_{z_k}$ be the open neighborhood corresponding to $U_{z_k}$ from Theorem \ref{cpl}. Choose an increasing sequence of $\{ \th_i \}_{i=0}^\infty$ and a sequence of positive numbers $\{\delta_i\}_{i=1}^\infty$ such that: 
\begin{itemize}
\item $\th_0=0$ and $\lim_{i\to \infty} \th_i=\infty$.
\item For any $i\ge 1$, intervals $[\th_i-\delta_i,\th_i+\delta_i]$ are disjoint and contain only regular values of $m$. 
\item Interval $[0,\th_1+\delta_1]$ contains only one critical value $\th_0=0$ and $m^{-1}([0,\th_1+\delta_1])$ has diameter less than $\ve_1$.
\item For any $i \ge 1$, each interval $[\th_i-\delta_i,\th_{i+1}+\delta_{i+1}]$ contains at most one critical value of $m$.  If $t_{z_k}$ is such critical value, then we have
\[
t_{z_k}-\delta_{z_k} < \th_i-\delta_i<\th_i+\delta_i<t_{z_k}<\th_{i+1}-\delta_{i+1}<\th_{i+1}+\delta_{i+1}<t_{z_k}+\delta_{z_k}.
\]
and
\begin{equation}\label{control1}
\di(x,F_{\th_i,t}(x))<\ve_k \qquad \text{for all } x\in J_{\th_i}\setminus V_{z_k}, t\in[\th_i-\delta_i,\th_{i+1}+\delta_{i+1}].
\end{equation}
\item If $[\th_i-\delta_i,\th_{i+1}+\delta_{i+1}]$, $i \ge 1$, contains no critical values and $t_{z_k}$ is the largest critical value less than $\th_i$, then
\begin{equation}\label{control2}
\di(x,F_{\th_i,t}(x))<\ve_k,\qquad \text{for all } x\in J_{\th_i}, t\in [\th_i-\delta_i,\th_{i+1}+\delta_{i+1}].
\end{equation}
\end{itemize} 
Such choice of $\{ \th_i \}_{i=1}^\infty$ and $\{\delta_i\}_{i=1}^\infty$ is always possible since the level submanifolds $J_{\th_i}$ are compact. In particular, in the presence of a critical point $z_k$, the set $J_{\th_i} \setminus V_{z_k}$ is also compact. Hence, \eqref{control1} and \eqref{control2} follow by uniform continuity of diffeomorphisms $F_{\th_i,t}$.

In light of Theorem \ref{qm} we can assume that $\delta:=\inf_{i\ge 1} \delta_i>0$. Indeed, if $\inf_{i\ge 1} \delta_i=0$, then we can construct a rapidly increasing $q$ such that a stretched Morse function $\hat m=q\circ m: M \to [0,\infty)$  fulfills the above properties when all $\delta_i$'s are replaced by a single $\delta>0$. This is a consequence of Theorem \ref{qm} which guarantees that level submanifolds $J_t$ and diffeomorphisms $F_{\th,t}$ corresponding to $m$ and $\hat m$ are the same after the change of parameter given by $q$.

Consequently, we can apply Corollary \ref{ncq} which yields a smooth decomposition of the identity by latitudinal projections $\{Q_{\th_i,\th_{i+1}}\}_{i=0}^\infty$. Now it suffices to decompose each projection $Q_{\th_i,\th_{i+1}}$, $i\ge 1$, as a finite sum of projections localized on open sets with small diameters as follows.
The first latitudinal projection $Q_{\th_0,\th_1}$ is already localized on set $W_{z_1}:=m^{-1}([0,\th_1+\delta))$ with diameter $<\ve_1$, and there is no need for further subdivision. Let $P_{W_{z_1}}:=Q_{\th_0,\th_1}$.

First, assume that interval $[\th_i-\delta,\th_{i+1}+\delta]$, $i\ge 1$, contains no critical values. Let $k\in\N$ be such that $t_{z_k}$ is the largest critical value less than $\th=\th_i$. Observe that $J_\th$ is a union of finite number of smooth connected and compact submanifolds of $M$ without boundary and of dimension $d-1$. Using  the induction hypothesis in the form of Theorem \ref{Main1} for each connected component of $J_\th$  separately, we conclude that there is a finite open cover $\mathcal U(J_\th)$ of $J_{\th}$ consisting of sets with diameters $<\ve_k$ and $\{P_U\}_{U\in\mathcal U(J_\th)}$ a smooth decomposition of identity in $L^p(J_\th, \psi(\th, \cdot)\nu_\th)$ subordinate to $\mathcal U(J_\th)$. Then by Theorem \ref{prop}(ii) we obtain $\{P_{O(U)}\}_{U\in \mathcal U(J_\th)}$ a smooth decomposition of the latitudinal projection $Q_{\th_i,\th_{i+1}}$ on $L^p(M)$. In particular, we have
\begin{equation}\label{control3}
\sum_{U\in \mathcal U(J_\th)} P_{O(U)}=Q_{\th_i,\th_{i+1}},
\end{equation}
and
\begin{equation}\label{control4}
\frac 1{C} ||Q_{\th_i,\th_{i+1}} f||_{p}  \le \bigg( \sum_{U \in\mathcal U(J_\th)} ||P_{O(U)} f||_p^p \bigg)^{1/p} \le C ||Q_{\th_i,\th_{i+1}} f||_p \qquad\text{for all }f\in L^p(M),
\end{equation}
where $C=C(d-1)$ is the constant in \eqref{unc} for dimension $d-1$.
Note that by \eqref{prop2} and \eqref{control2} for each $U \in \mathcal U(J_\th)$ we have $\diam O(U)<3\ve_k$. Moreover, by \eqref{prop2} and inductive hypothesis \eqref{luf} we have
\begin{equation}\label{control5}
\forall x\in m^{-1}[\th_i-\delta,\th_{i+1}+\delta] \ \exists \text{open } V \ni x \qquad
\# |\{ U \in \mathcal U(J_\th): O(U) \cap V \neq \emptyset \} |\le N(d-1).
\end{equation}

Next, assume that interval $[\th_i-\delta,\th_{i+1}+\delta]$, $i\ge 1$, contains exactly one critical value $t_{z_k}$. Let $U_0=J_{\th} \cap V_{z_k}$, where $\th=\th_i$. Observe that $\diam(U_0)<\ve_k$. By \eqref{cpl1} we can choose a finite open cover $\mathcal U(J_\th)$ of $J_{\th}$ consisting of sets with diameters $<\ve_k$ such that $U_0 \in \mathcal U(J_\th)$ and $U \subset \tilde J_\th$ for any $U_0 \ne U \in \mathcal U(J_\th)$. By the induction hypothesis, now in the form of Theorem \ref{Main2} in dimension $d-1$, there is a smooth decomposition of identity in $L^p(J_\th, \psi(\th, \cdot)\nu_\th)$ subordinate to $\mathcal U(J_\th)$. Then, by Theorem \ref{crit} we obtain  $\{P_{O(U)}\}_{U\in\mathcal U(J_\th)}$ a smooth decomposition of the latitudinal projection $Q_{\th_i,\th_{i+1}}$ subordinate to an open cover $ \{O(U)\}_{U\in\mathcal U(J_\th)}$. That is, \eqref{control3} and \eqref{control4} hold with the constant 
\begin{equation}\label{cprim}
C'= \max(2C^2, B_p (C^{6p}+(C^{4}+1)^p)^{1/p}),\qquad\text{where }C=C(d-1).
\end{equation}
For each $U_0 \ne U \in \mathcal U(J_\th)$ we have $U \subset J_\th \setminus V_{z_k}$ and hence by \eqref{prop2} and \eqref{control1} we have $\diam O(U)<3\ve_k$. On the other hand,
by \eqref{crit3} we have $\diam O(U_0)<\ve_k$. Moreover, the inductive hypothesis  \eqref{mc2} 
\begin{equation}\label{control6}
\forall x\in J_\th \ \exists \text{open } V \ni x \qquad
\# |\{ U \in \mathcal U(J_\th): K(P_U) \cap V \neq \emptyset \} |\le N(d-1).
\end{equation}
Replacing the sets $U \in \mathcal U(J_\th) \setminus \{U_0\}$ by sufficiently small neighborhoods of $K(P_U)$ we can guarantee that
\begin{equation}\label{control7}
\forall x\in J_\th \ \exists \text{open } V \ni x \qquad
\# |\{ U \in \mathcal U(J_\th) \setminus \{U_0\}: U \cap V \neq \emptyset \} |\le N(d-1).
\end{equation}
Hence, by \eqref{prop2} and \eqref{control7} we have
\begin{equation}\label{control8}
\forall x\in m^{-1}[\th_i-\delta,\th_{i+1}+\delta] \ \exists \text{open } V \ni x \qquad
\# |\{ U \in \mathcal U(J_\th): O(U) \cap V \neq \emptyset \} |\le N(d-1)+1.
\end{equation}

The above procedure defines an open cover $\mathcal U$ of $M$ by
\[
\mathcal U := \{W_{z_1}\} \cup \{ O(U): U \in \mathcal U(J_{\th_i}), \ i \in \N\}.
\]
Then, $\mathcal U$ is a locally uniformly finite cover and (i) holds with $\ve_i$ replaced by $3\ve_i$. Indeed, using  \eqref{control5}, \eqref{control8}, and  the fact that each point $x\in M$ belongs to at most two sets $m^{-1}[\th_i-\delta,\th_{i+1}+\delta]$, $i\in \N$, implies that \eqref{luf} holds with $N(d)=2(N(d-1)+1)$.

The above construction produces a family $\{P_W\}_{W\in\mathcal U}$ forming a smooth decomposition of identity in $L^p(M)$ subordinate to the cover $\mathcal U$.
Using Corollary \ref{ncq}, Theorem \ref{prop}, and Theorem \ref{crit}, we can verify properties (iii)-(vii) in Theorem \ref{Main1}. In particular, by \eqref{ncq2}, \eqref{control4}, and an analogue of \eqref{control4} with constant $C'$ for intervals containing critical values, we can show that \eqref{unc} holds with the constant 
\[
C(d):= \max( 2^{1-1/p}, 2^{1/p} B_p ) C',
\]
where $C'$ is given by \eqref{cprim}. In the case $p=2$, the projections $P_{O(U)}$, $U \in \mathcal U(J_\th)$ are orthogonal by Theorems \ref{prop} and \ref{crit}. Hence, so are projections $P_W$, $W\in\mathcal U$, and the decomposition constant $C(d)=1$ in all dimensions.
Finally, (vii) is a consequence of \eqref{ncq2} and \eqref{control3}.
This proves Theorem \ref{Main1}.
\end{proof}

\section{Smooth decomposition of identity in Sobolev spaces} \label{S7}

In this section we give applications of the main theorem to function spaces on manifolds. We show a decomposition of Sobolev spaces on manifolds extending results of Ciesielski and Figiel \cite{CF1, CF2, CF3} for compact manifolds and the first two authors \cite{BD} for the sphere. In addition, Triebel \cite{Tr1, Tr2} has extended the theory of Triebel-Lizorkin and Besov spaces on complete Riemannian manifolds with bounded geometry, see also \cite{Sk} and \cite[Ch. 7]{Tr}. More recently, Besov spaces on compact manifolds were studied by Geller and  Mayeli \cite{GM}. A survey on a recent progress space-frequency analysis on compact Riemannian manifolds and Riemannian manifolds with bounded geometry can be found in \cite{FHP}. In contrast to these developments, our results we do not require the assumption of bounded geometry.

\subsection{Compact manifolds}

The following result is an extension of a result on the sphere \cite[Theorem 6.1]{BD} to general compact manifolds in the spirit of results of Ciesielski and Figiel, see \cite[\S5]{FW}. Note that their decomposition depends on the choice of the smoothness parameter $r$. In contrast our decomposition using smooth orthogonal projections works for all values of $r\in\N$.

\begin{theorem}\label{cm}
Let $M$ be a smooth compact Riemannian manifold (without boundary).
Let $\mathcal U$ be a finite open cover of $M$. Let $\{P_U\}_{U\in \mathcal U}$ be a smooth decomposition of identity in $L^q(M)$ if $1\le q<\infty$ or on $C(M)$ if $q=\infty$, which is subordinate to $\mathcal U$. Let $\mathcal F(M)$ be either $C^r(M)$ or $W^r_p(M)$, $1\le p<\infty$, $r=0,1,\ldots$.
Then, we have a direct sum decomposition
\[
\mathcal F(M) = \bigoplus_{U \in \mathcal U} P_U( \mathcal F(M)),
\]
with equivalence of norms
\[
||f||_{\mathcal F(M)} \asymp \sum_{U \in\mathcal U} ||P_U f||_{\mathcal F(M)} \qquad\text{for all }f\in \mathcal F(M).
\]
\end{theorem}

The proof of Theorem \ref{cm} employs Theorem \ref{crm} and is shown the same way as in \cite{BD}. This is possible due to the fact that the number of projections $P_U$, $U\in\mathcal U$ is finite and hence they are uniformly bounded on $\mathcal F(M)$. For non-compact manifolds $M$, Theorem \ref{Main2} yields the same result for $L^p(M)$ spaces. However, there is no guarantee that projections $P_U$ are uniformly bounded on Sobolev spaces $W^r_p(M)$ if $M$ is non-compact. To deal with non-compact manifolds $M$ we switch to the setting of local Sobolev spaces. 

\subsection{Non-compact manifolds} 

We will need the following decomposition lemma on $L^1_{loc}(M)$, which is made possible thanks to Lemma \ref{skon}.

\begin{lemma}\label{punktowo}
Let $\{P_U\}_{U\in \mathcal U}$
be a smooth decomposition of identity in $L^q(M)$, $1\le q<\infty$, subordinate to an open and precompact cover $\mathcal U$ of $M$. If $f\in L^1_{loc}(M)$, then 
\begin{equation}\label{punkt1}
f=\sum_{U\in \mathcal U} P_U(f)\qquad \text{ a.e. on M}
\end{equation}
and
\begin{equation}\label{punkt2}
P_U\circ P_V(f) =\begin{cases} 0 & U\neq V
\\
P_U(f) & U=V.
\end{cases}
\end{equation}
\end{lemma}

\begin{proof}
Take any compact set $K \subset M$. Define
\[
\mathcal U(K) = \{ U \in\mathcal U: U \cap K  \ne \emptyset \text{ and } P_U \ne 0 \}.
\]
By Definition \ref{sdi}(i) $\mathcal U(K)$ is finite. Hence
\[
W= \bigcup_{U \in \mathcal U(K)} U
\]
is open and precompact.
By properties (i) and (v) in Definition \ref{sdi} and by \eqref{skon2}, we have
\begin{equation}\label{punkt3}
f \ch_W = \sum_{U \in \mathcal U(\ov W)} P_U (f \ch_W) \qquad\text{for all } f \in L^1_{loc}(M).
\end{equation}
Since each $P_U$ is bounded on $L^1(M)$, a density argument shows that the same holds for $f\in L^1_{loc}(M)$.

Now take any $f\in L^1_{loc}(M)$. Define $g= f \ch_W \in L^1(M)$. By \eqref{skon1}
\begin{equation}\label{punkt4}
P_U(g)=P_U(f) \qquad\text{for all }U \in \mathcal U(K).
\end{equation}
Hence, by \eqref{skon2}, \eqref{punkt3}, and \eqref{punkt4} for a.e. $x\in K$ we have
\[
f(x)=g(x)= \sum_{U \in \mathcal U(\ov W)} P_U g(x)= \sum_{U \in \mathcal U(K)} P_U g(x) = \sum_{U \in \mathcal U(K)} P_Uf(x) = \sum_{U \in\mathcal U} P_Uf(x).
\]
Since $K$ was arbitrary, \eqref{punkt1} is shown. Finally, \eqref{punkt2} for $f\in L^1_{loc}(M)$ is a consequence of property (iv) in Definition \ref{sdi} and \eqref{skon1}. The general case follows by a density argument.
\end{proof}

Local Sobolev spaces on open subsets of $\R^d$ were systematically studied by Antoni\'c and Burazin \cite{AB}. We shall adopt the following definition of local Sobolev spaces on a Riemannian manifold $M$. 

 \begin{definition}\label{locSob} For $r\geq 1$ and $1\leq p<\infty$, we define a local Sobolev space
 \[
 W^r_{p,loc}(M)=\{f\in L^1_{loc}(M): \forall{\eta\in C^\infty_c(M)} \quad
 \eta f\in W^r_p(M) \}.
 \]
 This space $W^r_{p,loc}(M)$ is a locally convex space with the topology given by  a  family of seminorms $\{\vr_\eta : \eta\in C^\infty_c \}$, where 
 \[
 \vr_\eta(f)=\|\eta f\|_{W^r_p(M)}.
 \]
  
\end{definition}

Clearly if $M$ is compact then $W^r_{p,loc}(M)= W^r_{p}(M)$, $r\geq 1$ and $1\leq p<\infty$. We will assume in this subsection that $M$ is non-compact.
In the sequel we will need an equivalent definition of $W^r_{p,loc}(M)$ using a specific family of seminorms indexed by $\N$. For this let  $(\Omega_j)_{j\in \N}$  be a locally finite cover  of $M$. 
Recall that a family $(\alpha_j)_{j\in \N}$ is a partition of unity subordinate to $(\Omega_j)_{j\in \N}$ if for all $j\in \N$, $\alpha_j:M\to [0,1]$ is smooth, $\supp \alpha_j \subset \Omega_j$, and
\[
\sum_{j\in \N} \alpha_j(x)=1\qquad\text{for all }x\in M.
\]

\begin{proposition}\label{locSob1} 
Let $(\Omega_j)_{j\in \N}$  be a locally finite cover  of $M$ such that all sets $\Omega_j$ are precompact. Let $(\alpha_j)_{j\in \N}$ be a partition of unity subordinate to $(\Omega_j)_{j\in \N}$. Let $f\in L^1_{loc}(M)$. Then
 $f\in W^r_{p,loc}(M)$ if and only if 
for all $j\in \N$, $\alpha_j f\in W^r_p(M)$. Moreover,
for every $\eta\in C^\infty_c(M)$, there exist a constant $C(\eta)>0$ and a finite set $N(\eta)\subset \N$  such that
 \[
 \|\eta f\|_{W^r_p(M)}\leq
C(\eta)  \sum_{j\in N(\eta)} \| \alpha_j  f\|_{W^r_p(M)}.
\]
Consequently, two families of seminorms $\{\vr_\eta: \eta\in C^\infty_c(M)\}$ and $\{\vr_{\alpha_j}:j\in \N\}$ define the same topology on $W^r_{p,loc}(M)$.
\end{proposition}

\begin{proof} Let
$\eta\in C^\infty_c(M)$.
 Then,
 \[
 \eta=\sum_{j\in \N} \alpha_j \eta=\sum_{j\in N(\eta)} \alpha_j \eta,
 \]
 where $N(\eta)=\{j\in \N: \supp \eta \cap \supp \alpha_j\neq \emptyset \}
 $. Since $(\Omega_j)_{j\in \N}$ is a locally finite cover of $M$, the set $N(\eta)$ is finite. For any $f\in L^1_{loc}(M)$, we have
  \[
 \eta f=\sum_{j\in N(\eta)} \eta(\alpha_j f).
 \]
Hence, if $\alpha_jf\in W^r_p(M)$, then by Theorem \ref{crm}, $\eta\alpha_jf\in W^r_p(M)$. Consequently, $\eta f\in W^r_p(M)$ as well.
Moreover,
  \[
 \|\eta f\|_{W^r_p(M)}= 
\bigg\|\eta \sum_{j\in N(\eta)} \alpha_j  f \bigg\|_{W^r_p(M)}\leq
C(\eta) \bigg\| \sum_{j\in N(\eta)} \alpha_j  f \bigg\|_{W^r_p(M)}\leq
C(\eta)  \sum_{j\in N(\eta)} \| \alpha_j  f\|_{W^r_p(M)}.
\]
This completes the proof of Proposition \ref{locSob1}.
\end{proof}

\begin{theorem}\label{pur}
Let $(\Omega_j)_{j\in \N}$  be a locally finite cover  of $M$ such that all sets $\Omega_j$ are precompact. Let $(\alpha_j)_{j\in \N}$ be a partition of unity subordinate to $(\Omega_j)_{j\in \N}$.
Let  $\{P_U\}_{U\in \mathcal U}$ be a smooth decomposition of identity in $L^q(M)$, $1\leq q<\infty$, subordinate  to an open and precompact cover  $\mathcal U$ of $M$. 
Let $f\in L^1_{loc}(M)$, $1\le p<\infty$, and $r\in \N$. 
Then,  $f\in W^r_{p,loc}(M)$ if and only if
$P_U f\in W^r_p(M)$ for all $U\in \mathcal U.$ Moreover a family of seminorms
\[
\kappa_U(f)=\|P_U f\|_{W^r_p(M)}, \qquad U\in \mathcal U,
\]
defines the same topology on $ W^r_{p,loc}(M)$ as in Definition \ref{locSob}.
\end{theorem}

\begin{proof}
Let $(\alpha_j)_{j\in \N}$ be a partition of unity subordinate to $(\Omega_j)_{j\in \N}$. To prove the theorem we need two sets of inequalities. 
First, for all $U\in \mathcal U$ there exists a finite family $N(U)\subset \N$ and a constant $C(U)>0$ such that for all $f\in L^1_{loc}(M)$, if $\alpha_j f\in W^r_p(M)$ for all $j\in N(U)$, then $P_Uf \in W^r_p(M)$ and
\begin{equation}\label{seminorm1}
\kappa_U(f)\leq C(U) \sum_{j\in N(U)} \vr_j(f), \qquad\text{where }
\vr_j=\vr_{\alpha_j}.
\end{equation}
Second, for all $j\in \N$ there are a finite family $\mathcal U(j)\subset \mathcal U$ and a constant $C(j)>0$
such that for all $f\in L^1_{loc}(M)$, if $P_U f\in W^r_p(M)$ for all $U\in \mathcal U(j)$, then $\alpha_j f\in W^r_p(M)$ and
\begin{equation}\label{seminorm2}
\vr_j(f)\leq C(j) \sum_{U\in \mathcal U(j)}  \kappa_U(f).
\end{equation}

 Let $U\in \mathcal U$. Fix $V\subset M$ an open and precompact set such that $\overline{U}\subset V$. Since $V$ is precompact then for 
 \[
 N(V)=\{j\in\N: \supp\alpha_j\cap \overline{V}\neq \emptyset\},
 \]
 we have
\[
\sum_{j\in N(V)} \alpha_j=1\qquad \text{on } \overline{V}.
\] 
If $f\in L^1_{loc}(M)$, then
 \[
f=\sum_{j\in N(V)} f \alpha_j \qquad \text{on } \overline{V}.
\]
Define
\[
g:= 
\sum_{j\in N(V)} f \alpha_j
\] then $f=g$ on some neighborhood of $\overline{U}\subset V$. Since the  operator $P_U$ is localized on $U$, by Lemma \ref{skon} we have
\[
P_U(f)=P_U(g).
\]
Hence, by Theorem \ref{crm} there exists a constant $C(U)>0$ such that
\[
\|P_U(f)\|_{W^r_p(M)} =\|P_U(g)\|_{ W^r_p(M)}
\leq C(U) \|g\|_{ W^r_p(M)}\leq C(U) \sum_{j\in N(U)} \|\alpha_j f\|_{W^r_p(M)}.
\]
This proves the first inequality \eqref{seminorm1}.

To show the second inequality, take $j\in \N$. Define
\[
{\mathcal U(j)}=\{U\in {\mathcal U}: U\cap \supp \alpha_j\neq \emptyset \}.
\]
The set ${\mathcal U(j)}$ is finite since ${\mathcal U}$ is a locally finite cover of $M$.
 If $f\in L^1_{loc}(M)$ then from Lemma \ref{punktowo} and the definition of ${\mathcal U}(j)$
\[
\alpha_j f=\alpha_j \sum_{U\in{\mathcal U}} P_U( f)=\alpha_j \sum_{U\in{\mathcal U(j)}} P_U( f) \qquad \text{a.e. on }M.
\]
The multiplication operator by $\alpha_j$ is a simple $H$-operator. Hence, by Theorem \ref{crm} there exists a constant $C(j)>0$ such that
\[
\|\alpha_j f\|_{W^r_p(M)}= \bigg\| \alpha_j  \sum_{U\in{\mathcal U(j)}} P_U( f) \bigg\|_{W^r_p(M)} 
\leq C(j)  \sum_{U\in{\mathcal U(j)}} \|P_U( f)\|_{W^r_p(M)}.
\]
This proves \eqref{seminorm2} and completes the proof of Theorem \ref{pur}.
\end{proof}

The locally convex space $W^r_{p,loc}(M)$ is metrizable. By Theorem \ref{pur} the following metrics induce the same topology:
 \[
 \di_1(f,g)=\sum_{j\in \N} \frac{1}{2^j} \frac{\vr_j(f-g)}{1+\vr_j(f-g)} \qquad\text{where }\vr_j=\vr_{\alpha_j},
 \]
or arranging $\mathcal U$ into a sequence  $(U_k)_{k\in \N}$ and putting 
\[
 \di_2(f,g)=\sum_{j\in \N} \frac{1}{2^j} \frac{\kappa_j(f-g)}{1+\kappa_j(f-g)}\qquad\text{where }\kappa_k=\kappa_{U_k}.
 \]
 
\begin{theorem}\label{fre}
Let $1\leq p<\infty$ and $r\geq 1$. Then,  $W^r_{p,loc}(M)$ is a Fr\'echet space. Moreover,
$C^\infty_c(M)$ is dense in $W^r_{p,loc}(M)$. 
\end{theorem}

\begin{proof}
Note that $W^r_{p,loc}(M)$ is a locally convex metrizable space. To show that $W^r_{p,loc}(M)$ is  Fr\'echet space it remaining to prove that 
$(W^r_{p,loc}(M),\di_1)$ is complete. 
Let $(f_n)$ be a Cauchy sequence  in  $W^r_{p,loc}(M)$. Since $\alpha_j$ has compact support for all $j\in \N$, there is $g_j \in W^r_p(M)$ such that $\alpha_j f_n \to g_j$ in $W^r_p(M)$ as $n\to \infty$. Passing to a subsequence by a diagonal argument  we can assume that for all $j\in \N$
\[
\alpha_j f_n \to g_j \quad \text{ a.e. on $M$ as } n\to \infty.
\]

We claim that there is a function $G\in L^1_{loc}(M)$ such that $g_j=\alpha_j G$ for all $j\in \N$. Indeed, define
for all $j\in \N$
\[
G(x)=\frac{g_j(x)}{\alpha_j(x)} \quad \text{if} \quad \alpha_j(x)\neq 0.
\]
The function $G$ is well defined since for  all $i,j \in \N$ and almost all
$x\in M$ such that $\alpha_j(x)\neq 0$ and $\alpha_i(x)\neq 0$, we have
\[
f_n(x) \to \frac{g_j(x)}{\alpha_j(x)}\qquad \text {as } n\to \infty,
\]
and simultaneously
\[
f_n(x) \to \frac{g_i(x)}{\alpha_i(x)}\qquad \text {as } n\to \infty.
\]
Note that $G\in W^r_{p,loc}(M)$ and for all $j\in \N$  
\[
\vr_j(f_n-G)=\|\alpha_j f_n - g_j\|_{W^r_p(M)}\to 0 \qquad \text {as } n\to \infty.
\]
This implies that $f_n\to G$ in $W^r_{p,loc}(M)$.

Next, we show that $C^\infty_c(M)$ is dense in $W^r_{p,loc}(M)$, i.e.,
 for any $f\in W^r_{p,loc}(M)$ and $\ve>0$ there is a function $G\in C^\infty_c(M)$ such that all $1\leq j\leq N:=-\log_2 \ve$
\[
\vr_j(f- G)=\|\alpha_j(f-G)\|_{W^r_p(M)}\leq \ve.
\]
Take $f\in W^r_{p,loc}(M)$
and a function $\eta\in C^\infty_c(M)$ such that
$\eta(x)=1$ for all $x\in \bigcup_{j=1}^N \supp \alpha_j$. Hence,
\begin{equation}\label{fre4}
 \eta \alpha_j=\alpha_j \qquad\text{for all }1 \le j \le N.
\end{equation}
By Proposition \ref{locSob1}, $\eta f\in W^r_p(M)$. Since $H_j(f)=\alpha_j f$ is a simple $H$-operator, it is bounded on $W^r_p(M)$ by Theorem \ref{crm}. Let
\begin{equation}\label{fre5}
C=\max\{\|H_j\|_{W^r_p(M)\to W^r_p(M)}: 1\leq j\leq N\}.
\end{equation}
 From the definition of Sobolev space $W^r_p(M)$ we can find a function $g\in C^\infty$ such that
 \begin{equation}\label{fre6}
 \|\eta f-g\|_{W^r_p(M)}<\ve/C.
 \end{equation}
 Define the function $G:=\eta g$. Now for  $1\leq j\leq N$, by \eqref{fre4}, \eqref{fre5}, and \eqref{fre6} we have
\[
\vr_j(f- G)=\|\alpha_j(f-\eta g)\|_{W^r_p(M)}=\|\alpha_j(\eta f- g)\|_{W^r_p(M)}\leq C \|\eta f- g\|_{W^r_p(M)}<\ve.
\]
This finishes the proof of Theorem \ref{fre}.
\end{proof}



\bibliographystyle{amsplain}

\begin{thebibliography}{99}

\bibitem{AB}
N. Antoni\'c, K. Burazin, {\it On certain properties of spaces of locally Sobolev functions}, Proceedings of the Conference on Applied Mathematics and Scientific Computing, 109--120, Springer, Dordrecht, 2005.

\bibitem{Au}
T. Aubin, {\it Nonlinear analysis on manifolds. Monge-Amp\`ere equations}, Grundlehren der Mathematischen Wissenschaften, 252. Springer-Verlag, New York, 1982.

\bibitem{AWW} P. Auscher, G. Weiss, M. V. Wickerhauser,
{\it Local sine and cosine bases of Coifman and Meyer and the construction of smooth wavelets.} Wavelets, 237--256,
Wavelet Anal. Appl., 2, Academic Press, Boston, MA, 1992.

\bibitem{BD}
M. Bownik, K. Dziedziul,
{\it Smooth orthogonal projections on sphere}, Const. Approx. {\bf 41} (2015), 23--48.

\bibitem{Ch}
I. Chavel, {\it Isoperimetric inequalities. Differential geometric and analytic perspectives},
Cambridge Tracts in Mathematics, 145. Cambridge University Press, Cambridge, 2001.

\bibitem{CF1} Z. Ciesielski, T. Figiel,
{\it Spline approximation and Besov spaces on compact manifolds.} Studia Math. {\bf 75} (1982), no. 1, 13--36.

\bibitem{CF2} Z. Ciesielski, T. Figiel,
{\it Spline bases in classical function spaces on compact $C^\infty$ manifolds. I.}
Studia Math. {\bf 76} (1983), no. 1, 1--58.

\bibitem{CF3} Z. Ciesielski, T. Figiel,
{\it  Spline bases in classical function spaces on compact $C^\infty$ manifolds. II.}
 Studia Math. {\bf 76} (1983), no. 2, 95--136.

\bibitem{CM} R. Coifman, Y. Meyer,
{\it Remarques sur l'analyse de Fourier \`a fen\^ etre.}
C. R. Acad. Sci. Paris S\' er. I Math. {\bf 312} (1991), no. 3, 259--261.





\bibitem{FHP}
H. Feichtinger, H. F\"uhr, I. Pesenson, 
{\it Geometric space-frequency analysis on manifolds.}
J. Fourier Anal. Appl. {\bf 22} (2016), no. 6, 1294--1355. 

\bibitem{FW}
T. Figiel, P. Wojtaszczyk,
{\it Special bases in function spaces.} Handbook of the geometry of Banach spaces, Vol. I, 561--597, North-Holland, Amsterdam, 2001.


\bibitem{GM}
D. Geller, A. Mayeli, 
{\it Besov spaces and frames on compact manifolds.}
Indiana Univ. Math. J. {\bf 58} (2009), no. 5, 2003--2042. 

\bibitem{He}
E. Hebey, {\it Nonlinear analysis on manifolds: Sobolev spaces and inequalities}. Courant Lecture Notes in Mathematics, 5. American Mathematical Society, Providence, RI, 1999.

\bibitem{HW} E. Hern\' andez, G. Weiss,
{\it A first course on wavelets.}
Studies in Advanced Mathematics. CRC Press, Boca Raton, FL, 1996.

\bibitem{Hes}
M. R. Hestenes, 
{\it Extension of the range of a differentiable function.}
Duke Math. J. {\bf 8}, (1941), 183--192. 

\bibitem{Hi}
M.W. Hirsch, {\it Differential Topology.} Springer-Verlag New York Heidelberg Berlin 1976.


\bibitem{LT}
J. Lindenstrauss, L. Tzafriri, {\it Classical Banach spaces. I. Sequence spaces.} Ergebnisse der Mathematik und ihrer Grenzgebiete, Vol. 92. Springer-Verlag, Berlin-New York, 1977.

\bibitem{Mi}
J. Milnor, {\it Morse theory.} Based on lecture notes by M. Spivak and R. Wells. Annals of Mathematics Studies, No. 51 Princeton University Press, Princeton, N.J. 1963.



\bibitem{NT}
F. Nazarov, S. Treil, {\it The hunt for a Bellman function: applications to estimates for singular integral operators and to other classical problems of harmonic analysis}. 
Algebra i Analiz {\bf 8} (1996), no. 5, 32--162.

\bibitem{Sk}
L. Skrzypczak, 
{\it Wavelet frames, Sobolev embeddings and negative spectrum of Schrödinger operators on manifolds with bounded geometry.}
J. Fourier Anal. Appl. {\bf 14} (2008), no. 3, 415--442. 

\bibitem{Tr1}
H. Triebel,
{\it Spaces of Besov-Hardy-Sobolev type on complete Riemannian manifolds.}
Ark. Mat. {\bf 24} (1986), no. 2, 299--337. 

\bibitem{Tr2}
H. Triebel,
{\it Characterizations of function spaces on a complete Riemannian manifold with bounded geometry.}
Math. Nachr. 130 (1987), 321--346. 

\bibitem{Tr}
H. Triebel, 
{\it  Theory of function spaces. II},
Monographs in Mathematics, 84. Birkh\"auser Verlag, Basel, 1992.

\bibitem{Woj}
P. Wojtaszczyk, {\it Banach spaces for analysts.} Cambridge Studies in Advanced Mathematics, 25. Cambridge University Press, Cambridge, 1991.


\end{thebibliography}

\end{document}